\documentclass[a4paper, 11pt]{article}
\usepackage[utf8]{inputenc}
\usepackage[T1]{fontenc}

\usepackage[
	backend = biber,
	hyperref = auto,
	backref = true,
	doi = false,
	url = false
]{biblatex}

\usepackage{amsthm}
\usepackage{amssymb}
\usepackage{mathtools}
\usepackage{stmaryrd}	\SetSymbolFont{stmry}{bold}{U}{stmry}{m}{n}	
\usepackage{bm}
\usepackage{mathrsfs}
\usepackage{euscript}
\usepackage{tikz-cd}
\usepackage{dynkin-diagrams}
\usetikzlibrary{positioning, shapes.geometric, patterns, graphs, decorations.pathmorphing}

\usepackage{paralist}

\usepackage[colorlinks]{hyperref}
\hypersetup{
	linkcolor=blue,
	citecolor=red,
	urlcolor=teal,
	pdftitle = {Intertwining operators in the Takeda-Wood isomorphism},
	pdfauthor = {Fei Chen, Wen-Wei Li}
}

\newcommand{\Z}{\ensuremath{\mathbb{Z}}}

\newcommand{\R}{\ensuremath{\mathbb{R}}}
\newcommand{\CC}{\ensuremath{\mathbb{C}}}

\newcommand{\Tr}{\operatorname{tr}}

\newcommand{\Ind}{\operatorname{Ind}}
\newcommand{\cInd}{\operatorname{c-Ind}}

\newcommand*{\transp}[2][-3mu]{\ensuremath{\mskip1mu\prescript{\smash{\mathrm t\mkern#1}}{}{\mathstrut#2}}}	
\newcommand{\twomatrix}[4]{ \ensuremath{\bigl(\begin{smallmatrix} #1 & #2 \\ #3 & #4 \end{smallmatrix}\bigr)} }
\newcommand{\twobigmatrix}[4]{ \ensuremath{\begin{pmatrix} #1 & #2 \\ #3 & #4 \end{pmatrix}} }

\newcommand{\dd}{\mathop{}\!\mathrm{d}}
\newcommand{\Schw}{\ensuremath{\mathcal{S}}}	
\newcommand{\mes}{\operatorname{mes}}	

\newcommand{\lrangle}[1]{\ensuremath{\langle #1 \rangle}}

\newcommand{\Stab}{\ensuremath{\mathrm{Stab}}}

\newcommand{\identity}{\ensuremath{\mathrm{id}}}

\newcommand{\Hom}{\operatorname{Hom}}

\newcommand{\End}{\operatorname{End}}
\newcommand{\rightiso}{\ensuremath{\stackrel{\sim}{\rightarrow}}}

\newcommand{\Ker}{\operatorname{ker}}

\newcommand{\dotimes}[1]{\ensuremath{\underset{#1}{\otimes}}}

\newcommand{\Gm}{\ensuremath{\mathbb{G}_\mathrm{m}}}
\newcommand{\Ga}{\ensuremath{\mathbb{G}_\mathrm{a}}}

\newcommand{\Supp}{\operatorname{Supp}}

\newcommand{\GL}{\operatorname{GL}}
\newcommand{\SO}{\operatorname{SO}}

\newcommand{\SL}{\operatorname{SL}}
\newcommand{\Sp}{\operatorname{Sp}}
\newcommand{\Mp}{\ensuremath{\widetilde{\mathrm{Sp}}}}
\newcommand{\pr}{\ensuremath{\mathbf{p}}}
\newcommand{\bmu}{\ensuremath{\bm\mu}}

\theoremstyle{plain}
\newtheorem{proposition}{Proposition}
\newtheorem{lemma}[proposition]{Lemma}
\newtheorem{theorem}[proposition]{Theorem}
\newtheorem{corollary}[proposition]{Corollary}
\theoremstyle{definition}
\newtheorem{definition}[proposition]{Definition}
\newtheorem{definition-theorem}[proposition]{Definition--Theorem}
\newtheorem{definition-proposition}[proposition]{Definition--Proposition}
\newtheorem{remark}[proposition]{Remark}

\theoremstyle{plain}
\newtheorem{Thm}{Theorem}

\numberwithin{equation}{section}
\numberwithin{proposition}{subsection}
\numberwithin{Thm}{section}	

\newcommand{\dcate}[1]{\ensuremath{\text{-}\mathsf{#1}}}	

\renewcommand{\emptyset}{\ensuremath{\varnothing}}	

\usepackage{amsmath}
\usepackage[nottoc]{tocbibind}

\usepackage{geometry}
\title{Intertwining operators in the Takeda--Wood isomorphism}
\author{Fei Chen \quad Wen-Wei Li}
\date{}

\DeclareFieldFormat{postnote}{#1}	
\makeatletter
\renewcommand{\l@section}{\@dottedtocline{1}{1.5em}{2.0em}}
\renewcommand{\l@subsection}{\@dottedtocline{2}{4.0em}{3.0em}}
\makeatother

\begin{document}

\maketitle

\begin{abstract}
	Over any non-Archimedean local field of characteristic not equal to $2$, Takeda and Wood constructed types for the two blocks containing the even and odd Weil representations of the metaplectic group $\tilde{G}$, and identified the resulting Hecke algebras $H_\psi^{\pm}$ with the Iwahori--Hecke algebras of odd orthogonal groups $G^{\pm}$ of the same rank. We describe normalized parabolic induction and Jacquet modules in terms of Hecke modules using a suitable variant of Bushnell--Kutzko theory. Furthermore, we match the standard intertwining operators of $\tilde{G}$ and $G^{\pm}$ by proving a variant of Gindikin--Karpelevich formula for $\tilde{G}$. As an application, we describe the behavior of normalized intertwining operators of $\tilde{G}$ in these blocks under Aubert involution, reducing everything to the $G^{\pm}$ side. This is mainly motivated by Arthur's local intertwining relations.
\end{abstract}



\setcounter{tocdepth}{1}
\tableofcontents

\section{Introduction}\label{sec:intro}
This work grows out of a desire to understand Arthur's \emph{local intertwining relations} (LIR) for metaplectic coverings, whose formulation for quasisplit classical groups can be found in \cite[\S 2.4]{Ar13}. The LIR is one of the pillars of Arthur's conjectures in Langlands program. Roughly speaking, it is a statement about the relation between intertwining operators and endoscopy, or the structure of Arthur packets under parabolic induction. The tempered LIR for metaplectic groups is settled in \cite{Is20} by reducing to odd orthogonal groups via $\theta$-lift. As for the non-tempered case, if one tries to adopt Arthur's strategy in \cite[Chapter 7]{Ar13}, some initial knowledge about LIR would be indispensable for bootstrapping the local-global argument. Such knowledge includes, but is not limited to the ``unramified case'', and one needs results alike at \emph{every} non-Archimedean place. This makes the metaplectic LIR for metaplectic groups more delicate than the classical counterparts, since one has to deal with $p=2$.

The aim of this work is to understand intertwining operators for genuine representations in the two simplest Bernstein blocks of metaplectic group over a non-Archimedean local field $F$ of \emph{arbitrary} residual characteristic $p$, namely the blocks containing the even and odd Weil representations, respectively. They are the metaplectic analogs of Iwahori-spherical Bernstein blocks. The hope is to transfer the problems to odd orthogonal groups as much as possible.

To fix notations, let $\mathfrak{o}$ (resp.\ $\varpi$) be the ring of integers (resp.\ a uniformizer) of $F$, and let $|\cdot|_F$ be the normalized absolute value on $F$. Let $(W, \lrangle{\cdot|\cdot})$ be a symplectic $F$-vector space with prescribed symplectic basis $e_1, \ldots, e_n, f_n, \ldots, f_1$, whence the standard Borel pair $(B^{\rightarrow}, T)$. The corresponding symplectic group $\Sp(W)$ is identified with its $F$-points. Fix an additive character $\psi$ of $F$ such that
\[ \psi_{4\mathfrak{o}} = 1, \quad \psi_{4\varpi^{-1}\mathfrak{o}} \not\equiv 1. \]

Let $\bmu_m$ be the group of $m$-th roots of $1$ in $\CC$. The twofold metaplectic covering is a central extension
\[ 1 \to \bmu_2 \to \operatorname{Mp}(W) \xrightarrow{\pr} \Sp(W) \to 1 \]
of locally compact groups. It arises from the Heisenberg group $H(W) = W \times F$ with multiplication $(w, t)(w', t') = (w + w', t + t' + \lrangle{w|w'})$.

Following \cite{Li11} and its sequels, we will push the covering out along $\bmu_2 \hookrightarrow \bmu_8$ to obtain the covering
\[ 1 \to \bmu_8 \to \Mp(W) \xrightarrow{\pr} \Sp(W) \to 1. \]
It carries the Weil representation $\omega_\psi = \omega_\psi^+ \oplus \omega_\psi^-$ (even $\oplus$ odd) which is genuine. Recall that a representation of $\Mp(W)$ is genuine if $\bmu_8$ acts tautologically.

Let $G := \Sp(W)$ and $\tilde{G} := \Mp(W)$. The main benefit of eightfold coverings is that given a Levi subgroup $M = \Sp(W^\flat) \times \prod_{i \in I} \GL(n_i)$ of $G$, where $W^\flat \subset W$ is a symplectic subspace, we have
\[ \tilde{M} := \pr^{-1}(M(F)) \simeq \Mp(W^\flat) \times \prod_{i \in I} \GL(n_i, F); \]
the canonical isomorphism above is given by the the Schrödinger model for $\omega_\psi$.

Moreover, it is a general property of coverings that the preimage of any parabolic $P = MU \subset G$ splits in $\tilde{G}$ into $\tilde{P} := \pr^{-1}(P(F)) = \tilde{M}U(F)$.

We are interested in the genuine smooth representations of  $\tilde{G}$. The category of such representation decomposes into Bernstein blocks. Let $\mathcal{G}_\psi^{\pm}$ be the blocks containing $\omega_\psi^\pm$.

Following the work \cite{GS12} of Gan and Savin (for $p > 2$), Takeda and Wood \cite{TW18} constructed types for $\mathcal{G}_\psi^{\pm}$ in the sense of Bushnell--Kutzko \cite{BK98}. They gave presentations for the corresponding Hecke algebras $H_\psi^{\pm}$ and established an isomorphism
\[ \mathrm{TW}: H^{\pm} \rightiso H_\psi^{\pm} \]
as Hilbert algebras. Here $H^{\pm} = H^{G^{\pm}}$ is the Iwahori--Hecke algebra for $G^{\pm} := \SO(V^{\pm})$, where $V^{\pm}$ is a quadratic $F$-vector space of dimension $2n+1$, discriminant $1$ and Hasse invariant $\pm 1$. Consequently, $\mathcal{G}_\psi^{\pm} \simeq \mathcal{G}^{\pm}$ where $\mathcal{G}^{\pm}$ denotes the Iwahori-spherical block for $G^{\pm}$.

Since we have split Levi subgroups $\tilde{M} \subset \tilde{G}$ into a smaller metaplectic group and $\GL$-factors, the results above extend to $\tilde{M}$ and the corresponding Levi subgroups $M^{\pm} \subset G^{\pm}$ (i.e.\ with the same pattern of $\GL(n_i)$'s): we have $\mathrm{TW}: H^{M^{\pm}} \rightiso H_\psi^{\tilde{M}, \pm} $.

For representations in $\mathcal{G}_\psi^{\pm}$, we wish to translate questions about intertwining operators to the $G^{\pm}$-side through $\mathrm{TW}$. The following issue has to be resolved first.
\begin{center}\begin{minipage}{0.8\textwidth}
	How to describe normalized Jacquet modules $r_{\tilde{P}}$ and parabolic inductions $i_{\tilde{P}}$ in terms of Hecke-modules?
\end{minipage}\end{center}

For the groups $G^{\pm}$ and the blocks $\mathcal{G}^{\pm}$, the corresponding question is solved in \cite{BK98} (see also \cite{BHK11}): there is a canonical embedding $\mathrm{t}^{\pm}_{\mathrm{nor}}: H^{M^{\pm}} \to H^{\pm}$ of Hilbert algebras, through which $r_{P^{\pm}}$ matches $(\mathrm{t}^{\pm}_{\mathrm{nor}})^*: H^{\pm}\dcate{Mod} \to H^{M^{\pm}}\dcate{Mod}$, where $P^{\pm} = M^{\pm} U^{\pm} \subset G^{\pm}$ is parabolic. By adjunction, $i_{P^{\pm}}$ matches $\Hom_{H^{M^{\pm}}}(H^{\pm}, \cdot)$.

On the metaplectic side, the problem is that the type constructed by Takeda--Wood for $\tilde{G}$ is \textsc{not} a \emph{cover} of that of $\tilde{M}$, in the sense of Bushnell--Kutzko \cite[\S 8]{BK98}, except in the $+$ case with $p > 2$ --- this can be seen from the fact that the inducing representation in the type for $\tilde{G}$ usually has dimension larger than its counterpart for $\tilde{M}$, for example. Instead, we define the embedding as
\[ \mathrm{t}_{\mathrm{nor}} := \text{the composition of}\; H_\psi^{\tilde{M}, \pm} \xrightarrow[\sim]{\mathrm{TW}^{-1}} H^{M^{\pm}} \xrightarrow{\mathrm{t}^{\pm}_{\mathrm{nor}}} H^{\pm} \xrightarrow[\sim]{\mathrm{TW}} H_\psi^{\pm}. \]

The main results in this article are summarized below. Each of them divides into $+$ (even) and $-$ versions that apply to representations in $\mathcal{G}_\psi^+$ and $\mathcal{G}_\psi^-$ respectively. The $-$ case is always more delicate to state and prove.

\begin{Thm}[= Theorem \ref{prop:Hecke-equivariance} and \eqref{eqn:r-compatibility}, \eqref{eqn:i-compatibility}]\label{prop:adjunction-Hecke}
	Let $P = MU$ be a reverse-standard parabolic subgroup of $G$. Via the equivalence $\mathbf{M}: \mathcal{G}_\psi^{\pm} \rightiso H_\psi^{\pm}\dcate{Mod}$ furnished by type theory and its avatar for $\tilde{M}$, the adjunction pair
	\[\begin{tikzcd}
		r_{\tilde{P}}: \mathcal{G}_\psi^{\pm} \arrow[shift left, r] & \mathcal{G}_{\psi}^{\tilde{M}, \pm} : i_{\tilde{P}} \arrow[shift left, l] 
	\end{tikzcd}\]
	corresponds to
	\[\begin{tikzcd}
		\mathrm{t}_{\mathrm{nor}}^*: H_\psi^{\pm}\dcate{Mod} \arrow[shift left, r] & H_\psi^{\tilde{M}, \pm}\dcate{Mod} : \Hom_{H_\psi^{\tilde{M}, \pm}}(H_\psi^{\pm}, \cdot) \arrow[shift left, l]
	\end{tikzcd}\]
	The canonical isomorphisms involved in this correspondence are explicit; see Proposition \ref{prop:Hecke-induction}.
\end{Thm}

Being \emph{reverse-standard} means that $P \supset B^{\leftarrow}$ where $B^{\leftarrow}$ is the Borel subgroup of $G$ containing the standard maximal torus $T$ whose simple roots are
\[ \epsilon_{i+1} - \epsilon_i \quad (1 \leq i < n), \quad 2\epsilon_1; \]
the same terminology applies to parabolic subgroups of $G^{\pm}$ as well. In contrast, the definition of types for $\mathcal{G}_\psi^{\pm}$ uses the standard Iwahori subgroup $I$ of $G(F)$.

The idea of the proof goes as follows. There is a canonical linear map $\mathrm{q}: \mathbf{M}(\pi) \to \mathbf{M}(r_{\tilde{P}}(\pi))$ for any smooth representation $\pi$ of $\tilde{G}$. Based on the results in \cite[\S 3]{TW18}, one can prove that $\mathrm{q}$ is a linear isomorphism, and it is $\mathrm{t}_{\mathrm{nor}}$-equivariant up to constants which depend on elements $T_w^{\tilde{M}} \in H_\psi^{\tilde{M}, \pm}$ (where $w$ ranges over $P$-positive elements in the affine Weyl group of $M$) but not on $\pi$. These constants are then pinned down by computing the case of $\pi = \omega_\psi^{\pm}$. A calculation of co-invariants of types (Proposition \ref{prop:coinvariant-type}) is also needed, and this is why we require $P$ to be reverse-standard.

From this, it is easy to match Aubert involutions between $\mathcal{G}_\psi^{\pm}$ and $\mathcal{G}^{\pm}$, for which we refer to \cite[\S A.2]{KMSW} for a summary.

\begin{Thm}[= Corollary \ref{prop:TW-vs-Aubert}]
	Let $(\cdot)^\wedge$ denote the Aubert involution as a functor on given Bernstein blocks. The following diagram commutes up to a canonical isomorphism
	\[\begin{tikzcd}
		\mathcal{G}_\psi^{\pm} \arrow[d, "{(\cdot)^\wedge}"'] \arrow[r, "\sim"] & \mathcal{G}^{\pm} \arrow[d, "{(\cdot)^\wedge}"] \\
		\mathcal{G}_\psi^{\pm} \arrow[r, "\sim"'] & \mathcal{G}^{\pm}
	\end{tikzcd}\]
	where the horizontal equivalences are induced by $\mathrm{TW}$.
\end{Thm}

We proceed to the comparison of standard (non-normalized) intertwining operators. Let $\Omega_0$ (resp.\ $\Omega^M_0$) be the Weyl groups of $G$ (resp.\ $M$).

\begin{Thm}[= Corollaries \ref{prop:comparison-J-general} and \ref{prop:comparison-J-general-odd}]\label{prop:comparison-J-preview}
	Let us begin with the $+$ case. Consider reverse-standard parabolic subgroups $P = MU \subset G$ and $P^+ = M^+ U^+ \subset G^+$ that match each other in the sense that $M = \Sp(W^\flat) \times \prod_{i=1}^r \GL(n_i)$ and $M^+ = \SO(V^{+, \flat}) \times \prod_{i=1}^r \GL(n_i)$ with $\dim V^{+, \flat} = \dim W^{\flat} + 1$. Let $w \in \Omega_0$ be such that $wMw^{-1} = M$ and $w$ has minimal length in its $\Omega^M_0$-coset.
	
	Let $\pi$ (resp.\ $\sigma$) be of finite length in $\mathcal{G}_\psi^{\tilde{M}, +}$ (resp.\ $\mathcal{G}^{M^+}$) and fix a Hecke-equivariant (relative to $\mathrm{TW}$) isomorphism between their Hecke modules. Consider the standard intertwining operators
	\[ J_w(\pi \otimes \chi) \in \Hom_{\tilde{G}}\left(i_{\tilde{P}}(\pi \otimes \chi), i_{\tilde{P}}({}^{\tilde{w}}(\pi \otimes \chi))\right) \]
	viewed as a rational family in unramified characters $\chi: \prod_{i=1}^r \GL(n_i, F) \to \CC^{\times}$, and its avatar $J^+_w(\sigma \otimes \chi)$ for $M^+ \subset G^+$. By interpreting normalized parabolic induction in terms of
	\[ \Hom_{H_\psi^{\tilde{M}, +}}(H_\psi^+, \cdot), \quad \Hom_{H^{M^+}}(H^+, \cdot) \]
	using Theorem \ref{prop:adjunction-Hecke}, and identifying these spaces via $\mathrm{TW}$, we have
	\[ J_w(\pi \otimes \chi) \xleftrightarrow{\text{matches}} |2|_F^{t(w)/2} J^+_w(\sigma \otimes \chi) \]
	as rational families in $\chi$. Here $t(w)$ is the number of $B^{\leftarrow}$-positive long roots $2\epsilon_i$ such that $w(2\epsilon_i)$ is $B^{\leftarrow}$-negative.
	
	The same holds in the $-$ case under the extra assumption $M \supset M^1 := \Sp(2) \times \GL(1)^{n-1}$. The result is that
	\[ J'_w(\pi \otimes \chi) \xleftrightarrow{\text{matches}} |2|_F^{t(w)/2} J^-_w(\sigma \otimes \chi). \]
\end{Thm}

Some explanations are in order.
\begin{itemize}
	\item Let $z_0$ (resp.\ $z_1$) be the vertex in the standard apartment of the Bruhat--Tits building of $G$ with coordinate $(0, \ldots, 0)$ (resp.\ $(\frac{1}{2}, 0, \ldots, 0)$). The Haar measures on unipotent subgroups involved in $J_w(\pi \otimes \chi)$ (resp.\ $J'_w(\pi \otimes \chi)$) are chosen so that their intersections with $K_0 := \Stab_{G(F)}(z_0)$ (resp.\ $K_1 := \Stab_{G(F)}(z_1)$) have measure one.
	
	The Haar measures involved in $J_w^{\pm}(\sigma \otimes \chi)$ are normalized using suitable maximal compact $K^{\pm} \subset G^{\pm}(F)$; the $-$ case requires some care, see Remark \ref{rem:K-pm}.
	
	\item The definition of $J_w$ (resp.\ $J'_w$) requires a sensible choice of representative $\tilde{w} \in \widetilde{K_0}$ (resp.\ $\mathring{w} \in \widetilde{K_1}$) of $w$; see \S\ref{sec:statement-int-op-even} (resp.\ \S\ref{sec:statement-int-op-odd}).
	
	\item The isomorphism between the Hecke modules of $(\pi, \sigma)$ and that between the $w$-twists $({}^w \pi, {}^w \sigma)$ must be chosen compatibly. See Lemma \ref{prop:Hecke-Weyl}.
	
	\item Finally, a comparison using the same intertwining operators for both the $\pm$ cases is given in Remark \ref{rem:general-Levi-change}. A factor $(-q^{-1})^{t(w)}$ will appear where $q := |\mathfrak{o}/(\varpi)|$.
\end{itemize}

The proof proceeds by reducing to $M = T$ (resp.\ $M = M^1$) in the $+$ (resp.\ $-$) case. By \cite{Ca80}, the standard intertwining operators on the $G^{\pm}$-side can be understood from their effect on spherical vectors, which form a $1$-dimensional vector space. The same holds for $\tilde{G}$ provided that the spherical parts of Hecke modules are defined using a suitable idempotent $e^{\pm} \in H_\psi^{\pm}$ (Definition \ref{def:spherical-idempotent}), which matches its $G^{\pm}$-counterpart under $\mathrm{TW}$ (Proposition \ref{prop:matching-e}). We are then reduced to proving a \emph{Gindikin--Karpelevich formula} for $\tilde{G}$, and compare it with the known version for $G^{\pm}$.

The proof of Gindikin--Karpelevich formula is the (relatively) technical part of this article. First off, it is routine to reduce to the case $G = \Sp(2) \simeq \SL(2)$ (resp.\ $\Sp(4)$) in the $+$ (resp.\ $-$) case.
\begin{description}
	\item[The $+$ case (\S\ref{sec:int-op-even})] Using an explicit Iwasawa decomposition in $\widetilde{\SL}(2, F)$, we evaluate the intertwining integral for well-chosen test functions from the type, which boils down to an explicit $p$-adic integral (Lemma \ref{prop:main-integral}) concerning averaged Weil constants over $\mathfrak{o}^{\times}$.
	
	\item[The $-$ case (\S\ref{sec:int-op-odd})] The scenario is similar, but the Iwasawa decomposition and the relevant intertwining integrals are not amenable to direct computation, for the authors at least. Instead, we minimize the amount of computations by exploiting the matching of $\mu$-functions from \cite{TW18} together with Knapp--Stein theory, thus reduce the comparison of intertwining integrals to the ``zeroth term'' only (Lemma \ref{prop:main-integral-odd}).
\end{description}

In principle, a uniform proof using the matching of $\mu$-functions can be given for both cases.

For the next result, observe that since $\mu$-functions are matched via $\mathrm{TW}$, one can also match normalizing factors in the sense of Arthur \cite[\S 2]{Ar89-IOR1} between $\mathcal{G}_\psi^{\pm}$ and $\mathcal{G}^{\pm}$, up to positive constants depending solely on the parabolic subgroups in question (Proposition \ref{prop:matching-r}). In this manner, we obtain the normalized versions $R_w(\pi)$ (resp.\ $R'_w(\pi)$) of $J_w(\pi)$ (resp.\ $J'_w(\pi)$), where $\pi$ is an object of $\mathcal{G}_\psi^{\tilde{M, +}}$ (resp.\ $\mathcal{G}_\psi^{\tilde{M, -}}$) of finite length; see \eqref{eqn:R-operator}.

There are similar normalized operators $R_w^{\pm}(\sigma)$ on the $G^{\pm}$-side for $\sigma$ corresponding to $\pi$, defined by compatible normalizing factors. All normalized operators are unitary when the inducing data are unitarizable.

In the statements below, $\tilde{w} \in \tilde{G}$ (resp.\ $\mathring{w} \in \tilde{G}$) and $\ddot{w} \in G^{\pm}(F)$ will denote canonical representatives for $w$ in $\Omega_0$ (resp.\ in $\Omega'_0 := \Stab_{\Omega_0}(\epsilon_1)$ in the $-$ setting): the elements $\tilde{w}$ and $\mathring{w}$ have been explained after Theorem \ref{prop:comparison-J-preview}, and $\ddot{w}$ is similarly (and more easily) defined using $F$-pinnings.

\begin{Thm}[= Theorem \ref{prop:intertwining-relation}]\label{prop:int-rel-preview}
	Let $P = MU$ be a reverse-standard parabolic subgroup of $G$. Let $\pi$ be irreducible in $\mathcal{G}_\psi^{\tilde{M}, +}$ and $\sigma \in \mathcal{G}^{M^+}$ be the corresponding object. Assume that $\pi$ and its Aubert dual $\hat{\pi}$ are both unitarizable (equivalently, $\sigma$ and $\hat{\sigma}$ are both unitarizable).
	
	Given $w \in \Omega_0$ such that $wMw^{-1} = M$ and $w$ has minimal length in its $\Omega^M_0$-coset, we fix an isomorphism $A(\tilde{w}): {}^{\tilde{w}} \pi \rightiso \pi$; it induces $\hat{A}(\tilde{w}): {}^{\tilde{w}} \hat{\pi} \rightiso \hat{\pi}$ canonically. We also have $A(\ddot{w})$ and $\hat{A}(\ddot{w})$ on the $M^+$-side.

	Then there exists $c^+(w) \in \CC^{\times}$, depending on $\pi$ but independent of the choice of $A(\tilde{w})$ and $A(\ddot{w})$, such that
	\begin{align*}
		\Tr\left( \hat{A}(\tilde{w}) R_w(\hat{\pi}) i_{\tilde{P}}(\hat{\pi}, f) \right) & = c^+(w) \Tr\left( \left(A(\tilde{w}) R_w(\pi) \right)^{\wedge} \; i_{\tilde{P}}(\pi)^\wedge(f) \right), \\
		\Tr\left( \hat{A}(\ddot{w}) R^+_w(\hat{\sigma}) i_{P^+}(\hat{\sigma}, f^+) \right) & = c^+(w) \Tr\left( \left(A(\ddot{w}) R^+_w(\sigma)\right)^\wedge \; i_{P^+}(\sigma)^{\wedge} (f^+) \right)
	\end{align*}
	for all anti-genuine $f \in \mathcal{H}(\tilde{G})$ (resp.\ $f^+ \in \mathcal{H}(G^+)$). Here $\mathcal{H}(\tilde{G})$ (resp.\ $\mathcal{H}(G^{\pm})$) denotes the Hecke algebra of $\tilde{G}$ (resp.\ $G^{\pm}$), and being anti-genuine means that $f(zx) = z^{-1} f(x)$ for all $z \in \bmu_8$.

	The same holds in the $-$ case, with $w \in \Omega'_0$, representative $\mathring{w}$ and another constant $c^-(w) \in \CC^{\times}$.
\end{Thm}

The equivalence of unitarizability follows from the fact that $\mathrm{TW}$ preserves Hilbert structures; see \cite[Corollary C]{BHK11}. The proof of the Theorem \ref{prop:int-rel-preview} proceeds by combining all the previous results, which implies the desired trace identities for $f$ coming from $H_\psi^{\pm}$. We then extend them to all $f$ using type theory.

The behavior of traces of intertwining operators under Aubert involution is crucial for Arthur's approach to LIR; see \cite[\S 7.1]{Ar13} and \cite[\S A.3]{KMSW}. Roughly speaking, the result above says that normalized intertwining operators (or their traces) commute with $(\cdot)^\wedge$ up to a constant.  The constants $c^{\pm}(w)$ are uniquely determined by the trace identities on $G^{\pm}$ when $\sigma$ (equivalently, $\pi$) is tempered; their precise description is part of Arthur's program, where endoscopy enters the picture.

The upshot is that the same constant works for both $\mathcal{G}_\psi^{\pm}$ and $\mathcal{G}^{\pm}$ if the inducing data match under $\mathrm{TW}$.

We close this introduction with several short remarks.

\begin{enumerate}
	\item The precise statement of the $-$ case of Theorem \ref{prop:int-rel-preview} uses the operators $R'_w(\pi)$, in which the representative $\mathring{w}$ and the Haar measures are chosen differently from the $+$ case. The difference turns out to be immaterial, see Remark \ref{rem:intertwining-relation-odd}.
	\item As in \cite{TW18}, to make all the results applicable $p=2$, we do not make any use of ``lattice models'' for Weil representations in this work.
	\item The conductor of $\psi$ is fixed throughout. In concrete applications, the results above should be used in a way that is independent of $\psi$. We do not pursue this point in the present work.
	\item In the $+$ case and $p > 2$, the description of parabolic inductions and Jacquet modules in terms of Hecke modules is already mentioned by Gan--Savin \cite{GS12}. Indeed, the theory of covers of Bushnell--Kutzko \cite[\S 8]{BK98} is applicable under these assumptions.
\end{enumerate}

\subsection*{Structure of this article}

In \S\S\ref{sec:Hecke-coinvariants}--\ref{sec:Mp-groups}, we review the basic definition of Hecke algebras, symplectic vector spaces and metaplectic groups. The relevant results are recorded for later use.

In \S\ref{sec:coinvariant}, we review the types constructed in \cite{TW18}, determine their co-invariants and record the Iwahori decomposition with respect to reverse-standard parabolic subgroups.

The Takeda--Wood isomorphism together with the relevant combinatorics are reviewed in \S\ref{sec:TW}; the determination of the $H_\psi^{\pm}$-modules associated with $\omega_\psi^{\pm}$ is also included.

The Hecke-equivariance of the co-invariant map $\mathrm{q}$ is proved in \S\ref{sec:Hecke-homomorphism}. It serves as the basis of the subsequent sections.

In \S\ref{sec:spherical-parts}, we define the idempotents $e^{\pm}$ and the spherical parts of $H_\psi^\pm$-modules. These idempotents are shown to match their counterparts for $G^{\pm}$ for explicitly chosen maximal compact subgroups $K^{\pm} \subset G^{\pm}(F)$. Consequently, the unramified principal series (in a suitable sense) of $\tilde{G}$ and $G^{\pm}$ are also matched.

In \S\S\ref{sec:int-op-even}--\ref{sec:int-op-odd}, we prove the Gindikin--Karpelevich formula for $\tilde{G}$, and deduce the matching of standard intertwining operators for $\tilde{G}$ and $G^{\pm}$ under the Takeda--Wood isomorphism. The aforementioned applications are given in \S\ref{sec:applications}.

\subsection*{Acknowledgements}
The authors are grateful to Kazuma Ohara, Shuichiro Takeda, and Shaoyun Yi for helpful conversations. Our thanks also go to the referee for numerous advices.

\subsection*{Conventions}
\paragraph*{Local fields}
Throughout this article, $F$ denotes a non-Archimedean local field with $\mathrm{char}(F) \neq 2$. Let
\begin{itemize}
	\item $\mathfrak{o}$: the ring of integers in $F$,
	\item $\mathrm{val}_F$: the normalized valuation of $F$,
	\item $\varpi \in \mathfrak{o}$: a chosen uniformizer,
	\item $p := \mathrm{char}(\mathfrak{o}/(\varpi))$,
	\item $q := |\mathfrak{o}/(\varpi)|$,
	\item $e := \mathrm{val}_F(2)$,
	\item $|\cdot|_F := q^{-\mathrm{val}_F(\cdot)}$: the normalized absolute value on $F$.
\end{itemize}

The symbol $\psi$ stands for a chosen non-trivial additive character $F \to \CC^{\times}$. Unless otherwise specified, it is assumed that
\[ \psi|_{4\mathfrak{o}} = 1, \quad \psi|_{4\varpi^{-1}\mathfrak{o}} \not\equiv 1. \]

\paragraph*{Varieties and groups}
For a scheme $X$ over a commutative ring, the set of $R$-points in $X$ is denoted by $X(R)$ for any commutative algebra $R$. In this article we only consider affine schemes of finite type, and $X(R)$ will be endowed with topology if $R$ is a topological ring.

Given a connected reductive group $G$ over $F$, we denote by $A_G$ the maximal split central torus in $G$. If $P \subset G$ is a parabolic subgroup, we set $A_P := A_M$ where $M$ is the Levi quotient of $P$. Let $X^*(G) := \Hom_{F\dcate{Grp}}(G, \Gm)$. If $T$ is an $F$-torus then $X_*(T) := \Hom_{F\dcate{Grp}}(\Gm, T)$.

If $U$ is a subscheme of $G$ and $g \in G(F)$, we use the shorthand
\[ {}^g U := gUg^{-1}, \quad U^g := g^{-1} U g; \]
the same notation applies to subsets of $G(F)$ as well.

Given a minimal Levi subgroup $M_0$ (resp.\ minimal parabolic subgroup $P_{\min}$) of $G$, a parabolic subgroup $P$ of $G$ is called semi-standard (resp.\ standard) if $P \supset M_0$ (resp.\ $P \supset P_{\min}$). Every semi-standard $P$ has a unique Levi decomposition $P = MU$ where $M \supset M_0$ is the Levi factor, and $U$ is the unipotent radical of $P$. Denote the opposite of $P$ by $\overline{P} = M\overline{U}$.

For a Levi subgroup $M$ of $G$, we denote by $\mathcal{P}(M)$ the set of parabolic subgroups with Levi factor $M$. We say $M$ is standard if $M \supset M_0$ and there exists $P \in \mathcal{P}(M)$ that is standard. These notions are relative to $P_{\min}$.

If $T$ is a maximal torus in $G$, and $\alpha \in X^*(T)$ is a root with coroot $\check{\alpha}$, we will denote any prescribed isomorphism from the additive group onto the corresponding root subgroup as
\[ x_\alpha: \Ga \hookrightarrow G. \]

The standard Borel pair of $\GL(n)$ is formed by upper triangular and diagonal matrices, respectively. For the non-Archimedean local field $F$, the standard Iwahori subgroup of $\GL(n, F)$ is defined by the same convention.

The letter $\Omega$, with appropriate decorations, will denote various Weyl groups.

\paragraph*{Fourier analysis}
Denote by $\Schw(F)$ the space of Schwartz functions on the given local field $F$. The subspace of even (resp.\ odd) functions is denoted as $\Schw^+(F)$ (resp.\ $\Schw^-(F)$).

In this article, the Fourier transform is defined to be
\[ \hat{f}(x) = \int_F f(y) \psi(2xy) \dd y, \quad f \in \Schw(F) \]
with respect to some Haar measure on $F$. There exists a unique Haar measure on $F$ such that $f^{\wedge\wedge}(x) = f(-x)$. In this case we say that the Haar measure is self-dual, and $f \mapsto \hat{f}$ is the unitary Fourier transform.

More generally, for every $n \in \Z_{\geq 1}$ we define $\Schw(F^n) \supset \Schw^{\pm}(F^n)$ and the unitary Fourier transform
\begin{equation}\label{eqn:unitary-Fourier}
	\phi \mapsto \hat{\phi}(x) := \int_{F^n} \psi(2 \transp{y} x) \phi(y) \dd\mu(y), \quad f \in \Schw(F^n),
\end{equation}
where $\mu$ is the self-dual Haar measure on $F^n$, equal to the product of self-dual measures on $F$.

Schwartz spaces are also defined for lattices in $F^n$, i.e.\ by restricting the supports of $f$, as well as quotients of lattices by sub-lattices, i.e.\ by imposing periodicity.

Note that these definitions work for any nontrivial additive character $\psi': F \to \CC^{\times}$. There is a corresponding notion of self-dual Haar measure on $F$ relative to $\psi'$.

\paragraph*{Weil constant}
For every $b \in F^{\times}$, let $\psi_b: F \to \CC^{\times}$ be the additive character $\psi_b(x) = \psi(bx)$. Let $\mu_b$ denote the self-dual Haar measure on $F$ relative to $\psi_b$, in the sense defined above.

Now let $a \in F^{\times}$. Following \cite{Weil64}, the Weil constant $\gamma_\psi(a)$ is defined to be
\[ \text{the principal value of}\quad \int_F \psi\left( \frac{ax^2}{2} \right) \dd\mu_{a/2}(x). \]
This is an $8$-th root of $1$ in $\CC^{\times}$, and depends only on the coset $a F^{\times 2}$. The principal value means the limit of $\int_{\varpi^{-k} \mathfrak{o}}$ as $k \to +\infty$.

The form of Weil constants given above is compatible with \cite{Li11} and its sequels, although our definition \eqref{eqn:unitary-Fourier} of Fourier transform differs.

\paragraph*{Covering groups}
Let $G$ be a connected reductive group over a local field $F$. Let $\bmu_m := \left\{z \in \CC^{\times} : z^m = 1 \right\}$. An $m$-fold covering group of $G(F)$ in this article means a central extension of locally compact groups
\[ 1 \to \bmu_m \to \tilde{G} \xrightarrow{\pr} G(F) \to 1. \]

For every subset $E$ of $G(F)$, we write $\tilde{E} := \pr^{-1}(E)$. If $H$ is a subgroup of $G$ in the algebraic sense, we write $\tilde{H} := \pr^{-1}(H(F)) \subset \tilde{G}$. By a Levi (resp.\ parabolic) subgroup of $\tilde{G}$, we mean $\tilde{M}$ (resp.\ $\tilde{P}$) where $M$ (resp.\ $P$) is a Levi (resp.\ parabolic) subgroup of $G$.

Following \cite[\S 2.5]{Li14a}, the Haar measures on $G(F)$ and $\tilde{G}$ will be matched in the manner that $\mes(\tilde{E}) = \mes(E)$ for all measurable subset $E \subset G(F)$. The modulus character $\delta_{\tilde{H}}$ of $\tilde{H}$ equals $\delta_H \circ \pr$ for all $H \subset G$.

For every $x \in \tilde{G}$ and $g \in G(F)$ with $\tilde{g} \in \pr^{-1}(g)$, the conjugate $\tilde{g}x\tilde{g}^{-1}$ depends only on $g$ and $x$, and can be denoted as $gxg^{-1}$. Given $g \in G(F)$ and a subset $V$ of $\tilde{G}$, we may write
\[ {}^g V := gVg^{-1}, \quad V^g := g^{-1} V g. \]

Unipotent radicals of parabolic subgroups split canonically and equivariantly in any covering by \cite[App.\ I]{MW94}; for the case $\mathrm{char}(F) = 0$, see also \cite[Proposition 2.2.1]{Li14a}. Therefore, if $P$ is a parabolic subgroup of $G$ with Levi decomposition $P = MU$, then $\tilde{P} = \tilde{M} U(F)$; we call this a Levi decomposition of $\tilde{P}$.

In this article, the homomorphism $x_\alpha: \Gm \to G$ attached to a root $\alpha$ will come with canonical lifting on the level of $F$-points, denoted as
\[ \tilde{x}_\alpha: F \to \tilde{G}. \]
It is a continuous homomorphism, and $\pr \circ \tilde{x}_\alpha = x_\alpha$.

\paragraph*{Representations}
The representations are realized on $\CC$-vector spaces in this article. For every locally profinite group $G$, we write
\[ G\dcate{Mod} := \text{the category of smooth representations of $G$}. \]
Similarly, for every unital $\CC$-algebra $A$, we write
\[ A\dcate{Mod} := \text{the category of (left) $A$-modules}. \]

For every homomorphism $\varphi: A \to B$, we obtain the change-of-rings functor $\varphi^*: B\dcate{Mod} \to A\dcate{Mod}$.

The underlying space of a representation $\tau$ is denoted by $V_\tau$. The contragredient of $\tau$ is denoted by $\check{\tau}$. The trivial representation of a group $G$ is denoted by $\mathbf{1}$.

If $\theta: G \to G'$ is an isomorphism, we denote by ${}^\theta \pi$ the representation of $G'$ on the same space $V_\pi$ with
\[ ({}^\theta \pi)(g') = \pi(\theta^{-1}(g')). \]
If $\theta$ is of the form $\theta(g) = wgw^{-1}$, the shorthand $({}^w \pi)(g) = \pi(w^{-1}g w)$ will also be employed.

In the case of a covering group
\[ 1 \to \bmu_m \to \tilde{G} \to G(F) \to 1, \]
we say an object $\pi$ of $\tilde{G}\dcate{Mod}$ is genuine if $\pi(z) = z \cdot \identity_{V_\pi}$ for all $z \in \bmu_m$. We say $f: \tilde{G} \to \CC$ is anti-genuine if $f(zx) = z^{-1} f(x)$ for all $z \in \bmu_m$.

\paragraph*{Operations on representations}
Given a locally profinite group $G$ and its closed subgroup $H$, the functor of smooth (resp.\ compactly supported) induction $H\dcate{Mod} \to G\dcate{Mod}$ is denoted by $\Ind^G_H$ (resp.\ $\cInd^G_H$).

For every $\pi$ in $G\dcate{Mod}$, the module of invariants (resp.\ co-invariants) under $H$-action is denoted by $\pi^H$ (resp.\ $\pi_H$): it is the maximal subobject (resp.\ quotient) on which $H$ acts trivially.

When $G$ is a connected reductive group over a non-Archimedean local field $F$ and $P$ is a parabolic subgroup with Levi quotient $M = P/U$, we have the adjoint pair
\[\begin{tikzcd}
	r_P: G(F)\dcate{Mod} \arrow[shift left, r] & M(F)\dcate{Mod}: i_P \arrow[shift left, l]
\end{tikzcd}\]
where
\begin{itemize}
	\item $i_P(\rho) = \Ind^{G(F)}_{P(F)}\left(\rho \otimes \delta_P^{\frac{1}{2}}\right)$ is the normalized parabolic induction, where $\rho$ is inflated to $P(F)$ first;
	\item $r_P(\pi) = \left( \pi \otimes \delta_P^{-\frac{1}{2}} \right)_{U(F)}$ is the normalized co-invariants, also known as Jacquet functor applied to $\pi$, where $\pi$ is restricted to $P(F)$ first.
\end{itemize}

In the case of a covering $\tilde{G}$ of $G(F)$, the same recipe gives the adjoint pair
\[\begin{tikzcd}
	r_{\tilde{P}}: \tilde{P}\dcate{Mod} \arrow[shift left, r] & \tilde{G}\dcate{Mod}: i_{\tilde{P}} \arrow[shift left, l]
\end{tikzcd}\]
by splitting $\pr: \tilde{G} \to G(F)$ canonically over $U(F)$.

\section{Hecke algebras and co-invariants}\label{sec:Hecke-coinvariants}
\subsection{Generalities on Hecke algebras}\label{sec:Hecke-generality}
The following is a quick review of \cite[\S 2]{BK98}; see also \cite[\S 1.5]{TW18}.

Throughout \S\S\ref{sec:Hecke-generality}---\ref{sec:Hecke-coinvariant}, we take $G$ to be a unimodular locally profinite group, and $K^i \subset G$ be compact open subgroups ($i=1,2$). Given finite-dimensional representations $\tau_i$ of $K^i$, we define
\[ \mathcal{H}(G; \tau_1, \tau_2) := \left\{\begin{array}{r|l}
	h: G \to \Hom_{\CC}(V_{\tau_1}, V_{\tau_2}) & \text{compactly supported} \\
	& \forall k_i \in K^i, \; g \in G \\
	& h(k_1 g k_2) = \tau_1(k_1) h(g) \tau_2(k_2)
\end{array}\right\}. \]
The support of $h \in \mathcal{H}(G; \tau_1, \tau_2)$ will be denoted by $\Supp(h)$.

Endow $G$ with the Haar measure such that $\mes(K^1)=1$. Every element of $\mathcal{H}(G; \tau_1, \tau_2)$ induces by convolution a $G$-equivariant map $\cInd_{K_1}^G(\tau_1) \to \cInd_{K_2}^G(\tau_2)$. The formula for the case $\tau_1 = \tau_2$ will be given below.

Now we take a compact open subgroup $K$ of $G$, a finite-dimensional representation $\tau$ of $K$, and use the Haar measure with $\mes(K)=1$. Our Hecke algebra in question is
\[ \mathcal{H}(G \sslash K, \tau) := \mathcal{H}(G; \check{\tau}, \check{\tau}). \]
This is a $\CC$-algebra under convolution $\star$. There is an isomorphism of algebras
\[\begin{tikzcd}[row sep=tiny]
	\mathcal{H}(G \sslash K, \tau) \arrow[r, "\sim"] & \End_G\left( \cInd_K^G \check{\tau} \right) \\
	h \arrow[mapsto, r] & {\left[ f \mapsto \int_G h(t) f(t^{-1} \bullet) \dd t\right]}.
\end{tikzcd}\]

We also have an isomorphism $\mathcal{H}(G \sslash K, \tau) \rightiso \mathcal{H}(G \sslash K, \check{\tau})^{\mathrm{op}}$ between algebras, mapping $f$ to $\check{f}: g \mapsto {}^{\mathrm{t}} f(g^{-1})$, where ${}^{\mathrm{t}}(\cdots)$ means the transposed operator on $V_\tau$.

For every smooth representation $\pi$ of $G$, we deduce a right action of $\mathcal{H}(G \sslash K, \check{\tau})$ on
\[ \Hom_G\left( \cInd_K^G \tau, \pi \right) \simeq \Hom_K(\tau, \pi). \]

In view of the anti-isomorphism $h \mapsto \check{h}$, this can be turned into a left action of $\mathcal{H}(G \sslash K, \tau)$. The precise description goes as follows: for every $\sigma \in \Hom_K(\tau, \pi)$ and $h \in \mathcal{H}(G \sslash K, \tau)$, we have
\[ (h\sigma)(w) = \int_G \pi(g) \sigma\left( \check{h}(g^{-1}) w \right) \dd g, \quad w \in V_\tau. \]

Yet another recipe is to identify $\Hom_K(\tau, \pi)$ with $(V_\pi \otimes V_\tau^\vee)^K$, on which $h \in \mathcal{H}(G \sslash K, \tau)$ acts on the left via
\begin{equation}\label{eqn:Hecke-action-tensor}
	\sum_i v_i \otimes e_i \mapsto \int_G \sum_i \pi(g)v_i \otimes h(g) e_i \dd g.
\end{equation}

We remind the reader that the isomorphism $\Hom_K(\tau, \pi) \rightiso \Hom_G(\cInd_K^G \tau, \pi)$ is given by mapping $\theta$ to $[f \mapsto \sum_{gK} \pi(g) \theta(f(g^{-1})) ]$, where $f \in \cInd_K^G \tau$.

\begin{lemma}\label{prop:Hecke-Ind}
	Let $H$ be another compact open subgroup of $G$ containing $K$, so that $\mathcal{H}(H \sslash K, \tau)$ is embedded into $\mathcal{H}(G \sslash K, \tau)$. There is a commutative diagram
	\[\begin{tikzcd}
		\mathcal{H}(H \sslash K, \tau) \arrow[hookrightarrow, d] \arrow[r, "\sim"] & \End_H\left( \cInd_K^H(\check{\tau}) \right) \arrow[d] \\
		\mathcal{H}(G \sslash K, \tau) \arrow[r, "\sim"] & \End_G\left( \cInd_K^G(\check{\tau}) \right)
	\end{tikzcd}\]
	where the rightmost vertical arrow comes from the functoriality of $\cInd^G_H$ and induction in stages.
\end{lemma}
\begin{proof}
	This is a routine and unsurprising check using the description of the horizontal isomorphisms given above. We omit the details.
\end{proof}

\begin{lemma}\label{prop:idempotent-projector}
	Let $K'$ be a compact open subgroup of $G$ that contains $K$. Suppose that $\tau$ extends to a representation of $K'$. Define an element $e'$ of $\mathcal{H}(G \sslash K, \tau)$ as follows
	\[ e'(g) := \begin{cases}
		\mes(K')^{-1} \check{\tau}(g), & g \in K' \\
		0, & g \notin K'.
	\end{cases}\]
	Then $e'$ is an idempotent in $\mathcal{H}(G \sslash K, \tau)$. Its action on $\Hom_K(\tau, \pi)$ is a projection onto the subspace $\Hom_{K'}(\tau', \pi)$.
\end{lemma}
\begin{proof}
	The fact that $e' \in \mathcal{H}(G \sslash K, \tau)$ and the idempotency of $e'$ is clear. Suppose that $\sigma \in \Hom_K(\tau, \pi)$ corresponds to $\sum_i v_i \otimes e_i \in (V_\pi \otimes V_\tau^\vee)^K$ as before, then $e' \sigma$ corresponds to
	\[ \mes(K')^{-1} \int_{K'} \sum_i \pi(g)v_i \otimes \check{\tau}(g) e_i \dd g. \]
	
	It is by now clear that $\sigma \mapsto e'\sigma$ is a projection onto the subspace of $K'$-invariants.
\end{proof}

Now comes the master functor of this work.

\begin{definition}\label{def:master-functor}
	Given $K \subset G$ and $\tau$ as above, we have the functor
	\[ \mathbf{M}_\tau := \Hom_K(\tau, \cdot): G\dcate{Mod} \to \mathcal{H}(G \sslash K, \tau)\dcate{Mod}. \]
\end{definition}

Each of the integrals over $G$ appeared above is a finite sum over $G/K$, so the theory can defined with values in fields or rings other than $\CC$. Nonetheless, the structure to be reviewed below are specific to complex coefficients. The reference is \cite[\S 4]{BHK11}.

Assume that $\tau$ is unitary, thus so is $\check{\tau}$. There is an anti-linear involution on $\mathcal{H}(G \sslash K, \tau)$ given by
\[ h \mapsto h^*(g) = {}^\dagger h(g^{-1}), \]
where $g \in G$ and ${}^\dagger (\cdot)$ denotes the Hermitian adjoint in $\End_{\CC}(V_{\check{\tau}})$. This makes $(\mathcal{H}(G \sslash K, \tau), *)$ into a Hilbert algebra (see \cite[\S 3.1]{BHK11} for this notion), the inner product being
\[ (h_1 | h_2) := (\dim \tau)^{-1} \Tr\left( h_1^* \star h_2 (1_G) \right). \]

One can then define the notions of unitarity, temperedness, discrete series for $\mathcal{H}(G \sslash K, \tau)$-modules; there is also a Plancherel decomposition. They match a piece of the counterparts on the group level; see \textit{loc.\ cit.} for a detailed discussion.

\subsection{Operations on Hecke algebras}\label{sec:operations-Hecke}
Various operations on Hecke algebras (resp.\ their modules) can be normalized to preserve the Hilbert structure (resp.\ unitarity). This often involves the following operation.

\begin{definition}\label{def:character-twist}
	Let $\chi: G \to \CC^{\times}$ be a continuous character, with $\chi|_K = 1$. Then the map $h \mapsto h\chi$ with $(h\chi)(g) := h(g) \chi(g)$ is an automorphism of the algebra $\mathcal{H}(G \sslash K, \tau)$.
\end{definition}

When $\chi$ is unitary, the automorphism $h \mapsto h\chi$ also preserves the Hilbert structure. We refer to \textit{loc.\ cit.} for the routine details.

Another basic operation is the transport of structures. Suppose that an isomorphism $\theta$ of locally compact groups
\[\begin{tikzcd}[row sep=tiny]
	G \arrow[r, "\sim", "\theta"'] & G' \\
	K \arrow[phantom, u, "\subset" description, sloped] \arrow[r, "\sim"'] & K' \arrow[phantom, u, "\subset" description, sloped]
\end{tikzcd}\]
is given, such that $\tau' = {}^\theta \tau$. Here $(G, K, \tau)$ and $(G', K', \tau')$ are as in \S\ref{sec:Hecke-generality}. Then $\theta$ induces an isomorphism of Hilbert algebras
\begin{equation}\label{eqn:Hecke-transport}\begin{aligned}
	\mathcal{H}(G \sslash K, \tau) & \rightiso \mathcal{H}(G' \sslash K', \tau') \\
	h & \mapsto h \circ \theta^{-1}.
\end{aligned}\end{equation}

\begin{proposition}\label{prop:Hecke-transport}
	Let $\pi$ be a smooth representation of $G$ and set $\pi' := {}^\theta \pi$. The identity map $\mathbf{M}_{\tau}(\pi) \xrightarrow{=} \mathbf{M}_{\tau'}(\pi')$ is equivariant with respect to \eqref{eqn:Hecke-transport}.
\end{proposition}
\begin{proof}
	By the definition of $\mathbf{M}(\cdots)$, it is clear that $\mathbf{M}_{\tau}(\pi) = \mathbf{M}_{{}^\theta \tau}({}^\theta \pi) = \mathbf{M}_{\tau'}(\pi')$ as vector spaces. The equivariance is a straightforward transport of structure, using the explicit formula \eqref{eqn:Hecke-action-tensor}.
\end{proof}

Next, denote by $\mathcal{H}(G)$ the space of smooth and compactly supported functions on $G$, which is a (non-necessarily unital) algebra under the convolution product $\star$, with the Haar measure satisfying $\mes(K)=1$. We now state the following relation between $\mathcal{H}(G \sslash K, \tau)\dcate{Mod}$ and a certain unital subalgebra of $\mathcal{H}(G)$.

\begin{proposition}\label{prop:Hecke-Morita}
	Let $e_\tau \in \mathcal{H}(G)$ be given by
	\[ e_\tau(g) := \begin{cases}
		\dim\tau \cdot \Tr(\tau(g^{-1})), & g \in K \\
		0, & g \notin K.
	\end{cases}\]
	It is an idempotent, thus $e_\tau \star \mathcal{H}(G) \star e_\tau$ is a unital algebra. Moreover, we have an isomorphism of Hilbert algebras
	\begin{align*}
		\Upsilon_\tau: \mathcal{H}(G \sslash K, \tau) \otimes \End_{\CC}(V_\tau) & \rightiso e_\tau \star \mathcal{H}(G) \star e_\tau \\
		h \otimes a & \mapsto \dim\tau \cdot \Tr\left( h(\cdot) \; \transp{a} \right)
	\end{align*}
	where $\transp{a} \in \End_{\CC}(V_{\check{\tau}})$ denotes the transpose of $a$.
\end{proposition}
\begin{proof}
	We refer to \cite[\S 4]{BHK11} for the part concerning Hilbert structures. The other assertions are all in \cite[(2.12)]{BK98}.
\end{proof}

Hence $\mathcal{H}(G \sslash K, \tau)\dcate{Mod}$ is equivalent to $e_\tau \star \mathcal{H}(G) \star e_\tau \dcate{Mod}$ by Morita theory.

Let $\pi$ be a smooth representation of $G$. Then $\pi^\tau := \pi(e_\tau) V_\pi$ is an $e_\tau \star \mathcal{H}(G) \star e_\tau$-module. We can consider the $\mathcal{H}(G \sslash K, \tau) \otimes \End_{\CC}(V_\tau)$-module $\Upsilon_\tau^* (\pi^\tau)$. The next result relates it to $\mathbf{M}_\tau(\pi)$ in the expected way.

\begin{proposition}\label{prop:trace-type}
	Let $\pi$ be a smooth representation of $G$. There is a functorial isomorphism of $\mathcal{H}(G \sslash K, \tau) \otimes \End_{\CC}(V_\tau)$-modules
	\begin{align*}
		\Psi: \mathbf{M}_\tau(\pi) \otimes \tau & \rightiso \Upsilon_\tau^* (\pi^\tau) \\
		\sigma \otimes \phi & \mapsto \sigma(\phi).
	\end{align*}
	
	Moreover, when $\pi$ is admissible and $E \in \End_G(\pi)$, we have
	\[ \dim\tau \cdot \Tr\left(\mathbf{M}_\tau(E) \mathbf{M}_\tau(\pi)(h) \right) = \Tr\left(E \pi( \Upsilon_\tau(h \otimes 1)) \right) \]
	for all $h \in \mathcal{H}(G \sslash K, \tau)$.
\end{proposition}
\begin{proof}
	The isomorphism $\Psi$ is established in \cite[(2.13)]{BK98}. Also note that the action of $\mathbf{M}_\tau(E) \otimes \identity$ sends $\sigma \otimes \phi$ to $E\sigma \otimes \phi$, which matches the action of $E$ on $\sigma(\phi)$ via $\Psi$. 
	
	Under $\Psi$, the operator $\mathbf{M}_\tau(\pi)(h) \otimes \identity$ on $\mathbf{M}_\tau(\pi) \otimes \tau$ corresponds to $\Upsilon_\tau(h \otimes 1)$ acting on $\pi^\tau$, thus the trace identity follows. For the special case $E = \identity$, see also \cite[p.73]{BHK11}.
\end{proof}

\subsection{Co-invariants}\label{sec:Hecke-coinvariant}
Keep the conventions in \S\ref{sec:Hecke-generality}, and assume that
\begin{itemize}
	\item $G$ is either the group of $F$-points of a connected reductive group, or a covering thereof;
	\item a parabolic subgroup $P$ of $G$ and its Levi decomposition $P=MU$ are given;
	\item a compact open subgroup $K \subset G$ is given, and we assume that $K \cap P = K_M K_U$ where $K_M := K \cap M$ and $K_U := K \cap U$; note that this implies the multiplication map $K_M \times K_U \to K \cap P$ is a homeomorphism.
\end{itemize}

Therefore, we have functors of co-invariants, abbreviated in the same way as
\begin{align*}
	(\cdot)_U: G\dcate{Mod} & \to M\dcate{Mod}, \\
	(\cdot)_{K_U}: K\dcate{Mod} & \to K_M\dcate{Mod}.
\end{align*}

We emphasize that the co-invariants here are non-normalized.

\begin{remark}
	When constrained to finite-dimensional representations of $K$, the functor $(\cdot)_{K_U}$ commutes with taking smooth contragredient, since one can pass between co-invariants and invariants by compactness.
\end{remark}

Given a smooth representation $\pi$ of $G$, it makes sense to compare $\mathbf{M}_\tau(\pi)$ and $\mathbf{M}_{\tau_{K_U}}(\pi_U)$, at least as vector spaces. The following construction mimics \cite[(7.9)]{BK98}.

\begin{proposition}\label{prop:q-map}
	Under the assumptions above, we have a functorial $\CC$-linear map
	\[ \mathrm{q} = \mathrm{q}_U: \mathbf{M}_\tau(\pi) \to \mathbf{M}_{\tau_{K_U}}(\pi_U). \]
	It sends $f \in \Hom_K(\tau, \pi)$ to the $K_M$-equivariant map $\overline{\phi} \mapsto \overline{f(\phi)}$, where $\overline{\phi} \in V_{\tau_{K_U}}$ has a preimage $\phi \in V_\tau$, and $\overline{f(\phi)} \in V_{\pi_U}$ denotes the image of $f(\phi) \in V_\pi$.
	
	The following transitivity holds: if we have another parabolic $Q \subset P$ with Levi decomposition $Q = LV$ such that $L \subset M$ and
	\[ (K_M) \cap (M \cap Q) = (K_M)_L (K_M)_{M \cap V}, \]
	then $K \cap Q = K_L K_V$, and the composite of
	\[ \mathbf{M}_\tau(\pi) \xrightarrow{\mathrm{q}_U} \mathbf{M}_{\tau_{K_U}}(\pi_U) \xrightarrow{\mathrm{q}_{M \cap V}} \mathbf{M}_{(\tau_{K_U})_{(K_M)_{M \cap V}}}((\pi_U)_{M \cap V}) \]
	equals $\mathrm{q}_V: \mathbf{M}_\tau(\pi) \to \mathbf{M}_{\tau_V}(\pi_V)$.
\end{proposition}
\begin{proof}
	Since $f$ is $K$-equivariant, if $\phi$ is replaced by $\phi + u\phi'-\phi'$ where $\phi' \in V_\tau$ and $u \in K_U$, then $f(\phi + u\phi'-\phi') = f(\phi) + uf(\phi') - f(\phi')$ has the same image as $f(\phi)$ in $V_{\pi_U}$. The linear map $\mathrm{q}$ is thus well-defined. The $K_M$-equivariance follows immediately, again from the $K$-equivariance of $f$. Functoriality in $\pi$ is clear.
	
	Consider transitivity. It is clear that $K \cap Q \supset K_L K_V$. On the other hand, every $g \in K \cap Q \subset K \cap P$ decomposes into $mu$ where $m \in K_M$ and $u \in K_U$; but then $u \in V$, thus $u \in K_V$. We now proceed to decompose $m = gu^{-1} \in K \cap M \cap Q$ using the assumption, proving that $g \in K_L K_V$. Hence $K \cap Q = K_L K_V$.
	
	We claim that $K_V = (K_M)_{M \cap V} K_U$. It suffices to prove $\subset$. Given $v \in K_V$, we deduce from $V = (M \cap V) U$ the decomposition $v = mu$. However, from $K \cap P = K_M K_U$ and $V \subset P$ we see $u \in K_U \subset K_V$, hence $m \in K \cap M \cap V$ as well.
	
	Now $\mathrm{q}_{M \cap V}$ and $\mathrm{q}_V$ both make sense, and $\mathrm{q}_V = \mathrm{q}_{M \cap V} \mathrm{q}_U$ follows immediately from the claim above.
\end{proof}

Unlike \cite[(7.9)]{BK98}, we do not claim $\mathrm{q}$ is equivariant; in fact, we do not have a natural homomorphism from $\mathcal{H}(M \sslash K_M, \tau_{K_U})$ to $\mathcal{H}(G \sslash K, \tau)$ in this generality.

\section{The metaplectic side}\label{sec:Mp-groups}
\subsection{Symplectic groups}\label{sec:Sp}
Let $F$, $\psi$, etc.\ be as in \S\ref{sec:intro}.

Consider a symplectic $F$-vector space $(W, \lrangle{\cdot|\cdot})$ with $n := \frac{1}{2}\dim W \in \Z_{\geq 1}$. Fix once and for all a symplectic basis for $W$, namely
\[ e_1, \ldots, e_n, f_n, \ldots, f_1, \quad \lrangle{e_i | f_j} = \delta_{i, j}, \quad \lrangle{e_i | e_j} = \lrangle{f_i | f_j} = 0, \]
where $\delta_{i, j}$ stands for Kronecker's delta. The Heisenberg group $H(W)$ associated with $(W, \lrangle{\cdot|\cdot})$ is the set $W \times F$ with multiplication
\[ (w_1, t_1) (w_2, t_2) := (w_1 + w_2, t_1 + t_2 + \lrangle{w_1|w_2}). \]
Its center is identified with $F$, via $t \mapsto (0, t) \in H(W)$.

The symplectic group associated with $(W, \lrangle{\cdot|\cdot})$ is denoted by $G := \Sp(W)$. We will occasionally denote it as $\Sp(2n)$ when there is no confusion. To avoid the clumsy notation $\Sp(W, F)$, we will identify $\Sp(W)$ with the group of its $F$-points throughout this article.

The symplectic basis gives rise to the standard maximal torus $T$ (resp.\ Borel subgroup $B^{\rightarrow}$) of $G$. Parabolic subgroups $P \supset B^{\rightarrow}$ are the stabilizers of totally isotropic flags of the form
\begin{align*}
	Fe_1 \oplus \cdots \oplus Fe_{n_1} & \subset Fe_1 \oplus \cdots \oplus Fe_{n_1 + n_2} \subset \cdots \\
	& \subset Fe_1 \oplus \cdots \oplus Fe_{\sum_i n_i}
\end{align*}
with $n^\flat + n_1 + \cdots + n_r = n$ and $n^\flat \geq 0$, $n_i \geq 1$. Accordingly, in the standard Levi decomposition $P = MU$ we have
\[ M \simeq \prod_{i=1}^r \GL(n_i) \times \Sp(W^\flat) \]
where $(W^\flat, \lrangle{\cdot|\cdot})$ is the $2n^\flat$-dimensional symplectic subspace spanned by $e_i, f_i$ with $n_1 + \cdots + n_r < i \leq n$.

Taking $n^\flat=0$ and $n_1 = \cdots = n_n = 1$ yields $B^{\rightarrow}$. Take the standard basis $\epsilon_1, \ldots, \epsilon_n$ for $X^*(T)$. The $B^{\rightarrow}$-positive roots are
\[ \epsilon_i \pm \epsilon_j \quad (1 \leq i < j \leq n), \quad 2\epsilon_i \quad (1 \leq i < n); \]
the $B^{\rightarrow}$-simple roots are
\[ \epsilon_i - \epsilon_{i+1} \quad (1 \leq i < n), \quad 2\epsilon_n. \]

There is another ``reverse-standard'' Borel subgroup $B^{\leftarrow}$ containing $T$, namely the stabilizer of the flag
\begin{equation}\label{eqn:Borel-flag-left}
	Fe_n \subset Fe_n \oplus F_{e_{n-1}} \subset \cdots \subset Fe_n \oplus \cdots \oplus Fe_1.
\end{equation}

The $B^{\leftarrow}$-positive roots are
\[ \epsilon_i \pm \epsilon_j \quad (1 \leq j < i \leq n), \quad 2\epsilon_i \quad (1 \leq i < n); \]
the $B^{\leftarrow}$-simple roots are
\[ \epsilon_{i+1} - \epsilon_i \quad (1 \leq i < n), \quad 2\epsilon_1. \]

The notation concerning $B^{\rightarrow}$ and $B^{\leftarrow}$ is self-explanatory. The parabolic subgroups $P \supset B^{\leftarrow}$ are stabilizers of flags of the form
\begin{equation}\label{eqn:Levi-flag-left}\begin{aligned}
	Fe_n \oplus \cdots \oplus Fe_{n - n_1 + 1} & \subset Fe_n \oplus \cdots Fe_{n - n_1 - n_2 + 1} \subset \cdots \\
	& \subset Fe_n \oplus \cdots \oplus Fe_{n - \sum_i n_i + 1}.
\end{aligned}\end{equation}
The Levi factor of $P$ containing $T$ will be written in the form
\[ M = \Sp(W^\flat) \times \prod_{i=1}^r \GL(n_i), \quad \frac{1}{2}\dim W^\flat = n^\flat := n - \sum_{i=1}^r n_i, \]
for obvious reasons. Note that $B^{\leftarrow}$ and $B^{\rightarrow}$ both contain $T$, but they are not opposite. We write
\[ B^{\rightarrow} = TU^{\rightarrow}, \quad B^{\leftarrow} = TU^{\leftarrow}. \]

\subsection{Metaplectic covering and Weil representation}
By the Stone--von Neumann theorem, $H(W)$ has a unique irreducible smooth representation $\rho_\psi$ of central character $\psi$, up to isomorphism. From this we construct the eightfold metaplectic covering as a topological central extension
\[ 1 \to \bmu_8 \to \Mp(W) \xrightarrow{\pr} \Sp(W) \to 1, \]
or written as $1 \to \bmu_8 \to \tilde{G} \to G(F) \to 1$. We refer to \cite[\S 2]{Li11} for detailed explanations about eightfold covers. The notation $\Mp(2n)$ will occasionally be used, with the convention $\Mp(0) = \bmu_8$.

An advantage of eightfold coverings is that, when a Levi subgroup $M$ decomposes as
\[ M = \prod_{i=1}^r \GL(n_i) \times \Sp(W^\flat), \quad \text{cf.\ \S\ref{sec:Sp}}, \]
we have
\begin{equation}\label{eqn:Levi-splitting}
	\tilde{M} = \prod_{i=1}^r \GL(n_i, F) \times \Mp(W^\flat)
\end{equation}
canonically, where $\Mp(W^\flat)$ is associated with $\psi \circ \lrangle{\cdot|\cdot}: W^\flat \times W^{\flat} \to \CC^\times$; see \textit{loc.\ cit.}

Recall the following key property \cite[II.5 Lemme]{MVW87}: for all $x, y \in \Mp(W)$, we have $xy=yx$ if and only if $\pr(x)\pr(y) = \pr(y)\pr(x)$.

We will be interested in genuine smooth representations of $\Mp(W)$. The most important instance is the Weil representation
\[ \omega_\psi = \omega_\psi^+ \oplus \omega_\psi^- \]
of $\tilde{G}$. The even $\omega_\psi^+$ and odd $\omega_\psi^-$ pieces of $\omega_\psi$ are both irreducible and genuine. All these constructions depend only on $W$ and $\psi \circ \lrangle{\cdot|\cdot}: W \times W \to \CC^{\times}$.

It should be emphasized that, in \cite{Li11} and its sequels leading up to \cite{Li21}, the multiplication in $H(W)$ was $\left(w_1, t_1) (w_2, t_2) = (w_1 + w_2, t_1 + t_2 + \frac{1}{2}\lrangle{w_1|w_2}\right)$ instead. As a result, the $\Mp(W)$ and $\omega_\psi$ considered here are associated with $W$ and $\psi\left( 2\lrangle{\cdot|\cdot} \right)$ from the viewpoint of \cite{Li11}.

Recall that for each root $\alpha$ of $T \subset G$, the corresponding $x_\alpha: \Ga \hookrightarrow G$ lifts canonically, on the level of $F$-points, to a homomorphism $\tilde{x}_\alpha: F \to \tilde{G}$.

In this article, $\omega_\psi$ is realized on $\Schw\left( \bigoplus_{i=1}^n Ff_i\right)$ via the Schrödinger model. For each $\phi$ in this space, $y = \sum_{i=1}^n y_i f_i$ and $t \in F$, we have
\begin{equation}\label{eqn:root-action}
	\begin{aligned}
		\left(\omega_\psi\left( \tilde{x}_{\epsilon_i - \epsilon_j}(t) \right)\phi\right)(y) & = \phi(y + ty_i f_j), \\
		\left(\omega_\psi\left( \tilde{x}_{\epsilon_i + \epsilon_j}(t) \right)\phi\right)(y) & = \psi(2t y_i y_j) \phi(y), \\
		\left(\omega_\psi\left( \tilde{x}_{2\epsilon_i}(t) \right)\phi\right)(y) & = \psi(ty_i^2) \phi(y),
	\end{aligned}
\end{equation}
where $1 \leq i, j \leq n$, $i \neq j$.

Moreover, if $a \in \GL(n, F) \hookrightarrow \Sp(W)$ (the Siegel Levi subgroup), acting on $F^n \simeq \bigoplus_{i=1}^n Fe_i$, then its image in $\Mp(W)$ under \eqref{eqn:Levi-splitting} acts by
\begin{equation}\label{eqn:Siegel-action}
	(\omega_\psi(a)\phi)(y) = |\det a|_F^{\frac{1}{2}} \phi\left( \transp{a} y \right)
\end{equation}
where the transpose $\transp{a}$ acts on $F^n \simeq \bigoplus_{i=1}^n Ff_i$. See \cite[II.6]{MVW87}.

Moreover, there is a canonical $\tilde{w}_{\mathrm{long}} \in \tilde{G}$ representing the longest Weyl element in $N_G(T)$, such that $\omega_\psi(\tilde{w}_{\mathrm{long}})$ acts as the unitary Fourier transform on $F^n \simeq \bigoplus_{i=1}^n Ff_i$; see \eqref{eqn:unitary-Fourier}.

Finally, the subrepresentation $\omega_\psi^\pm$ of $\omega_\psi$ is realized on $\Schw\left( \bigoplus_{i=1}^n Ff_i\right)^\pm$.

\subsection{Building and lattices}
The definitions below follow \cite{GS12, TW18}.

Let $G = \Sp(W)$ as before. Let $\check{\epsilon}_1, \ldots, \check{\epsilon}_n \in X_*(T)$ be the dual basis of $\epsilon_1, \ldots, \epsilon_n \in X^*(T)$. The standard apartment of the Bruhat--Tits building of $G$ is then identified with $X_*(T) \otimes \R = \bigoplus_{i=1}^n \R\check{\epsilon_i} \simeq \R^n$. The standard alcove therein has vertices $z_0, \ldots, z_n$ where
\[ z_0 = (0, \ldots, 0), \; z_1 = \left(\frac{1}{2}, 0, \ldots\right), \; \ldots, \; z_n = \left(\frac{1}{2}, \ldots, \frac{1}{2}\right) \]
with barycenter $\left(\frac{n}{2(n+1)}, \ldots, \frac{1}{2(n+1)} \right)$.

The Coxeter diagram for the Bruhat--Tits building is depicted below.
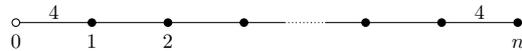
\[ \dynkin[extended, Coxeter, labels={0, 1, 2, ,,, n}, edge length=1cm] C{} \]
Specifically, the simple reflections are labeled as $s_0, \ldots, s_n$. The $s_1, \ldots, s_n$ are the reflections on $X_*(T) \otimes \R$ along the simple roots $\epsilon_1 - \epsilon_2, \ldots, 2\epsilon_n$, whilst $s_0$ is the affine reflection along $1 - 2\epsilon_1$.

Let $\Omega_{\mathrm{aff}} = X_*(T) \rtimes \Omega_0$ be the affine Weyl group. Here $\Omega_{\mathrm{aff}}$ (resp.\ $\Omega_0$) is generated by $s_0, \ldots, s_n$ (resp.\ $s_1, \ldots, s_n$).

For every $0 \leq i \leq n$, let $K_i$ be the stabilizer of the lattice
\[ \mathcal{L}_i := \bigoplus_{j=1}^n \mathfrak{o} e_j \oplus \bigoplus_{j=1}^i (\varpi) f_j \oplus \bigoplus_{j=i+1}^n \mathfrak{o} f_j \]
in $G(F)$. It is the maximal compact subgroup of $G(F)$ corresponding to $z_i$.

In particular, the origin $z_0$ corresponds to the standard hyperspecial subgroup
\[ K_0 := \Stab_{G(F)}\left( \bigoplus_{j=1}^n \mathfrak{o}e_j \oplus \mathfrak{o}f_j \right). \]

The standard Iwahori subgroup of $G(F)$ is thus $I = \bigcap_{j=0}^n K_j$. We have
\[ G(F) = \bigsqcup_{w \in \Omega_{\mathrm{aff}}} IwI, \quad K_0 = \bigsqcup_{w \in \Omega_0} IwI. \]
The length function $\ell$ on $\Omega_{\mathrm{aff}}$ satisfies
\[ |IwI/I| = q^{\ell(w)}, \quad w \in \Omega_{\mathrm{aff}}. \]

Assuming $n \geq 1$, set
\[ J := \bigcap_{j=1}^n K_i = I \sqcup I s_0 I. \]
We will also encounter the subgroups
\begin{align*}
	J_m := \bigcap_{\substack{0 < j \leq n \\ j \neq m}} K_j, \quad 1 \leq m < n.
\end{align*}

\section{Co-invariants of types}\label{sec:coinvariant}
\subsection{Types}
Following \cite[\S 1.4]{TW18}, we consider the following $K_i$-invariant subspaces under $\omega_\psi$, for each $0 \leq i \leq n$ (and with some abuse of notation):
\[ \Schw\left( \frac{\mathcal{L}_0 \cap \bigoplus_{j=1}^n F f_j }{2 \mathcal{L}_i \cap \bigoplus_{j=1}^n F f_j } \right) = \Schw\left( \bigoplus_{j=1}^i \frac{\mathfrak{o} f_j}{(2\varpi) f_j} \oplus \bigoplus_{j=i+1}^n \frac{\mathfrak{o} f_j}{(2) f_j} \right). \]
Denote the corresponding $K_i$-subrepresentation of $\omega_\psi$ as $\tau_i$. When $n=0$, we view $\tau_0$ as the tautological representation of $\bmu_8$ on $\CC$.

Furthermore, $\tau_i = \tau_i^+ \oplus \tau_i^-$ by breaking it into the pieces consisting of even ($+$) and odd ($-$) functions under the Schrödinger model. Note that
\begin{gather*}
	\tau_0^+ = \tau_0, \quad \tau_i \subset \tau_{i+1}, \\
	\dim\tau_i^{\pm} = \frac{1}{2} q^{en} (q^i \pm 1) \quad (e := \mathrm{val}_F(2)).
\end{gather*}

Now consider a parabolic subgroup $P \supset B^{\leftarrow}$ of $G = \Sp(W)$ associated with a flag as in \eqref{eqn:Levi-flag-left}. Take the Levi decomposition $P = MU$ with $M \supset T$, so that $M = \Sp(W^\flat) \times \prod_{i=1}^r \GL(n_i)$ and $\dim W^\flat = 2n^\flat$.

\begin{lemma}\label{prop:cpt-intersection-Levi}
	The intersection $I \cap M(F)$ is the product of the Iwahori subgroups of $\Sp(W^\flat)$ and of $\GL(n_1, F), \ldots, \GL(n_r, F)$, associated with the following data.
	\begin{itemize}
		\item For $\Sp(W^\flat)$, it is the standard Iwahori subgroup $I^\flat$ arising from the symplectic basis
		\[ e_1, \ldots, e_{n^\flat}, f_{n^\flat}, \ldots, f_1. \]
		\item For each $\GL(n_i)$, it is the standard Iwahori subgroup arising from the basis
		\[ e_{n - \sum_{j \leq i} n_j + 1}, \ldots, e_{n - \sum_{j < i} n_j} \]
		of the graded piece of the flag \eqref{eqn:Levi-flag-left}.
	\end{itemize}
	
	The same is true for $J \cap M(F)$ (resp.\ $J_m \cap M(F)$ when $1 \leq m < n^\flat$) if we assume $n^\flat$ is positive. One simply puts the avatar for $J$ (resp.\ $J_m$) on the $\Sp(W^\flat)$ factor, and the standard Iwahori subgroups on each $\GL(n_i, F)$, via the bases above.
\end{lemma}
\begin{proof}
	Although it is possible to give a building-theoretic proof, here we work with the lattices. We shall inspect the condition $g \mathcal{L}_h = \mathcal{L}_h$ defining $K_h$, for various $h$. Write $g \in M(F)$ as $(g^\flat, (g_i)_{i=1}^r)$. In view of the definition of $\mathcal{L}_h$, we may consider the conditions imposed on $g^\flat$ and each $g_i$ separately, as they act on different direct summands of $W$.
	
	First, the condition on $g^\flat \in \Sp(W^\flat)$ reads
	\[ g^\flat \mathcal{L}^\flat_{\min\{h, n^\flat\}} = \mathcal{L}^\flat_{\min\{h, n^\flat\}}, \quad 0 \leq h \leq n, \]
	where $\mathcal{L}^\flat_h$ means the lattice defined with respect to $W^\flat$. When $h$ ranges over $0, \ldots, n$ (for $I$), $1, \ldots, n$ (for $J$), or $\{0, \ldots, n\} \smallsetminus \{0, m\}$ (for $J_m$ when $1 \leq m \leq n^\flat$), these conditions cut out the corresponding subgroup of $\Sp(W^\flat)$.
	
	Secondly, consider $g_i \in \GL(n_i, F)$. Only the indices $j$ lying between $n - \sum_{j \leq i} n_j + 1$ and $n - \sum_{j < i} n_j$ matter in the discussion below; let $\mathbf{J}_i$ be the set of indices in this range.
	
	Note that $g_i$ (as an element of $G(F)$) stabilizes both $\bigoplus_{j \in \mathbf{J}_i} Fe_j$ and $\bigoplus_{j \in \mathbf{J}_i} Ff_j$. Identify $\GL(n_i)$ with the group of linear automorphisms on $\bigoplus_{j \in \mathbf{J}_i} Fe_j$. It follows from $g \mathcal{L}_h = \mathcal{L}_h$ that $g_i$ stabilizes both $\bigoplus_j \mathfrak{o}e_j$ and $\bigoplus_j \mathfrak{o}f_j$, hence $g_i \in \GL(n_i, \mathfrak{o})$.
	
	Therefore, the condition on $g_i$ imposed by $g\mathcal{L}_h = \mathcal{L}_h$ is equivalent to:
	\begin{itemize}
		\item $g_i \in \GL(n_i, \mathfrak{o})$, and
		\item ${}^{\mathrm{t}} g_i^{-1}$, viewed as a linear automorphism of $\bigoplus_{j \in \mathbf{J}_i} Ff_j$, stabilizes $\mathcal{L}_h \cap \bigoplus_{j \in \mathbf{J}_i} Ff_j$.
	\end{itemize}
	Note that the latter condition is redundant if $h \notin \mathbf{J}_i$.
	
	In particular, we see that when $h$ ranges over $0, \ldots, n$ (for $I$), $1, \ldots, n$ (for $J$), or $\{0, \ldots, n\} \smallsetminus \{0, m\}$ (for $J_m$ when $1 \leq m \leq n^\flat$), these conditions cut out the standard Iwahori subgroup of $\GL(n_i, F)$, as asserted.
\end{proof}

\begin{remark}\label{rem:K-cap-M}
	Suppose that $0 \leq i \leq n^\flat$, and let $K_i^\flat \subset \Sp(W^\flat)$ be the avatar of $K_i \subset \Sp(W)$. The arguments above actually showed
	\[ K_i^\flat \times \prod_{j=1}^r \GL(n_j, \mathfrak{o}) = K_i \cap M(F). \]
	On the level of coverings, by \eqref{eqn:Levi-splitting} we also have
	\[ \widetilde{K_i^\flat} \times \prod_{j=1}^r \GL(n_j, \mathfrak{o}) = \widetilde{K_i} \cap \tilde{M}. \]
\end{remark}

Next, for all $0 \leq i \leq n^\flat$ we define the avatars of $\tau_i$ and $\tau_i^\pm$ for $\tilde{M}$ as the representations
\begin{equation*}
	\tau_{i, \tilde{M}} := \tau_i^\flat \otimes \bigotimes_{i=1}^r \mathbf{1}_i	, \quad
	\tau_{i, \tilde{M}}^{\pm} := \tau_i^{\flat, \pm} \otimes \bigotimes_{i=1}^r \mathbf{1}_i
\end{equation*}
of $\widetilde{K_i^\flat} \times \prod_{j=1}^n \GL(n_j, \mathfrak{o}) = \widetilde{K_1} \cap \tilde{M}$, where $\tau_i^{\flat}$ is the avatar of $\tau_i$ for $W^\flat$, and $\mathbf{1}_i$ is the trivial representation of $\GL(n_i, F)$. Note that $\tau_{i, \tilde{M}}^\pm$ is realized on the space
\[ \Schw\left( \bigoplus_{j=1}^i \frac{\mathfrak{o} f_j}{(2\varpi) f_j} \oplus \bigoplus_{j=i+1}^{n^\flat} \frac{\mathfrak{o} f_j}{(2) f_j} \right)^\pm. \]

Denote by $\omega_\psi^\flat$ the Weil representation for $\Mp(W^\flat)$, realized on $\Schw(\bigoplus_{j=1}^{n^\flat} Ff_j)$.

For each $0 \leq i \leq n^\flat$, restriction of functions gives rise to
\begin{equation}\label{eqn:res-M}
	\begin{tikzcd}[row sep=tiny]
		\Schw\left( \bigoplus_{j=1}^i \frac{\mathfrak{o} f_j}{(2\varpi)f_j} \oplus \bigoplus_{j=i+1}^n \frac{\mathfrak{o} f_j}{(2) f_j} \right) \arrow[r, "{\mathrm{res}_{i, \tilde{M}}}"] & \Schw\left( \bigoplus_{j=1}^i \frac{\mathfrak{o} f_j}{(2\varpi) f_j} \oplus \bigoplus_{j=i+1}^{n^\flat} \frac{\mathfrak{o} f_j}{(2)f_j} \right) \\
		\phi \arrow[mapsto, r] & \phi(\cdot, \underbracket{0, \ldots, 0}_{n - n^\flat}).
	\end{tikzcd}
\end{equation}

Note that $\widetilde{K_i}$ (resp.\ $\widetilde{K_i^\flat}$) acts on the left (resp.\ right) hand side through $\omega_\psi$ (resp.\ $\omega_\psi^\flat$). We inflate the $\widetilde{K_i^\flat}$-action to $\widetilde{K_i} \cap \tilde{M}$ by letting each $\GL(n_j, \mathfrak{o})$ act trivially.

\begin{lemma}\label{prop:res-equivariance}
	The map $\mathrm{res}_{i, \tilde{M}}$ is surjective and preserves parity. Moreover, it is $\widetilde{K_i} \cap \tilde{M}$-equivariant.
\end{lemma}
\begin{proof}
	The map is clearly surjective and commutes with $\phi(y) \mapsto \phi(-y)$, hence parity-preserving. To prove equivariance, it suffices to notice that the restriction map
	\[\begin{tikzcd}[row sep=tiny]
		\Schw\left( \bigoplus_{j=1}^n Ff_j \right) \arrow[r] & \Schw\left( \bigoplus_{j=1}^{n^\flat} Ff_j \right) \\
		\phi \arrow[mapsto, r] & \phi(\cdot, 0, \ldots, 0)
	\end{tikzcd}\]
	is $\Mp(W^\flat) \times \prod_{j=1}^r \GL(n_j, \mathfrak{o})$-equivariant, where we let $\Mp(W^\flat)$ (resp.\ $\prod_{j=1}^r \GL(n_j, \mathfrak{o})$) act by $\omega_\psi^\flat$ (resp.\ trivially) on the right hand side. Indeed, this follows immediately from the explicit formulas for $\omega_\psi$ and $\omega_\psi^\flat$, since the $\GL(n_j, \mathfrak{o})$-action on $\phi$ does not affect $\phi(\cdot, 0, \ldots, 0)$; see \eqref{eqn:root-action} and \eqref{eqn:Siegel-action}.
\end{proof}

\subsection{Calculation of co-invariants}
Consider
\begin{itemize}
	\item a parabolic subgroup $P = MU \supset B^{\leftarrow}$ of $G$,
	\item a compact open subgroup $K$ of $G(F)$.
\end{itemize}

Write $M$ in the form $\Sp(W^\flat) \times \prod_{i=1}^r \GL(n_i)$. Put
\begin{equation}\label{eqn:KM}
	K_M := K \cap M(F), \quad K_U := K \cap U(F), \quad K_{\overline{U}} := K \cap \overline{U}(F).
\end{equation}
We assume that
\begin{equation}\label{eqn:coinvariant-K-assumption}
	K \cap P(F) = K_M K_U.
\end{equation}

Under the assumption above, the decomposition $\tilde{P} = \tilde{M}U(F)$ induces $\tilde{K} \cap \tilde{P} = \widetilde{K_M} K_U$, canonically.

The general formalism reviewed in \S\ref{sec:Hecke-coinvariant} furnishes the functor $(\cdot)_{K_U}: \tilde{K}\dcate{Mod} \to \widetilde{K_M}\dcate{Mod}$. No normalization is required here.

The following result describes $(\tau_i^{\pm})_{K_U}$. The argument is parallel to the well-known calculation of $r_{\tilde{P}}(\omega_\psi)$. Let $I_U := I \cap U(F)$.

\begin{proposition}\label{prop:coinvariant-type}
	Let $0 \leq i \leq n^\flat$. Suppose that $K$ satisfies \eqref{eqn:coinvariant-K-assumption}, $K \subset K_i$ and $K_U = I_U$. There is a unique isomorphism of $\widetilde{K_M}$-representations $\tau_{i, \tilde{M}}^{\pm} \rightiso (\tau_i^{\pm})_{K_U}$ making the following diagram commutative.
	\[\begin{tikzcd}
		\tau_i^{\pm} \arrow[twoheadrightarrow, d] \arrow[twoheadrightarrow, r, "{\mathrm{res}_{i, \tilde{M}}}"] & \tau_{i, \tilde{M}}^{\pm} \arrow[ld, "\sim"' sloped] \\
		(\tau_i^{\pm})_{K_U} &
	\end{tikzcd}\]
	We are abusing notation here: one should really replace $\tau_i^{\pm}$, etc.\ by the underlying vector spaces, as not all the maps are literally $\widetilde{K_M}$-equivariant.
\end{proposition}
\begin{proof}
	For the ease of notation, consider the case without superscript $\pm$ and omit $f_j$. For all
	\[ \phi \in \Schw\left( \bigoplus_{j=1}^i \frac{\mathfrak{o}}{(2\varpi)} \oplus \bigoplus_{j=i+1}^n \frac{\mathfrak{o}}{(2)} \right), \quad
	y = (y_1, \ldots, y_i, \ldots, \underbracket{y_{n^\flat + 1}, \ldots, y_n}_{\not\equiv (0, \ldots, 0) \bmod 2} ) \in \mathfrak{o}^n , \]
	there exist $\ell > n^\flat$ and $t \in \mathfrak{o}$ such that
	\[ \omega_\psi(\tilde{x}_{2\epsilon_\ell}(t)\phi)(y) = \underbracket{\psi(ty_\ell^2)}_{\neq 1} \phi(y) \]
	by our assumptions on $\psi$. Since $x_{2\epsilon_\ell}(t) \in I_U = K_U$, the Dirac functions supported on such $y$ are mapped to zero in $(\tau_i)_{K_U}$. As these functions span the kernel of $\mathrm{res}_{i, \tilde{M}}$, this yields the arrow in the assertion, and it is surjective as $\tau_i \to (\tau_i)_{K_U}$ is. The $\widetilde{K_M}$-equivariance follows from Lemma \ref{prop:res-equivariance}.
	
	It remains to prove the injectivity of $\tau_{i, \tilde{M}} \to (\tau_i)_{K_U}$. Consider $\phi$ as above, and
	\[ y = (y_1, \ldots, y_i, \ldots, \underbracket{y_{n^\flat + 1}, \ldots, y_n}_{\equiv (0, \ldots, 0) \bmod 2} ) \in \mathfrak{o}^n. \]
	
	Recalling that $P \supset B^{\leftarrow}$, we consider the action of root subgroups in 4 cases.
	\begin{enumerate}
		\item If $n^\flat < j < \ell$ and $t \in \mathfrak{o}$, then $ty_\ell \in 2\mathfrak{o}$ and
		\[ \omega_\psi(\tilde{x}_{\epsilon_\ell - \epsilon_j}(t)\phi)(y) = \phi(y + ty_\ell f_j) = \phi(y). \]
		
		\item If $j \leq n^\flat < \ell$ and $t \in (\varpi)$, then $ty_\ell \in 2\varpi\mathfrak{o}$ and
		\[ \omega_\psi(\tilde{x}_{\epsilon_\ell - \epsilon_j}(t)\phi)(y) = \phi(y + ty_\ell f_j) = \phi(y). \]
		
		\item If $j < \ell$, at least one of them is $> n^\flat$, and $t \in \mathfrak{o}$, then $2t y_\ell y_j \in 4\mathfrak{o}$ and
		\[ \omega_\psi(\tilde{x}_{\epsilon_\ell + \epsilon_j}(t)\phi)(y) = \psi(2t y_\ell y_j)\phi(y) = \phi(y). \]
		\item If $\ell > n^\flat$ and $t \in \mathfrak{o}$, then $ty_\ell^2 \in 4\mathfrak{o}$ and
		\[ \omega_\psi(\tilde{x}_{2\epsilon_\ell}(t)\phi)(y) = \psi(ty_\ell^2)\phi(y) = \phi(y). \]
	\end{enumerate}
	
	From the structure for $I_U$ (see \cite[(7.2.6)]{BT72}), we see that the elements above generate $I_U$ as $t$ and $j, \ell$ vary. Specifically, in the case 1 we shall take $t \in (\varpi)$, since it concerns $B^{\rightarrow}$-negative roots.
	
	It follows that $\mathrm{res}_{i, \tilde{M}}\left(\omega_\psi(u)\phi - \phi\right) = 0$ where $u \in K_U = I_U$ and $\phi$ is as above. This implies the injectivity of $\tau_{i, \tilde{M}} \to (\tau_i)_{K_U}$.
\end{proof}

\begin{definition}\label{def:P1}
	Take the parabolic subgroup $P^1 = M^1 U^1 \supset B^{\leftarrow}$ with Levi factor
	\[ M^1 = \Sp(W^1) \times \GL(1)^{n-1} \supset T, \]
	where $W^1 := Fe_1 \oplus Ff_1$. Let $T^1$ be the $\GL(1)^{n-1}$ factor in $M^1$, so that $T = \GL(1) \times T^1$.
\end{definition}

Observe that $B^{\leftarrow}$ (resp.\ $P^1$) is in good position relative to $K_0$ (resp.\ $K_1$), or more precisely to the corresponding vertex $z_0$ (resp.\ $z_1$). Hence there are decompositions
\begin{align*}
	K_0 \cap B^{\leftarrow}(F) & = (K_0 \cap T(F)) (K_0 \cap U^{\leftarrow}(\mathfrak{o})) \\
	& = T(\mathfrak{o}) U^{\leftarrow}(\mathfrak{o}), \\
	K_1 \cap P^1(F) & = (K_1 \cap M^1(F)) (K_1 \cap U^1(F))
\end{align*}

The following two variants of Proposition \ref{prop:coinvariant-type} will be needed in \S\ref{sec:spherical-parts}. First, identify the preimage of $K_0 \cap T(F) = T(\mathfrak{o})$ in $\tilde{T}$ with $T(\mathfrak{o}) \times \bmu_8$ using \eqref{eqn:Levi-splitting} (with $M = T$).

\begin{proposition}\label{prop:coinvariant-U-leftarrow}
	Let $\tau_{0, U^{\leftarrow}(\mathfrak{o})}$ be the co-invariants of $\tau_0$ with respect to $U^{\leftarrow}(\mathfrak{o})$. Let $\Xi$ denote the projection $\widetilde{T(\mathfrak{o})} \to \bmu_8$. Then
	\[ \tau_{0, U^{\leftarrow}(\mathfrak{o})} \simeq \Xi \]
	as representations of $\widetilde{T(\mathfrak{o})}$: it is induced by mapping $\phi \in V_{\tau_0}$ to $\phi(0, \ldots, 0)$.
\end{proposition}
\begin{proof}
	Re-iterate the proof for Proposition \ref{prop:coinvariant-type} with $U^{\leftarrow}(\mathfrak{o})$. Although $I_{U^{\leftarrow}} \subsetneq U^{\leftarrow}(\mathfrak{o})$, the first part of the proof (for obtaining the arrow $\Xi \twoheadrightarrow \tau_{0, U^{\leftarrow}(\mathfrak{o})}$) is exactly the same: it suffices to note that $U^{\leftarrow}(\mathfrak{o})$ contains $x_{2\epsilon_\ell}(\mathfrak{o})$ for all $0 \leq \ell \leq n$.

	As to the second part of the proof (for proving surjectivity), one still has to see how the root subgroups inside $U^{\leftarrow}(\mathfrak{o})$ act. This amounts to taking $n^\flat = 0$ in the earlier analysis. Below are the relevant cases, where we assume $t \in \mathfrak{o}$:
	\[\begin{array}{|c|c|c|}\hline
		\tilde{x}_{\epsilon_\ell - \epsilon_j}(t) & \tilde{x}_{\epsilon_\ell + \epsilon_j}(t) & \tilde{x}_{2\epsilon_\ell}(t) \\
		j < \ell & \ell \neq j & \forall \ell \\ \hline
	\end{array}\]
	
	As in the proof for Proposition \ref{prop:coinvariant-type}, if $u$ is any one of the elements above, then
	\[ (\omega_\psi(u)\phi)(0, \ldots, 0) = \phi(0, \ldots, 0). \]
	The other arguments carry over verbatim.
\end{proof}

Next, use \eqref{eqn:Levi-splitting} (with $M = M^1$) and Remark \ref{rem:K-cap-M} to identify the preimage of $K_1 \cap M^1(F)$ in $\tilde{M}^1$ with $\widetilde{K_1^{\flat}} \times (\mathfrak{o}^\times)^{n-1}$ where $K_1^{\flat} := K_1 \cap \Sp(W^1)$ is the avatar of $K_1$ for $\Sp(W^1)$. Let $\tau_1^{\flat, -}$ denote the avatar of $\tau_1^-$ on $\Mp(W^1)$.

\begin{proposition}\label{prop:coinvariant-U1}
	Consider the representation $\tau_{1, -}^{\flat} \otimes \mathbf{1}$ of $\widetilde{K_1^{\flat}} \times (\mathfrak{o}^\times)^{n-1}$. Then we have
	\[ (\tau_1^-)_{K_1 \cap U^1(F)} \rightiso \tau_1^{\flat, -} \otimes \mathbf{1} \]
	induced by $\phi \mapsto \phi(\cdot, 0, \ldots, 0) \in V_{\tau_1^{\flat, -}}$.
\end{proposition}
\begin{proof}
	Again, we re-iterate the proof for Proposition \ref{prop:coinvariant-type} with $K_1 \cap U^1(F)$. Recall that $K_1$ is the stabilizer of the vertex $z_1 = (\frac{1}{2}, 0, \ldots, 0)$ in the standard apartment.
	
	The elements $\tilde{x}_{2\epsilon_\ell}(t)$ with $t \in \mathfrak{o}$ and $\ell > 1$, used in the first part of the proof, are still in $K_1 \cap U^1(F)$: this we can see easily from $2\epsilon_\ell(z_1) = 0$ when $\ell > 1$. In fact, by Bruhat--Tits theory \cite[\S 6.2, \S 6.4]{BT72}, for every root $\alpha \in X^*(T)$ and vertex $z$ in the building, we have
	\[ x_\alpha(t) \in \Stab_{G(F)}(z) \iff \mathrm{val}_F(t) + \alpha(z) \geq 0. \]
	
	Let us turn to the elements $\tilde{x}_{\epsilon_\ell \pm \epsilon_j}(t)$ and $\tilde{x}_{2\epsilon_\ell}(t)$ analyzed in the second part of the proof, now with $n^\flat = i = 1$.
	\begin{enumerate}
		\item If $1 < j < \ell$, then $x_{\epsilon_\ell - \epsilon_j}(t) \in K_1 \cap U^1(F) \iff t \in \mathfrak{o}$ since $(\epsilon_\ell - \epsilon_j)(z_1) = 0$.
		
		\item If $\ell > 1$, then $x_{\epsilon_\ell - \epsilon_1}(t) \in K_1 \cap U^1(F) \iff t \in (\varpi)$ since $(\epsilon_\ell - \epsilon_1)(z_1) = -\frac{1}{2}$.
		
		\item If $j < \ell$, then $x_{\epsilon_\ell + \epsilon_j}(t) \in K_1 \cap U^1(F) \iff t \in \mathfrak{o}$ since $(\epsilon_\ell + \epsilon_j)(z_1) \in \{0, \frac{1}{2}\}$.
		
		\item If $\ell > 1$, then $x_{2\epsilon_\ell}(t) \in K_1 \cap U^1(F) \iff t \in \mathfrak{o}$ as seen above.
	\end{enumerate}
	
	The conditions are exactly those in the proof for Proposition \ref{prop:coinvariant-type}. Hence the arguments carry over.
\end{proof}

\subsection{Some Iwahori decompositions}
Let $P = MU$ be a parabolic subgroup of $G = \Sp(W)$ with $P \supset B^{\leftarrow}$ and $M = \Sp(W^\flat) \times \prod_{i=1}^r \GL(n_i)$.

For every proper subset $\Theta$ of the set $\Delta_{\mathrm{aff}}$ of simple affine reflections $\{s_0, \ldots, s_n\}$, we have the compact open subgroup
\[ K_\Theta := \bigsqcup_{t \in \lrangle{\Theta}} I\dot{t}I, \]
where $\dot{t} \in N_{G(F)}(T(F))$ is any representative of $t$, unique only modulo $T(\mathfrak{o})$. For example, $I = K_{\emptyset}$. Note that the reflections $s_1, \ldots, s_n$ are simple relative to $B^{\rightarrow}$, not $B^{\leftarrow}$.

On the other hand, denote by $\Phi$ (resp.\ $\Phi_M$) the set of roots of $G$ (resp.\ $M$) relative to $T$. Let $\Phi^+$ be the set of $B^{\leftarrow}$-positive roots in $\Phi$, and similarly for $\Phi_M^+$.

\begin{proposition}\label{prop:Iwahori-decomp}
	Let $\Theta$ be a proper subset of $\Delta_{\mathrm{aff}}$ such that $\dot{t} \in M(F)$ for all $t \in \Theta$. Let $K := K_\Theta$ and adopt the notation $K_M, K_U, K_{\overline{U}}$ from \eqref{eqn:KM}. Then:
	\begin{enumerate}[(i)]
		\item $K = K_{\overline{U}} K_M K_U$, thus the multiplication map $K_{\overline{U}} \times K_M \times K_U \to K$ is a homeomorphism and \eqref{eqn:coinvariant-K-assumption} holds;
		\item $K_U = I_U$ and $K_{\overline{U}} = I_{\overline{U}}$.
	\end{enumerate}
\end{proposition}
\begin{proof}
	We will utilize Bruhat--Tits theory \cite{BT72} below, which is greatly simplified by the fact that $G$ is simply connected.
	
	Let $\Omega$ be the facet of the standard alcove corresponding to $\Theta$, in the manner that the whole alcove (resp.\ the vertex $z_0$) corresponds to $\Theta = \emptyset$ (resp.\ $\Theta = \{s_1, \ldots, s_n\}$). Consider the quasi-concave function
	\[ f = f_{\Omega}: \Phi \to \R \]
	determined by $\Omega$. Then $K$ is the parahoric subgroup of $G(F)$ associated with $f$.
	
	For (i), we begin with the decomposition \cite[(6.4.9)]{BT72} associated with $f$. The $N_f$ in \textit{loc.\ cit.} can be absorbed into $M(F)$ via its description in \cite[(7.1.3)]{BT72} and the assumption on $\Theta$. Note that one does not have to assume $P$ standard (i.e.\ containing $B^{\rightarrow}$) in these results.
	
	For (ii), we use the description of $U_f^\pm$ in \cite[(6.4.9)]{BT72}. To show $K_U = I_U$, it suffices to identify $U_{f, \alpha}$ with its counterpart in $I$, for all $\alpha \in \Phi^+ \smallsetminus \Phi_M^+$. Apply the description of $f$ at the end of \cite[(7.2.6)]{BT72} with data $(x, E)$, where $x$ is a vertex of the facet $\Omega$, and $(x, E)$ determines $\Omega$. This is to be compared with $(x, D)$, where $D$ is a vectorial chamber such that $(x, D)$ determines the standard chamber associated with $I$. In this comparison, only the values $f(\alpha)$ with $\alpha|_{A_M} = 1$ (i.e.\ $\alpha \in \Phi_M$) can make difference. This yields the required identity. The case of $K_{\overline{U}} = I_{\overline{U}}$ is similar.
\end{proof}

\begin{proposition}\label{prop:IJ-Iwahori-decomp}
	The premises in Proposition \ref{prop:Iwahori-decomp} are fulfilled when
	\begin{itemize}
		\item $K = I$,
		\item $K = J$ (if $n^\flat \geq 1$), or
		\item $K = J_m$ (if $1 \leq m \leq n^\flat$) 
	\end{itemize}
	
	In particular, Proposition \ref{prop:coinvariant-type} can be applied to these groups $K$, with the proviso that $i \neq 0$ when $K=J$, and $i \notin \{0, m\}$ when $K=J_m$. 
\end{proposition}
\begin{proof}
	Consider the corresponding subset $\Theta \subset \Delta_{\mathrm{aff}}$. The case $K=I$ is the simplest one: it corresponds to $\Theta = \emptyset$.
	
	The case $K=J$ corresponds to $\Theta = \{s_0\}$. The element $s_0$ being the reflection along $2\epsilon_1 = 1$, it has a representative $\dot{s}_0$ in $M(F)$ as $n^\flat \geq 1$.
	
	The case $K=J_m$ corresponds to $\Theta = \{s_0, s_m\}$, and $\dot{s}_m$ can be taken in $M(F)$ as $m \leq n^\flat$.
	
	For the final assertion, observe that the provisos serve to ensure that $K \subset K_i$.
\end{proof}

\section{The Takeda--Wood isomorphism}\label{sec:TW}
\subsection{Hecke algebras: the metaplectic case}\label{sec:Hecke-Mp}
The following is a recapitulation of \cite[\S 2.1 and \S 2.3]{TW18}. Consider a symplectic $F$-vector space $W$ of dimension $2n$ as in \S\ref{sec:Sp}. Use the Haar measures on $\tilde{G} = \Mp(W)$ and $G(F) = \Sp(W)$ with $\mes(\tilde{I}) = \mes(I) = 1$. Define
\begin{equation*}
	H_\psi^+ := \mathcal{H}(\tilde{G} \sslash \tilde{I}, \tau_0).
\end{equation*}

For all $0 \leq i \leq n$, let $I_i = I \sqcup I \dot{s}_i I$ where $\dot{s}_i$ is any representative of $s_i$. Take any preimage $\tilde{s}_i$ of $\dot{s}_i \in G(F)$, and let $T_i \in H_\psi^+$ be the element characterized by $\Supp(T_i) = \tilde{I} \tilde{s}_i \tilde{I}$ and that, as an element of the subalgebra
\[ H_{\psi, i}^+ := \mathcal{H}(\tilde{I}_i \sslash \tilde{I}, \tau_0) \simeq \End_{\tilde{I}_i}\left(\Ind^{\tilde{I}_i}_{\tilde{I}} \check{\tau}_0 \right), \]
it is normalized as follows. Let
\[ \tau_{0, i} := \begin{cases}
	\tau_0, & i \neq 0 \\
	\tau_1^+, & i = 0,
\end{cases}\]
which is the subspace of $\omega_\psi$ generated $\tau_0$ acted upon by $\tilde{I}_i$. In \textit{loc.\ cit.} it is shown that
\[ \Ind^{\tilde{I}_i}_{\tilde{I}} \check{\tau}_0 = \check{\tau}_{0, i} \oplus \text{another non-isomorphic irreducible}. \]
Then $T_i$ acts as $(q, -1)$ (resp.\ $(1, -1)$) if $i \neq 0$ (resp.\ $i=0$).

Now consider $w \in \Omega_{\mathrm{aff}}$ with reduced expression $w = s_{i_1} \cdots s_{i_\ell}$. Define
\begin{equation*}
	T_w := T_{i_1} \cdots T_{i_\ell} \in H_\psi^+ .
\end{equation*}
It is shown in \cite[p.1110]{TW18} that $T_w$ is independent of reduced expressions of $w$, and $\Supp(T_w) = \tilde{I}\tilde{w}\tilde{I}$ where $\tilde{w} \in \tilde{G}$ is any representative of $w$. Furthermore, $H_\psi^+$ is generated by $T_0, \ldots, T_n$ with braid relations of affine type $\mathrm{C}_n$, together with the quadratic relations
\begin{align*}
	(T_0 + 1)(T_0 - 1) & = 0, \\
	(T_i + 1)(T_i - q) & = 0, \quad 1 \leq i \leq n.
\end{align*}

The extreme case $n=0$ is incorporated by setting $H^+_{\psi} = \CC$. We denote the length function on $\Omega_{\mathrm{aff}}$ by $\ell$.

Now we turn to the structure of
\begin{equation*}
	H_\psi^- := \mathcal{H}(\tilde{G} \sslash \tilde{J}, \tau_1^-).
\end{equation*}
Recall that $J = I \sqcup I s_0 I$ and here the Haar measures satisfy $\mes(\tilde{J}) = \mes(J) = 1$.

By \cite[Lemma 2.3 + Theorem 2.4]{TW18}, $H_\psi^-$ is supported on
\[ \tilde{J} \Omega'_{\mathrm{aff}} \tilde{J}, \]
where $\Omega'_{\mathrm{aff}} := \lrangle{s'_1, \ldots, s'_n}$ with
\[ s'_i := \begin{cases}
	s_i, & i > 1 \\
	s_1 s_0 s_1, & i = 1.
\end{cases}\]

Note that $s_1 s_0 s_1$ is the affine reflection with respect to $-2\epsilon_2 + 1$. This gives a Coxeter group of affine type $\mathrm{C}_{n-1}$ acting on the affine $(n-1)$-space $\epsilon_1 = \frac{1}{2}$.

For each $1 \leq i \leq n$, we construct $T'_i \in H_\psi^-$ with $\Supp T'_i = \tilde{J} \tilde{s}'_i \tilde{J}$ where $\tilde{s}'_i$ is any representative of $s'_i$ in $\tilde{G}$, as follows. As in \textit{loc.\ cit.}, we put
\[ \widetilde{J_i} := \lrangle{\tilde{J}, \tilde{J} \tilde{s}'_i \tilde{J}}; \]
set $H_{\psi, i}^- := \mathcal{H}(\widetilde{J_i} \sslash \tilde{J}, \tau_1^-) \simeq \End_{\widetilde{J_i}}\left( \Ind_{\tilde{J}}^{\widetilde{J_i}}(\check{\tau}_1^-) \right)$, which is a subalgebra of $H_\psi^-$. It is shown to be two-dimensional. Set
\[ \tau_{1, i} := \begin{cases}
	\tau_1^-, & i \geq 2 \\
	\tau_2^-, & i = 1;
\end{cases}\]
this is also the subspace generated by $\tau_1^-$ acted upon by $\widetilde{J_i}$. We have
\[ \Ind_{\tilde{J}}^{\widetilde{J_i}}(\check{\tau}_1^-) = \check{\tau}_{1, i} \oplus \text{another non-isomorphic irreducible}. \]
Now $T'_i$ is normalized to act as $(q, -1)$ (resp.\ $(q^2, -1)$) if $i \neq 1$ (resp.\ $i=1$).

From this, one defines $T'_w \in H_\psi^-$ for every $w \in \Omega'_{\mathrm{aff}}$ using reduced expressions. Again, it is shown in \cite[p.1115]{TW18} that this is independent of reduced expressions of $w$, and $\Supp(T'_w) = \tilde{J}\tilde{w}\tilde{J}$ where $\tilde{w} \in \tilde{G}$ is any representative of $w$. Furthermore, $H_\psi^-$ is generated by $T'_1, \ldots, T'_n$ with braid relations of affine type $\mathrm{C}_{n-1}$, together with the quadratic relations
\begin{align*}
	(T'_1 + 1)(T'_1 - q^2) & = 0, \\
	(T'_i + 1)(T'_i - q) & = 0, \quad 2 \leq i \leq n.
\end{align*}

We also record the fact from \textit{loc.\ cit.} that
\begin{equation*}
	(J\dot{w}J: J) = q^{\ell'(w)}, \quad w \in \Omega'_{\mathrm{aff}}
\end{equation*}
where $\dot{w}$ is any representative of $w$ and $\ell': \Omega'_{\mathrm{aff}} \to \Z_{\geq 0}$ is the weighted length function with $\ell'(s'_1) = 3$ and $\ell'(s_i) = 1$ for $i \geq 2$.

\subsection{Hecke algebras: the orthogonal side}\label{sec:Hecke-O}
The following is a recapitulation of \cite[\S 2.2 and \S 2.4]{TW18}.

\begin{definition}\label{def:V}
	Denote by $V^\pm$ the quadratic $F$-vector spaces of dimension $2n+1$, discriminant $1$ and Hasse invariant equal to $\pm 1$. Define
	\begin{equation*}
		G^{\pm} := \SO(V^\pm), \quad I^\pm := \text{the standard Iwahori subgroup of}\; G^\pm(F).
	\end{equation*}

	Thus we have the Iwahori--Hecke algebras
	\begin{equation*}
		H^\pm := \mathcal{H}(G^\pm \sslash I^\pm, \mathbf{1}).
	\end{equation*}
\end{definition}

To explicate the structure of $H^\pm$, let us begin with the $+$ case. Denote by $T^+$ the standard split maximal torus of $G^+$. Write $X^*(T^+) = \bigoplus_{i=1}^n \Z\epsilon_i$ in the usual way, so that the standard positive roots are $\epsilon_i - \epsilon_j$ and $\epsilon_i$ where $1 \leq i < j \leq n$. The Coxeter diagram for the Bruhat--Tits building of $G^+$ is:
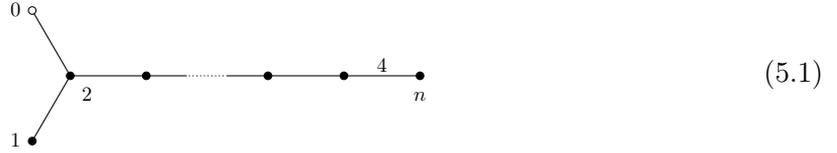
\begin{equation}\label{eqn:Gplus-Dynkin}
	\dynkin[extended, Coxeter, labels={0, 1, 2, ,,, n}, edge length=1cm] B{}
\end{equation}

Accordingly, the simple affine reflections $s^\dagger_0, s_1, \ldots, s_n$ generate the affine Weyl group $\Omega^+_{\mathrm{aff}}$. The $s^\dagger_0$ is associated with the affine root $1 - \epsilon_1 - \epsilon_2$, and $s_1, \ldots, s_n$ are the linear ones. The standard alcove in the standard apartment $X_*(T) \otimes \R$ is given by
\[ 1 - \epsilon_2 > \epsilon_1 > \cdots > \epsilon_n > 0. \]

Following \cite{GS12, TW18}, we define $s_0$ to be the affine reflection with respect to $2\epsilon_1 = 1$, so that
\begin{gather*}
	s^\dagger_0 = s_0 s_1 s_0, \\
	\Omega^+ := \lrangle{s_0, s_1, \ldots, s_n} \supset \Omega^+_{\mathrm{aff}} \quad \text{is the extended affine Weyl group for}\; G^+.
\end{gather*}
Extend the length function on $\Omega^+_{\mathrm{aff}}$ to $\ell_0: \Omega^+ \to \Z_{\geq 0}$ by
\[ \ell_0(s_i) = \begin{cases}
	1, & i \neq 0 \\
	0, & i = 0.
\end{cases}\]

For each $w \in \Omega^+$, we fix an arbitrary representative $\dot{w}$. Then we have
\[ (I^+ \dot{w} I^+ : I^+) = q^{\ell_0(w)}. \]

Summing up, $(\Omega^+, \{s_0, \ldots, s_n\})$ is isomorphic to the affine Weyl group $\Omega_{\mathrm{aff}}$ of $G$, by matching the reflections $s_0, \ldots, s_n$. Their actions on $\R^n$ are matched as well, although the length functions differ.

Define the following elements of $H^+$:
\begin{align*}
	T^+_w & := \mathbf{1}_{I^+ \dot{w} I^+}, \quad w \in \Omega^+ , \\
	T^+_i & := T^+_{s_i}, \quad 0 \leq i \leq n.
\end{align*}

By \textit{loc.\ cit.}, the  algebra $H^+$ is generated by $T^+_0, \ldots, T^+_n$ with braid relations of affine type $\mathrm{C}_n$, together with quadratic relations
\begin{align*}
	(T^+_0 + 1)(T^+_0 - 1) & = 0, \\
	(T^+_i + 1)(T^+_i - q) & = 0, \quad 1 \leq i \leq n.
\end{align*}

The extreme case $n=0$ is incorporated by setting $H^+ = \CC$.

We now turn to the case of $G^-$. The Coxeter diagram is
\begin{equation}\label{eqn:Gminus-Dynkin}
	\dynkin[extended, Coxeter, edge length=1cm] C{} \quad n\;\text{vertices.}
\end{equation}

Let $\Omega^- := \Omega_{\mathrm{aff}}^-$ be its affine Weyl group, generated by the simple affine reflections labeled by $r_1, \ldots, r_n$ corresponding to the vertices. Define $\ell_2: \Omega^- \to \Z_{\geq 0}$ as the weighted length function with
\[ \ell_2(r_i) := \begin{cases}
	2, & i = 1 \\
	1, & 1 < i \leq n.
\end{cases}\]

For each $w \in \Omega^-$, we fix an arbitrary representative $\dot{w}$. It is shown in \cite[\S 13]{GS12} that
\[ (I^- \dot{w} I^- : I^-) = q^{\ell_2(w)}. \]

As in the $+$ case, $\Omega^-$ is isomorphic to $\Omega'_{\mathrm{aff}}$ by sending $r_i$ to $s'_i$, for $i = 1, \ldots, n$.

Define the following elements of $H^-$:
\begin{align*}
	T_w^- & := \mathbf{1}_{I^- \dot{w} I^-}, \quad w \in \Omega^- , \\
	T_i^- & := T_{r_i}^-, \quad 1 \leq i \leq n.
\end{align*}

By \cite[\S 2.3]{TW18}, the algebra $H^-$ is generated by $T^-_1, \ldots, T^-_n$ with braid relations of affine type $\mathrm{C}_{n-1}$, together with quadratic relations
\begin{align*}
	(T^-_1 + 1)(T^-_1 - q^2) & = 0, \\
	(T^-_i + 1)(T^-_i - q) & = 0, \quad 2 \leq i \leq n.
\end{align*}

\subsection{The isomorphism of Takeda--Wood}\label{sec:TW-isom}
The main result of \cite{TW18} is recorded below.

\begin{theorem}\label{prop:TW-isom}
	There are isomorphisms of Hilbert algebras
	\[ \mathrm{TW}: H^{\pm} \rightiso H_\psi^{\pm}. \]
	It maps $T_i^+$ (resp.\ $T_i^-$) to $T_i$ (resp.\ $T'_i$) for all $0 \leq i \leq n$ (resp.\ $1 \leq i \leq n$).
\end{theorem}

Note that the same symbol $\mathrm{TW}$ is used for both the $\pm$ cases.

In \textit{loc.\ cit.}, the isomorphisms are obtained by matching the braid and quadratic relations for both sides, see \S\S\ref{sec:Hecke-Mp}---\ref{sec:Hecke-O}. As a byproduct, we infer that for each $w \in \Omega_{\mathrm{aff}} \simeq \Omega^+$ (resp.\ $w \in \Omega'_{\mathrm{aff}} \simeq \Omega^-$), it maps $T_w^+$ (resp.\ $T_w^-$) to $T_w$ (resp.\ $T'_w$).

\begin{corollary}\label{prop:TW-support}
	For each $w \in \Omega_{\mathrm{aff}}$ (resp.\ $w \in \Omega'_{\mathrm{aff}}$), up to $\CC^{\times}$ there exists a unique element $T \in H_\psi^+$ (resp.\ $T' \in H_\psi^-$) with $\Supp(T) = \tilde{I}\tilde{w}\tilde{I}$ (resp.\ $\Supp(T') = \tilde{J}\tilde{w}\tilde{J}$), where $\tilde{w} \in \tilde{G}$ is any representative of $w$.
	
	Moreover, if $T$ (resp.\ $T'$) is such an element, then it is invertible.
\end{corollary}
\begin{proof}
	The assertions are known to hold for the Iwahori--Hecke algebras $H^{\pm}$, Moreover, the elements $T_w^{\pm}$ with support $I^{\pm}\dot{w}I^{\pm}$ form a basis of $H^\pm$ as a vector space. As $\mathrm{TW}$ maps $T_w^+$ (resp.\ $T_w^-$) to $T_w$ (resp.\ $T'_w$) and $\Supp(T_w) = \tilde{I}\tilde{w}\tilde{I}$ (resp.\ $\Supp(T'_w) = \tilde{J}\tilde{w}\tilde{J}$), the assertions carry over to $H_\psi^{\pm}$.
\end{proof}

We continue to summarize some key notions and results from \cite{TW18}. The only (minor) difference is that $\tilde{G}$ is an eightfold covering of $G(F)$ here.

\begin{definition}
	Let $\mathcal{G}_\psi^{\pm}$ be the Bernstein blocks of $\tilde{G}\dcate{Mod}$ containing $\omega_\psi^{\pm}$. Let $\mathcal{G}^\pm$ be the Bernstein blocks of $G^{\pm}$ containing the trivial representation $\mathbf{1}_{G^{\pm}(F)}$, i.e.\ the Iwahori-spherical blocks.
\end{definition}

Therefore, $\mathcal{G}^{\pm}$ is equivalent to $H^{\pm}\dcate{Mod}$ via the functor $\sigma \mapsto \sigma^{I^{\pm}}$, i.e.\ $\mathbf{M}_\tau$ in Definition \ref{def:master-functor} with $(K, \tau) = (I^{\pm}, \mathbf{1})$.

Inside $G$, we have $P^1 = M^1 U^1 \supset B^{\leftarrow}$ and $T^1 \subset M^1$ as in Definition \ref{def:P1}. The splitting via Schrödinger model in \eqref{eqn:Levi-splitting} yields
\begin{gather*}
	\tilde{T} \simeq (F^{\times})^n \times \bmu_8, \quad
	\widetilde{T^1} \simeq (F^{\times})^{n-1} \times \bmu_8.
\end{gather*}

Using the splittings above, the genuine characters of $\widetilde{T}$ (resp.\ $\widetilde{T^1}$) will be identified with the characters of $T(F)$ (resp.\ $T^1(F)$). We say such a genuine character is unramified if the corresponding character of $T(F)$ (resp.\ $T^1(F)$) is.

Let $\chi$ be an unramified character of $T(F)$ (resp.\ $T^1(F)$). Denote by $\omega_\psi^{\flat, -}$ the odd Weil representation of $\Mp(W^1)$, and consider the normalized parabolic inductions
\[ i_{\tilde{B}^{\leftarrow}}(\chi), \quad i_{\tilde{P}^1}\left(\omega_\psi^{\flat, -} \boxtimes\chi\right). \]

%
%

\begin{theorem}\label{prop:TW-components}
	The irreducibles in $\mathcal{G}_\psi^+$ (resp.\ $\mathcal{G}_\psi^-$) are precisely the irreducible subrepresentations of $i_{\tilde{B}^{\leftarrow}}(\chi)$ (resp.\ $i_{\tilde{P}^1}(\omega_\psi^{\flat, -} \boxtimes\chi)$) where $\chi$ is some unramified genuine character of $\widetilde{T}$ (resp.\ $\widetilde{T^1}$).
	
	Moreover, by taking the functor $\mathbf{M}_{\tau_0}$ (resp.\ $\mathbf{M}_{\tau_1^-}$) from Definition \ref{def:master-functor}, where we take $K = \tilde{I}$ (resp.\ $K = \tilde{J}$), we obtain equivalences
	\[ \mathcal{G}_\psi^{\pm} \simeq H_\psi^{\pm}\dcate{Mod}. \]
	
	As a consequence, $\mathrm{TW}^*: H_{\psi}^{\pm}\dcate{Mod} \rightiso H^{\pm}\dcate{Mod}$ give rise to equivalences $\mathcal{G}_\psi^\pm \rightiso \mathcal{G}^\pm$.
\end{theorem}
\begin{proof}
	The first part is \cite[Lemma 3.2, 3.5]{TW18}, modulo two minor differences.
	\begin{itemize}
		\item We induce from $B^{\leftarrow}$ and $P^1$ instead of the standard ones.
		\item We use eightfold coverings, so that the inducing data $\chi$ is easier to describe.
	\end{itemize}
	
	The second part is \cite[Theorem 3.4, 3.6]{TW18}.
\end{proof}

There are counterparts for Levi subgroups for all the aforementioned Hecke algebras. For the metaplectic side, consider a covering of the form
\[ \tilde{M} = \Mp(W^\flat) \times \prod_{i=1}^r \GL(n_i, F), \]
where $W^\flat$ is a symplectic $F$-vector space of dimension $2n^\flat$, possibly zero.

Define the algebras $H_\psi^{\pm, \flat}$ for $\Mp(W^\flat)$. Note that $H_\psi^{-, \flat}$ makes sense only when $n^\flat \geq 1$. For each $1 \leq i \leq r$, let $H_i$ denote the standard Iwahori--Hecke algebra for $\GL(n_i, F)$. We then define the algebras
\begin{equation*}\begin{aligned}
	H_\psi^{\tilde{M}, +} & := H_\psi^{+, \flat} \otimes \bigotimes_{i=1}^r H_i , \\
	H_\psi^{\tilde{M}, -} & := H_\psi^{-, \flat} \otimes \bigotimes_{i=1}^r H_i \quad \text{if}\; n^\flat \geq 1.
\end{aligned}\end{equation*}

For the orthogonal side, let the quadratic $F$-vector space $V^{\pm, \flat}$ be as in Definition \ref{def:V}, but now with dimension $2n^\flat + 1$. We have the algebra $H^{\pm, \flat}$ for $\SO(V^{\pm, \flat})$; note that $H^{-, \flat}$ makes sense only when $n^\flat \geq 1$. Consider
\[ M^{\pm} = \SO(V^{\pm, \flat}) \times \prod_{i=1}^r \GL(n_i, F). \]
Its standard Iwahori--Hecke algebra is simply
\begin{equation*}\begin{aligned}
	H^{M^+} & := H^{+, \flat} \otimes \bigotimes_{i=1}^r H_i, \\
	H^{M^-} & := H^{-, \flat} \otimes \bigotimes_{i=1}^r H_i, \quad \text{if}\; n^\flat \geq 1.
\end{aligned}\end{equation*}

They have natural structures of Hilbert algebras, as each of $H_\psi^{\pm}$, $H^{\pm}$ and $H_i$ does.

\begin{theorem}
	We have isomorphisms of Hilbert algebras
	\[ \mathrm{TW}^{\tilde{M}} := \mathrm{TW}^\flat \otimes \bigotimes_{i=1}^r \identity_{H_i} : H^{M^\pm} \rightiso H_\psi^{\tilde{M}, \pm} \]
	by assuming $n^\flat \geq 1$ in the $-$ case. Here $\mathrm{TW}^\flat$ means the Takeda--Wood isomorphism from $H_\psi^{\pm, \flat}$ to $H^{\pm, \flat}$.
\end{theorem}
\begin{proof}
	Immediate from Theorem \ref{prop:TW-isom}.
\end{proof}

If there is no confusion, $\mathrm{TW}^{\tilde{M}}$ will often be abbreviated as $\mathrm{TW}$.

Finally, the analogues of Corollary \ref{prop:TW-support} and Theorem \ref{prop:TW-components} continue to hold in this setting. It suffices to combine the results on $\Mp(W^\flat)$ with the known results on each $\GL(n_i)$.

In particular we have elements $T^{\tilde{M}}_w$ and $(T^{\tilde{M}}_w)'$ in $H^{\tilde{M}, +}_\psi$ and $H^{\tilde{M}, -}_\psi$, respectively.

\subsection{The case of Weil representations}
Our goal is to determine the Hecke modules associated with $\omega_\psi^{\pm}$ via the equivalences in Theorem \ref{prop:TW-components}. This is probably implicit in \cite{GS12, TW18}. Due to the lack of an adequate reference, we give a direct proof below.

Recall the identification between affine Weyl groups $\Omega_{\mathrm{aff}} \simeq \Omega^+$, $\Omega'_{\mathrm{aff}} \simeq \Omega^-$ and the weighted length functions $\ell_0$, $\ell_2$ on them, respectively.

\begin{proposition}\label{prop:Weil-rep-M}
	We have:
	\begin{enumerate}[(i)]
		\item $\mathbf{M}_{\tau_0}(\omega_\psi^+)$ is the $1$-dimensional $H_\psi^+$-module with $T_w$ acting as $q^{\ell_0(w)}$ for all $w \in \Omega_{\mathrm{aff}}$;
		\item $\mathbf{M}_{\tau_1^-}(\omega_\psi^-)$ is the $1$-dimensional $H_\psi^-$-module with $T'_w$ acting as $q^{\ell_2(w)}$ for all $w \in \Omega'_{\mathrm{aff}}$.
	\end{enumerate}
\end{proposition}
\begin{proof}
	Consider (i) first. Let $0 \leq i \leq n$. Let $\iota \in \mathbf{M}_{\tau_0}(\omega_\psi^+)$ be the inclusion $\tau_0 \hookrightarrow \omega_\psi^+|_{\tilde{I}}$. The image of $\iota$ in $\Hom_{\tilde{I}_i}\left( \Ind^{\tilde{I}_i}_{\tilde{I}} \tau_0, \omega_\psi^+ \right)$ under Frobenius reciprocity is $f \mapsto f(1_{\tilde{G}})$.
	Note that $\Ind = \cInd$ here. Recall that there is a decomposition
	\[ \Ind^{\tilde{I}_i}_{\tilde{I}} \tau_0 = \sigma_0 \oplus \sigma_1, \quad \sigma_0 \not\simeq \sigma_1, \]
	where $\sigma_0 = \tau_{0, i} \subset \omega_\psi^+$ is generated by $\tau_0$ acted upon by $\tilde{I}_i$, and $\sigma_1$ is irreducible, $\sigma_1 \not\simeq \sigma_0$. Hence
	\[ \Ker\left[ f \mapsto f(1_{\tilde{G}}) \right] \in \left\{ \sigma_0 \oplus \sigma_1, \sigma_0, \sigma_1, 0 \right\}. \]
	
	The first two cases are impossible since, by the constructions in \S\ref{sec:Hecke-Mp}, $\sigma_0$ contains a copy of $\tau_0$ supported inside $\tilde{I}$. The final case is impossible since $f(1_{\tilde{G}}) = 0$ if $\Supp(f) \cap \tilde{I} = \emptyset$. Hence the kernel is precisely $\sigma_1$.
	
	By Lemma \ref{prop:Hecke-Ind}, to determine $T_i \iota$ it suffices to inspect how the element corresponding to $T_i$ in
	\[ H_{\psi, i}^+ \simeq \End_{\tilde{I}_i}(\Ind^{\tilde{I}_i}_{\tilde{I}}\left( \check{\tau}_0)\right) \]
	acts. By the above discussion, we are reduced to study its action on $\check{\sigma}_0$. According to the characterization of $T_i$, this equals $1$ (if $i=0$) or $q$ (if $i \neq 0$), i.e.\ $T_i$ acts by $q^{\ell_0(s_i)}$ under $\Omega_{\mathrm{aff}} \simeq \Omega^+$.
	
	Therefore the line $\CC \iota$ is $H_\psi^+$-invariant. Since $\omega_\psi^+$ is irreducible and lies in $\mathcal{G}_\psi^+$, we see $\mathbf{M}_{\tau_0}(\omega_\psi^+)$ is irreducible. Hence $\mathbf{M}_{\tau_0}(\omega_\psi^+) = \CC\iota$ on which $H_\psi^+$ acts via $q^{\ell_0(\cdot)}$.
	
	The case (ii) is similar. It suffices to replace $I$ by $J$, $\tau_0$ by $\tau_1^-$, then recall the characterization of $T'_i$ and the definition of $\ell_2$.
\end{proof}

\section{Homomorphism between Hecke algebras}\label{sec:Hecke-homomorphism}
\subsection{The comparison map}\label{sec:comparison-map}
Let $G = \Sp(W)$ as before, where $\dim W = 2n$. Let $\pi$ be in $\tilde{G}\dcate{Mod}$. It makes sense to consider the following maps in Proposition \ref{prop:q-map}.
\begin{itemize}
	\item Let $P = MU$ be a parabolic subgroup, $P \supset B^{\leftarrow}$ and $M \supset T$. Putting $K = \tilde{I}$ and $\tau = \tau_0$, we get
	\begin{equation*}
		\mathrm{q}_U: \mathbf{M}_{\tau_0}(\pi) \to \mathbf{M}_{\tau_{0, \tilde{M}}}(\pi_U).
	\end{equation*}
	\item Suppose furthermore that $P \supset P^1$; see Definition \ref{def:P1}. Putting $K =\tilde{J}$ and $\tau = \tau_1^-$, we get
	\begin{equation*}
		\mathrm{q}_U: \mathbf{M}_{\tau_1^-}(\pi) \to \mathbf{M}_{\tau^-_{1, \tilde{M}}}(\pi_U).
	\end{equation*}
\end{itemize}

Here the premise $K \cap P(F) = K_M K_U$ is ensured by Proposition \ref{prop:IJ-Iwahori-decomp}, and $\tau_{K_U}$ is identified with the $\tilde{M}$-counterpart of $\tau$ using Proposition \ref{prop:coinvariant-type}.

\begin{theorem}\label{prop:q-isom}
	In both cases, $\mathrm{q}_U$ is a linear isomorphism.
\end{theorem}
\begin{proof}
	Let us begin with the minimal case $P = B^{\leftarrow}$ and $(K, \tau) = (\tilde{I}, \tau_0)$. The assertion is then contained in the proof of \cite[Theorem 3.4]{TW18} or \cite[\S 8]{GS12}. In \textit{loc.\ cit.} one takes $B^{\rightarrow}$ instead of $B^{\leftarrow}$, but the argument carries over:
	\begin{itemize}
		\item Surjectivity is proved as in \textit{loc.\ cit.} by a variant of Jacquet's lemma, requiring only the Iwahori decomposition (Proposition \ref{prop:IJ-Iwahori-decomp}).
		\item The proof of injectivity only uses abstract properties of $H_\psi^+$, namely the invertibility of Hecke operators $T_w$.
	\end{itemize}
	
	Similarly, for $P = P^1$ and $(K, \tau) = (\tilde{J}, \tau_1^-)$, the assertion is contained in the proof of \cite[Theorem 3.6]{TW18}.
	
	Next, let us replace $\tilde{G}$ by $\GL(m, F)$ for some $m \geq 1$, take $\pi_{\GL}$ in $\GL(m, F)\dcate{Mod}$, and consider the standard Iwahori subgroup $I_{\GL}$ of $\GL(m, F)$. For every standard (or opposite) parabolic subgroup $Q = LV$ of $\GL(m, F)$, the corresponding linear map
	\[ \mathrm{q}_V: \mathbf{M}_{\mathbf{1}}(\pi_{\GL}) \to \mathbf{M}_{\mathbf{1}}(\pi_{\GL, V}) \]
	is known to be an isomorphism: see eg.\ \cite[(7.9) (b.ii)]{BK98}.
	
	Looking back at $\tilde{G}$ and a $P = MU \supsetneq B^{\leftarrow}$, we decompose $\tilde{M}$ into $\Mp(W^\flat) \times \prod_{i=1}^r \GL(n_i, F)$. We now know that
	\[ \mathbf{M}_{\tau_0}(\pi) \xrightarrow{\mathrm{q}_U} \mathbf{M}_{\tau_{0, \tilde{M}}}(\pi_U) \xrightarrow{\mathrm{q}_{U^{\leftarrow} \cap M}} \mathbf{M}_{\tau_{0, \tilde{T}}}(\pi_{U^{\leftarrow}}) \]
	composes to an isomorphism, by the second part of Proposition \ref{prop:q-map}. Lemma \ref{prop:cpt-intersection-Levi} implies $I \cap M(F)$ is the product of standard Iwahori subgroups on each factor. Also note that $P \cap M$ is the product of
	\begin{itemize}
		\item a parabolic subgroup of $\Sp(W^\flat)$ containing $B^{\leftarrow} \cap \Sp(W^\flat)$,
		\item parabolic subgroups of $\GL(n_i)$ containing $B^{\leftarrow} \cap \GL(n_i)$, i.e.\ opposite to a standard one.
	\end{itemize}
	Hence $\mathrm{q}_{U^{\leftarrow} \cap M}$ decomposes into its avatars for $\Mp(W^\flat)$ and various $\GL(n_i)$. The latter ones are seen to be isomorphisms, whilst the first is an isomorphism by induction. Hence so is $\mathrm{q}_U$.
	
	The case of $P \supsetneq P^1$ and $\tau_1^-$ is settled in exactly the same manner.
\end{proof}

\subsection{Alignment of Hecke elements}\label{sec:align-Hecke}
Let $P = MU \supset B^{\leftarrow}$ be a parabolic subgroup of $G$ with $M = \Sp(W^\flat) \times \prod_{i=1}^r \GL(n_i) \supset T$.

Define $I_M$, $I_U$, $I_{\overline{U}}$ by \eqref{eqn:KM}. As explained in \S\ref{sec:TW-isom}, we have the Hecke algebras $H_\psi^{\tilde{M}, \pm}$ and $H_\psi^\pm$. For the statements concerning $H_\psi^-$, we assume furthermore that
\[ n^\flat := \frac{1}{2} \dim W^\flat \geq 1. \]

\begin{definition}[See {\cite[Definition 6.5]{BK98}}]
	We say $z \in M(F)$ is $P$-positive (relative to $I$) if
	\begin{equation*}
		z I_U z^{-1} \subset I_U, \quad z^{-1} I_{\overline{U}} z \subset I_{\overline{U}}.
	\end{equation*}
	For $z \in \tilde{M}$, we say $z$ is $P$-positive if its image in $M(F)$ is.
\end{definition}

Note that $P$-positivity depends only on the $I_M$-double coset containing $z$. Recall the linear isomorphism
\begin{align*}
	\{\Phi \in H_\psi^+ : \Supp(\Phi) \subset \tilde{I} z\tilde{I} \} & \rightiso \Hom_{\tilde{I} \cap z\tilde{I}z^{-1}}\left( {}^z \check{\tau}_0, \check{\tau}_0 \right) \\
	\Phi & \mapsto \Phi(z).
\end{align*}

The same recipe works for $\widetilde{I_M} \subset \tilde{M}$, giving rise to an isomorphism mapping $\varphi \in H_\psi^{\tilde{M}, +}$ with $\Supp(\varphi) \subset \widetilde{I_M} z \widetilde{I_M}$ to $\varphi(z)$.

\begin{lemma}
	Let $z \in M(F)$. Suppose that $z I_U z^{-1} \subset I_U$ (eg.\ when $z$ is $P$-positive relative to $I$). Then every $E \in \Hom_{\tilde{I} \cap z\tilde{I}z^{-1}}\left( {}^z \check{\tau}_0, \check{\tau}_0 \right)$ induces an element of
	\[ E_{\tilde{M}} \in \Hom_{\widetilde{I_M} \cap z\widetilde{I_M}z^{-1}}\left( {}^z (\check{\tau}_{0, \tilde{M}}), \check{\tau}_{0, \tilde{M}}\right). \]
\end{lemma}
\begin{proof}
	To see that $E$ induces a linear endomorphism $E_{\tilde{M}}$ on the underlying spaces of $\check{\tau}_{0, \tilde{M}}$, note that for all $u \in I_U$ and $\phi \in V_{\check{\tau}_0}$,
	\[ E\left( \check{\tau}_0(u) \phi \right) = E\left( {}^z \check{\tau}_0( zuz^{-1} ) \phi \right) = \check{\tau}_0(\underbracket{zuz^{-1}}_{\in I_U}) E(\phi). \]
	It remains to check the $\widetilde{I_M} \cap z \widetilde{I_M} z^{-1}$-equivariance with $z$-twist, which is obvious.
\end{proof}

Combining these results, for each $P$-positive $z \in \tilde{M}$ we obtain a linear map
\begin{equation}\label{eqn:Phi-to-phi}
	\Phi \mapsto \varphi, \quad \Supp(\Phi) \subset \tilde{I}z\tilde{I}, \quad \Supp(\varphi) \subset \widetilde{I_M}z\widetilde{I_M}
\end{equation}
characterized by $\varphi(z) = \Phi(z)_{\tilde{M}}$.

The same holds in the $-$ case by replacing $I$ (resp.\ $I_M$) by $J$ (resp.\ $J_M$). The following is a variant of \cite[Theorem 7.9 (i)]{BK98}.

\begin{lemma}\label{prop:Hecke-Phi-to-phi}
	Let $z \in \tilde{M}$ be $P$-positive. Suppose that $\Phi \in H_\psi^+$, $\Supp(\Phi) \subset \tilde{I}z\tilde{I}$ and $\varphi \in H_\psi^{\tilde{M}, +}$ is the image of $\Phi$ under \eqref{eqn:Phi-to-phi}. For all $\pi$ in $\tilde{G}\dcate{Mod}$, consider the map $\mathrm{q} = \mathrm{q}_U$ in Theorem \ref{prop:q-isom}. In $\mathbf{M}_{\tau_{0, \tilde{M}}}(\pi_U)$ we have
	\[ \mathrm{q}\left( \delta_{\tilde{P}}(z) \Phi \cdot f \right) = \varphi \cdot \mathrm{q}(f), \quad f \in \mathbf{M}_{\tau_0}(\pi). \]
	
	The same is true in the $-$ case by replacing $I$ by $J$, and $\tau_0$ (resp.\ $\tau_{0, \tilde{M}}$) by $\tau_1^-$ (resp.\ $\tau_{1, \tilde{M}}^-$).
\end{lemma}
\begin{proof}
	The argument is similar to \cite[Theorem 7.9 (i)]{BK98}, so we only give a sketch for the $+$ case. The initial step is the same as \textit{loc.\ cit.}: use the $P$-positivity of $z$ and the Iwahori decomposition of Proposition \ref{prop:IJ-Iwahori-decomp} to write
	\[ \tilde{I}z\tilde{I} = I_U \widetilde{I_M} z \widetilde{I_M} I_{\overline{U}}. \]
	
	Express $f \in \mathbf{M}_{\tau_0}(\pi)$ as $\sum_i v_i \otimes e_i$ where $v_i \in V_\pi$ and $e_i \in V_{\check{\tau}_0}$. Then
	\[ (\delta_{\tilde{P}} \Phi)(\sum_i v_i \otimes e_i) = \sum_i \int_{I_U \widetilde{I_M} z \widetilde{I_M} I_{\overline{U}}} \delta_{\tilde{P}}(z) \left( \pi(g)v_i \otimes \Phi(g)e_i \right) \dd g. \]
	
	If we write $g = y_U y_M y_{\overline{U}}$ where $y_M \in \widetilde{I_M}z\widetilde{I_M}$, then $\delta_{\tilde{P}}(z) = \delta_{\tilde{P}}(y_M)$. Observe that $y_{\overline{U}} \in I_{\overline{U}}$ acts trivially on $\sum_i v_i \otimes e_i$. The integration formula in \cite[(7.8)]{BK98} with $\mes(\tilde{I}) = 1$ and $\mes(I_U) = 1 = \mes(I_{\overline{U}})$ yields
	\[ \sum_i \int_{I_U} \int_{\widetilde{I_M}z\widetilde{I_M}} \pi(y_U y_M) v_i \otimes \Phi(y_U y_M) e_i \dd y_M \dd y_U. \]
	
	After applying $\mathrm{q}$, the actions of $y_U \in I_U$ are trivialized, and then $\Phi(y_U y_M)$ can be replaced by $\varphi(y_M)$ since $y_M$ is also $P$-positive. We now arrive at $\varphi \cdot \mathrm{q}(f)$.
\end{proof}

\begin{remark}\label{rem:Hecke-Phi-to-phi}
	If we take $\mathbf{M}_{\tau_{0, \tilde{M}}}(r_{\tilde{P}}(\pi))$ instead of the non-normalized $\mathbf{M}_{\tau_{0, \tilde{M}}}(\pi_U)$, the linear map $\mathrm{q}$ is unaffected but the $H_\psi^{\tilde{M}, +}$-action is changed. In $\mathbf{M}_{\tau_0}(r^{\tilde{G}}_{\tilde{P}}(\pi))$ the resulting identity becomes
	\[ \mathrm{q}\left( \delta_{\tilde{P}}(z)^{\frac{1}{2}} \Phi \cdot f \right) = \varphi \cdot \mathrm{q}(f). \]
	Ditto in the $-$ case.
\end{remark}

Next, we discuss the opposite direction, going from $\varphi$ to $\Phi$.

\begin{proposition}\label{prop:Phi-to-phi-isom}
	Suppose that $z \in \tilde{M}$ is $P$-positive, then the linear map \eqref{eqn:Phi-to-phi} is bijective in both the $\pm$ cases.
\end{proposition}
\begin{proof}
	For ease of notation, we only consider the $+$ case. The other case is similar.
	
	Corollary \ref{prop:TW-support} implies that the spaces
	\[ \Hom_{\tilde{I} \cap z\tilde{I}z^{-1}}\left( {}^z \check{\tau}_0, \check{\tau}_0 \right), \quad \Hom_{\widetilde{I_M} \cap z\widetilde{I_M}z^{-1}}\left( {}^z (\check{\tau}_{0, \tilde{M}}), \check{\tau}_{0, \tilde{M}}\right) \]
	both have dimension $1$. Hence it suffices to show that \eqref{eqn:Phi-to-phi} is not identically zero.
	
	Suppose that $\Supp(\Phi) \subset \tilde{I}z\tilde{I}$, and $\Phi \mapsto \varphi$. Take any $\pi$ from $\mathcal{G}_\psi^+$ such that $\pi_U \neq 0$, eg.\ $\pi = \omega_\psi^+$ (see Proposition \ref{prop:Jacquet-Weil}).
	
	If $\Phi \neq 0$ then $\Phi$ is invertible by Corollary \ref{prop:TW-support}. Lemma \ref{prop:Hecke-Phi-to-phi} and the fact that $\mathrm{q}$ is an isomorphism (Theorem \ref{prop:q-isom}) entail that $\varphi$ also acts invertibly on $\mathbf{M}_{\tau_{0, \tilde{M}}}(\pi_U)$. Since $\pi_U \neq 0$ and lies in the counterpart for $\tilde{M}$ of $\mathcal{G}_\psi^+$, we have $\mathbf{M}_{\tau_{0, \tilde{M}}}(\pi_U) \neq 0$. In turn, this implies $\varphi \neq 0$.
\end{proof}

We will need the well-known computation of Jacquet modules of $\omega_\psi^\pm$, to be recorded below. Given $P = MU$ as above, we let $P^{\pm}$ be the corresponding reverse-standard (of the same sort as $P$) parabolic subgroup of $G^{\pm}$, with Levi subgroup
\[ M^\pm = \SO(V^{\pm, \flat}) \times \prod_{i=1}^r \GL(n_i) \]
that contains the standard maximal torus $T^{\pm}$ of $G^{\pm}$.

In this way, we may identify $\delta_{P^{\pm}}$ with a character of $\prod_{i=1}^r \GL(n_i, F)$. Hence it can be inflated to $M(F)$.

\begin{proposition}\label{prop:Jacquet-Weil}
	Let $P$ and $P^{\pm}$ be as above. Then
	\[ r_{\tilde{P}}(\omega_\psi^\pm) \simeq \omega_\psi^{\pm, \flat} \boxtimes \delta_{P^\pm}^{-1/2}. \]
	Here we assume $n^\flat := \frac{1}{2}\dim W^\flat$ is positive in the $-$ case.
\end{proposition}
\begin{proof}
	This is well-known if $M = \Sp(W^\flat) \times \GL(n - n^\flat)$, See for example \cite[Theorem 5.1]{MH10} together with \cite[Example 3.2]{Ze80}. The general case follows by taking a further Jacquet functor on $\GL(n - n^\flat)$ using the general fact that $r_Q(\mathbf{1}) = \delta_Q^{-1/2}$.
	
	Note that in \cite{MH10} and many other references, one considers twofold metaplectic coverings. Consequently, $\tilde{M}$ was not split over $\GL$-factors and some terms $\chi_\psi$ appeared. These factors fade away once we compute with Schrödinger models in eightfold coverings; see eg.\ \cite[p.71]{MVW87}.
\end{proof}

Let $\Omega^M_{\mathrm{aff}}$ (resp.\ $(\Omega^M_{\mathrm{aff}})'$) be the product of the affine Weyl group $\Omega^{\flat}_{\mathrm{aff}}$ for $\Sp(W^\flat)$ (resp.\ the $-$ counterpart $(\Omega^{\flat}_{\mathrm{aff}})'$) and the extended affine Weyl groups for various $\GL(n_i)$. It embeds into $\Omega_{\mathrm{aff}}$ (resp.\ $\Omega'_{\mathrm{aff}}$), hence the notion of $P$-positivity makes sense for elements thereof.

\begin{lemma}\label{prop:omega-q}
	Let $w \in \Omega^M_{\mathrm{aff}}$ (resp.\ $(\Omega^M_{\mathrm{aff}})'$) be $P$-positive relative to $I$ (resp.\ $J$). Then
	\[ \mathrm{q}\left( T_w \cdot f \right) = \delta_{P^+}(w)^{-1/2} T^{\tilde{M}}_w \cdot \mathrm{q}(f), \]
	resp.\
	\[ \mathrm{q}\left( T'_w \cdot f \right) = \delta_{P^-}(w)^{-1/2} (T^{\tilde{M}}_w)' \cdot \mathrm{q}(f) \]
	where
	\begin{itemize}
		\item $f$ is in $\mathbf{M}_{\tau_0}(\omega_\psi^+)$ (resp.\ $\mathbf{M}_{\tau_1^-}(\omega_\psi^-)$);
		\item $\mathrm{q} = \mathrm{q}_U$ is as in Theorem \ref{prop:q-isom} with $\pi = \omega_\psi^+$ (resp.\ $\omega_\psi^-$), but we take $r_{\tilde{P}}$ instead of $(\cdot)_U$;
		\item $T^{\tilde{M}}_w$ (resp.\ $(T^{\tilde{M}}_w)'$) acts on $\mathbf{M}_{\tau_{0, \tilde{M}}}(r_{\tilde{P}}(\omega_\psi^+))$ (resp.\ $\mathbf{M}_{\tau_{1, \tilde{M}}^-}(r_{\tilde{P}}(\omega_\psi^-))$), and the equality is interpreted in this space;
		\item we choose representatives of $w$ to make sense of $\delta_{P^\pm}(w)$.
	\end{itemize}
\end{lemma}
\begin{proof}
	Treat the $+$ case first. Decompose $w$ into $(w^\flat, w_{\GL})$ according to $M = \Sp(W^\flat) \times \prod_{i=1}^r \GL(n_i)$. Denote by $\ell_0^\flat$ (resp.\ $\ell_0^{\GL}$) the (weighted) length function on $\Mp(W^\flat)$ (resp.\ the product of $\GL(n_i)$'s), and define $\ell_0^{\tilde{M}} := \ell_0^\flat \ell_0^{\GL}$.
	
	The $T_w$-action is $q^{\ell_0(w)}$. In view of Proposition \ref{prop:Weil-rep-M}, the $T^{\tilde{M}}_w$-action on $\mathbf{M}_{\tau_{0, \tilde{M}}}(r_{\tilde{P}} (\omega_\psi^+)) \simeq \mathbf{M}_{\tau_{0, \tilde{M}}}(\omega_\psi^{+, \flat} \boxtimes \delta_{P^+}^{-1/2})$ (via Proposition \ref{prop:Jacquet-Weil}) equals
	\[ \underbracket{q^{\ell_0^\flat(w^\flat)}}_{\omega_\psi^{+, \flat}\;\text{on}\; \Mp(W^\flat)} \;
	\underbracket{q^{\ell^{\GL}_0(w_{\GL})}}_{\mathbf{1}\;\text{on}\; \prod_i \GL(n_i, F)} \;
	\delta_{P^+}(w)^{-1/2}. \]
	Our assertion thus amounts to $q^{\ell_0(w)} = \delta_{P^+}(w)^{-1} q^{\ell^{\tilde{M}}_0(w)}$.
	
	The same question can be posed for $P^+ = M^+ U^+ \subset G^+$ and the trivial representation $\mathbf{1}$ of $G^+(F)$: letting $\Omega^{M^+}$ be the $M^+$-counterpart of $\Omega^+$, the equation
	\[ \mathrm{q}(T^+_w \cdot f^+) = \delta_{P^+}(w)^{-1/2} T^{M^+}_w \cdot q(f^+) \]
	holds true in $\mathbf{M}_{\mathbf{1}}(r_{P^+}(\mathbf{1})) = r_{P^+}(\mathbf{1})^{I_{M^+}}$, provided that $w \in \Omega^{M^+}$ is $P^+$-positive. Indeed, this falls within the scope of \cite[Theorem 7.9 (i)]{BK98} in the Iwahori--Hecke case, by noting that $T^{G^+}_w = (\mathrm{t}\delta_{P^+}) (T_w^{M^+})$ in the notation of \textit{loc.\ cit.}
	
	By the standard computation of normalized co-invariants and Hecke modules associated with $\mathbf{1}$, the equality above is seen to imply
	\[ q^{\ell_0(w)} = \delta_{P^+}(w)^{-1} q^{\ell_0^{M^+}(w)} \]
	for $P^+$-positive $w$. The natural isomorphism $\Omega^M_{\mathrm{aff}} \simeq \Omega^{M^+}$ matches $P$-positive and $P^+$-positive elements, whereas the length function $\ell_0^{M^+}$ for $M^+$ matches $\ell_0^{\tilde{M}}$. This proves the desired equality.
	
	Finally, the $-$ case is entirely similar, by reducing to a parallel statement on $P^- \subset G^-$ and the trivial representations, using the (weighted) length function $\ell_2$.
\end{proof}

We now generalize this comparison to general smooth representations $\pi$ of $\tilde{G}$.

\begin{lemma}\label{prop:Hecke-equivariance-positive}
	Let $\pi$ be in $\tilde{G}\dcate{Mod}$.
	\begin{enumerate}[(i)]
		\item Let $w \in \Omega^M_{\mathrm{aff}}$ be $P$-positive (relative to $I$). Then
		\[ \mathrm{q}(T_w \cdot f) = \delta_{P^+}(w)^{-1/2} T_w^{\tilde{M}} \cdot \mathrm{q}(f) \]
		in $\mathbf{M}_{\tau_{0, \tilde{M}}}(r_{\tilde{P}}(\pi))$ for all $f \in \mathbf{M}_{\tau_0}(\pi)$.
		\item Let $w \in (\Omega^M_{\mathrm{aff}})'$ be $P$-positive (relative to $J$). Then
		\[ \mathrm{q}(T'_w \cdot f) = \delta_{P^-}(w)^{-1/2} (T_w^{\tilde{M}})' \cdot \mathrm{q}(f) \]
		in $\mathbf{M}_{\tau_{1, \tilde{M}}^-}(r_{\tilde{P}}(\pi))$ for all $f \in \mathbf{M}_{\tau_1^-}(\pi)$.
	\end{enumerate}
\end{lemma}
\begin{proof}
	It is enough to treat (i) since (ii) is entirely similar. Put $\varphi := T^{\tilde{M}}_w \in H_\psi^{\tilde{M}, +}$, and let $\Phi$ be its preimage under \eqref{eqn:Phi-to-phi} using Proposition \ref{prop:Phi-to-phi-isom}. By comparing supports and using Corollary \ref{prop:TW-support}, there exists $\lambda_w \in \CC^{\times}$ such that $\Phi = \lambda_w T_w$ in $H_\psi^+$.
	
	According to Remark \ref{rem:Hecke-Phi-to-phi}, for all $\pi$ and all $f \in \mathbf{M}_{\tau_0}(\pi)$ we have
	\[ \delta_P(w)^{1/2} \lambda_w \mathrm{q}\left(T_w \cdot f \right) = \mathrm{q}\left( \delta_P(w)^{1/2} \Phi \cdot f \right) = T_w^{\tilde{M}} \cdot \mathrm{q}(f) \]
	in $\mathbf{M}_{\tau_{0, \tilde{M}}}(r_{\tilde{P}}(\pi))$. Apply this to $\pi = \omega_\psi^+$ and invoke Lemma \ref{prop:omega-q}. Recalling that $T_w$ and $T_w^{\tilde{M}}$ are invertible, $\mathrm{q}$ is bijective, and all Hecke modules in view are nonzero in this case, we infer that
	\[ \delta_P(w)^{1/2} \lambda_w = \delta_{P^+}(w)^{1/2}. \]
	
	Plugging $\lambda_w = \left(\frac{\delta_{P^+}(w)}{\delta_P(w)}\right)^{1/2}$ back into the case of general $\pi$, the desired equality follows.
\end{proof}

\subsection{Hecke equivariance}
Retain the assumptions on $P = MU \subset G$ and $P^{\pm} \subset G^{\pm}$ from \S\ref{sec:align-Hecke}.

According to (6.3), (6.12) and (7.2) of \cite{BK98}, applied to the special case of Iwahori--Hecke algebras, there exists a unique embedding of algebras
\[ \mathrm{t}^{\pm}: H^{M^\pm} \hookrightarrow H^{\pm} \]
which maps $T^{M^\pm}_w$ to $T^{\pm}_w$ for all $w \in \Omega^{\pm}$ that is $P^{\pm}$-positive.

In \textit{loc.\ cit.}, one works primarily with the homomorphism $\mathrm{t}^{\pm} \delta_{P^{\pm}}: h \mapsto \mathrm{t}^{\pm}(\delta_{P^{\pm}}h)$ (cf.\ Definition \ref{def:character-twist}). In this article we prefer
\begin{equation*}
	\mathfrak{t}^{\pm}_{\mathrm{nor}} := \mathrm{t}^{\pm} \delta_{P^\pm}^{\frac{1}{2}} : H^{M^\pm} \hookrightarrow H^{\pm}.
\end{equation*}

The benefit is that $\mathfrak{t}^{\pm}_{\mathrm{nor}}$ is a homomorphism between Hilbert algebras. See \cite[\S\S 6.2---6.3]{BHK11}.

\begin{definition}\label{def:tnor}
	Let $\mathrm{t}_{\mathrm{nor}}$ be the composition of
	\[ H_\psi^{\tilde{M}, \pm} \xrightarrow[\sim]{{(\mathrm{TW}^{\tilde{M}})^{-1}}} H^{M^\pm} \xrightarrow{\mathrm{t}_{\mathrm{nor}}^\pm} H^{\pm} \xrightarrow[\sim]{\mathrm{TW}} H_\psi^\pm. \]
	It is an embedding of Hilbert algebras, since each arrow above is.
\end{definition}

We are ready to state the first main result of this article. Recall from Theorem \ref{prop:q-isom} that we have
\begin{itemize}
	\item $\mathrm{q} = \mathrm{q}_U: \mathbf{M}_{\tau_0}(\pi) \to \mathbf{M}_{\tau_{0, \tilde{M}}}(r_{\tilde{P}}(\pi))$ (in the $+$ case),
	\item $\mathrm{q} = \mathrm{q}_U: \mathbf{M}_{\tau_1^-}(\pi) \to \mathbf{M}_{\tau_{1, \tilde{M}}^-}(r_{\tilde{P}}(\pi))$ (in the $-$ case)
\end{itemize}
between vector spaces, functorial in $\pi$ from $\tilde{G}\dcate{Mod}$.

\begin{theorem}\label{prop:Hecke-equivariance}
	The linear map $\mathrm{q}$ is Hecke-equivariant with respect to $\mathrm{t}_{\mathrm{nor}}$. Namely, for all
	\begin{itemize}
		\item $f \in \mathbf{M}_{\tau_0}(\pi)$ and $T \in H_\psi^{\tilde{M}, +}$ (in the $+$ case),
		\item $f \in \mathbf{M}_{\tau_1^-}(\pi)$ and $T \in H_\psi^{\tilde{M}, -}$ (in the $-$ case),
	\end{itemize}
	we have
	\[ \mathrm{q}(\mathrm{t}_{\mathrm{nor}}(T) \cdot f) = T \cdot \mathrm{q}(f). \]
\end{theorem}
\begin{proof}
	For ease of notation, we only treat the $+$ case. By linearity, it suffices to deal with the case $T = T^{\tilde{M}}_w$ for some $w \in \Omega^M_{\mathrm{aff}}$. There exists a central and $P$-positive element $z \in \Omega^M_{\mathrm{aff}}$ such that $zw$ is $P$-positive, which we fix. Thus, $T^{\tilde{M}}_z T^{\tilde{M}}_w = T^{\tilde{M}}_{zw}$.
	
	By Lemma \ref{prop:Hecke-equivariance-positive},
	\[ \mathrm{q}\left(\mathrm{t}_{\mathrm{nor}}\left(T^{\tilde{M}}_{zw}\right) \cdot f\right) = T^{\tilde{M}}_{zw} \cdot \mathrm{q}(f) = T^{\tilde{M}}_z T^{\tilde{M}}_w \cdot \mathrm{q}(f). \]
	
	We have $\mathrm{t}_{\mathrm{nor}}\left( T^{\tilde{M}}_{zw} \right) = \mathrm{t}_{\mathrm{nor}}\left( T^{\tilde{M}}_z T^{\tilde{M}}_w \right) = \mathrm{t}_{\mathrm{nor}}\left( T^{\tilde{M}}_z \right) \mathrm{t}_{\mathrm{nor}}\left( T^{\tilde{M}}_w \right)$. Set $r := \mathrm{t}_{\mathrm{nor}}\left( T^{\tilde{M}}_w\right) \cdot f$. Using Lemma \ref{prop:Hecke-equivariance-positive} once again, we get
	\[ \mathrm{q}\left(\mathrm{t}_{\mathrm{nor}}\left(T^{\tilde{M}}_z\right) \cdot r \right) = T^{\tilde{M}}_z \cdot \mathrm{q}(r). \]
	
	We conclude that
	\[ T^{\tilde{M}}_z T^{\tilde{M}}_w \cdot \mathrm{q}(f) = T_z^{\tilde{M}} \cdot \mathrm{q}\left(\mathrm{t}_{\mathrm{nor}}\left( T^{\tilde{M}}_w\right) \cdot f\right). \]
	As $T_z^{\tilde{M}}$ is invertible, the desired equality follows.
\end{proof}

\subsection{Jacquet functor and parabolic induction for Hecke modules}
Let $B^{\leftarrow} \subset P = MU \subset G$ be as in \S\ref{sec:align-Hecke}. With the homomorphism $\mathrm{t}_{\mathrm{nor}}$ of Definition \ref{def:tnor}, Theorem \ref{prop:Hecke-equivariance} can be rephrased as the commutative ``$2$-cell''
\begin{equation}\label{eqn:r-compatibility}\begin{tikzcd}
	\tilde{G}\dcate{Mod} \arrow[r, "\mathbf{M}"] \arrow[d, "{r_{\tilde{P}} }"'] & {H_\psi^\pm}\dcate{Mod} \arrow[d, "{\mathrm{t}_{\mathrm{nor}}^*}"] \arrow[Rightarrow, ld, "\mathrm{q}", "\sim" sloped] \\
	\tilde{M}\dcate{Mod} \arrow[r, "\mathbf{M}"'] & {H_\psi^{\tilde{M}, \pm}}\dcate{Mod}
\end{tikzcd}\end{equation}
where the functors $\mathbf{M}$ (Definition \ref{def:master-functor}) are associated with $(\tilde{I}, \tau_0)$, $(\tilde{J}, \tau_1^-)$ or their avatars for $\tilde{M}$. By adjunction, we get the commutative $2$-cell
\begin{equation}\label{eqn:i-compatibility}\begin{tikzcd}
	\tilde{M}\dcate{Mod} \arrow[r, "\mathbf{M}"] \arrow[d, "{i_{\tilde{P}} }"'] & {H_\psi^{\tilde{M}, \pm}}\dcate{Mod} \arrow[d, "{ \Hom_{H_\psi^{\tilde{M}, \pm}}(H_\psi^\pm, \cdot) }"] \\
	\tilde{G}\dcate{Mod} \arrow[r, "\mathbf{M}"'] \arrow[Rightarrow, ru, "\sim" sloped] & {H_\psi^\pm}\dcate{Mod}
\end{tikzcd}\end{equation}
where $H_\psi^\pm$ is viewed as an $(H_\psi^{\tilde{M}, \pm}, H_\psi^\pm)$-bimodule via $\mathrm{t}_{\mathrm{nor}}$.

Let $\mathcal{G}_\psi^{\tilde{M}, \pm}$ (resp.\ $\mathcal{G}^{M^\pm}$) be the avatars of $\mathcal{G}_\psi^{\pm}$ (resp.\ $\mathcal{G}^{\pm}$) for $\tilde{M}$ (resp.\ $M^\pm$): on the $\GL$ factors they are simply the Iwahori-spherical blocks. Let $P^{\pm} \subset G^{\pm}$ be the reverse-standard parabolic subgroups corresponding naturally to $P \subset G$.

\begin{proposition}\label{prop:ri-compatibility}
	The following diagram commutes up to a canonical isomorphism.
	\[\begin{tikzcd}
		\mathcal{G}_\psi^{\tilde{M}, \pm} \arrow[r] \arrow[d, "{i_{\tilde{P}}}"'] & \mathcal{G}^\pm \arrow[d, "{i_{P^\pm}}"] \\
		\mathcal{G}_\psi^{\pm} \arrow[r] & \mathcal{G}^{\pm}
	\end{tikzcd}\]
	where the horizontal arrows are given by $\mathrm{TW}$ (Theorem \ref{prop:TW-components}).
	
	There is a similar diagram for normalized Jacquet functors, which commutes up to a canonical isomorphism.
\end{proposition}
\begin{proof}
	Consider the case of parabolic induction. Combine the diagrams
	\begin{equation*}\begin{tikzcd}
		\mathcal{G}_\psi^{\tilde{M}, \pm} \arrow[r, "\sim"] \arrow[d, "{i_{\tilde{P}} }"'] & {H_\psi^{\tilde{M}, \pm}}\dcate{Mod} \arrow[d, "{ \Hom_{H_\psi^{\tilde{M}, \pm}}(H_\psi^\pm, \cdot) }"] \\
		\mathcal{G}_\psi^{\pm} \arrow[r, "\sim"'] & {H_\psi^\pm}\dcate{Mod}
	\end{tikzcd} \quad \begin{tikzcd}
		\mathcal{G}^{M^{\pm}} \arrow[r, "\sim"] \arrow[d, "{i_{P^{\pm}} }"'] & {H^{M^{\pm}}}\dcate{Mod} \arrow[d, "{ \Hom_{H^{M^{\pm}}}(H^\pm, \cdot) }"] \\
		\mathcal{G}^{\pm} \arrow[r, "\sim"'] & {H^\pm}\dcate{Mod}
	\end{tikzcd}\end{equation*}
	coming from \eqref{eqn:i-compatibility} and its analogue for $G^\pm$ established in \cite[\S 7]{BK98} (using $\mathrm{t}^{\pm}_{\mathrm{nor}}: H^{M^\pm} \to H^{\pm}$), and
	\begin{equation*}\begin{tikzcd}
		H_\psi^{\tilde{M}, \pm}\dcate{Mod} \arrow[r, "\sim", "{\mathrm{TW}^{\tilde{M}, *}}"'] \arrow[d] & H^{M^{\pm}}\dcate{Mod} \arrow[d] \\
		H_\psi^{\pm}\dcate{Mod} \arrow[r, "\sim"', "{\mathrm{TW}^*}"] & H^{\pm}\dcate{Mod}
	\end{tikzcd}\end{equation*}
	where the vertical functors are given by $\Hom$ as before. They all commute up to canonical isomorphisms: for the last diagram, this follows directly from the definition of $\mathrm{t}_{\mathrm{nor}}$.
	
	The case of normalized Jacquet modules is entirely similar.
\end{proof}

We have to make the natural isomorphism $\mathbf{M} \circ i_{\tilde{P}} \rightiso \Hom_{H_\psi^{\tilde{M}, \pm}}(H_\psi^\pm, \mathbf{M}(\cdot))$ explicit. It turns out to be ``the obvious one''.

\begin{proposition}\label{prop:Hecke-induction}
	The aforementioned natural isomorphism is given as follows. Let $\pi$ be in $\tilde{M}\dcate{Mod}$. For all $\sigma$ in $\mathbf{M}_{\tau_0}(i_{\tilde{P}}(\pi))$ (resp.\ $\mathbf{M}_{\tau_1^-}(i_{\tilde{P}}(\pi))$). Its image is the homomorphism $H_\psi^+ \to \mathbf{M}_{\tau_{0, \tilde{M}}}(\pi)$ (resp.\ $H_\psi^- \to \mathbf{M}_{\tau_{1, \tilde{M}}^-}(\pi)$) that maps $h \in H_\psi^+$ (resp.\ $h \in H_\psi^-$) to the element
	\[ \overline{\phi} \mapsto \underbracket{(h \cdot \sigma)(\phi)}_{\in i_{\tilde{P}}(\pi)} (1_{\tilde{G}}) \]
	where $\overline{\phi} \in V_{\tau_{0, \tilde{M}}}$ (resp.\ $\overline{\phi} \in V_{\tau_{1, \tilde{M}}^-}$) and $\phi \in V_{\tau_0}$ (resp.\ $\phi \in V_{\tau_1^-}$) is any preimage of $\overline{\phi}$.
\end{proposition}
\begin{proof}
	The argument is formal, and we sketch the $+$ case only. There is an $H_\psi^{\tilde{M}, +}$-linear isomorphism
	\[ \mathrm{q}: \mathbf{M}_{\tau_0}(i_{\tilde{P}}(\pi)) \rightiso \mathbf{M}_{\tau_{0, \tilde{M}}}(r_{\tilde{P}} i_{\tilde{P}}(\pi)) \]
	where $H_\psi^{\tilde{M}, +}$ acts on the left hand side through $\mathrm{t}_{\mathrm{nor}}$. Composed with the co-unit $\epsilon: r_{\tilde{P}} i_{\tilde{P}}(\pi) \to \pi$, namely the evaluation at $1_{\tilde{G}}$, we obtain an $H_\psi^{\tilde{M}, +}$-linear map $\mathbf{M}_{\tau_0}(i_{\tilde{P}}(\pi)) \to \mathbf{M}_{\tau_{0, \tilde{M}}}(\pi)$.
	
	By adjunction, we obtain an $H_\psi^+$-linear map
	\[ \mathbf{M}_{\tau_0}(i_{\tilde{P}}(\pi)) \to \Hom_{H_\psi^{\tilde{M}, +}}\left( H_\psi^+ , \mathbf{M}_{\tau_{0, \tilde{M}}}(\pi) \right). \]
	
	This furnishes the natural transformation, for formal reasons. To be precise, by recalling the definition of $\mathrm{q}$ in Proposition \ref{prop:q-map}, every $\sigma \in \mathbf{M}_{\tau_0}(i_{\tilde{P}}(\pi))$ is mapped to the homomorphism
	\[ \underbracket{h}_{\in H_\psi^+} \mapsto \left[ \underbracket{\overline{\phi}}_{\in V_{\tau_{0, \tilde{M}}}} \mapsto \epsilon(\mathrm{q}(h \cdot \sigma)(\overline{\phi})) = (h \cdot \sigma)(\phi)(1_{\tilde{G}}) \right], \]
	where $\phi \in V_{\tau_0}$ is an arbitrary preimage of $\overline{\phi}$.
\end{proof}

\begin{remark}
	Define $G^{\pm} = \SO(V^\pm)$ as in \S\ref{sec:Hecke-O}, and let $P^{\pm} = M^{\pm} U^{\pm}$ correspond to $P = MU$ as before. Of course, the recipe above applies to $G^{\pm}$, $M^{\pm}$ and their Iwahori--Hecke algebras. For the description of parabolic induction, the corresponding statement is that
	\begin{equation*}
		i_{P^\pm}(\pi)^{I^\pm} \rightiso \Hom_{H^{M^\pm}}\left(H^{\pm}, \pi^{I_{M^+}}\right)
	\end{equation*}
	as left $H^+$-modules, where $H^{\pm}$ is regarded as an $(H^{M^\pm}, H^\pm)$-bimodule through $\mathrm{t}^{\pm}_{\mathrm{nor}}$ on the right hand side. More precisely, the isomorphism maps $\sigma \in i_{P^+}(\pi)^{I^+}$ to
	\begin{equation*}
		\underbracket{h}_{\in H^\pm} \mapsto \underbracket{(h \cdot \sigma)}_{\in i_{P^\pm}(\pi)^{I^\pm}} (1_{G^+}).
	\end{equation*}

	Indeed, the case of $G^\pm$ is covered by \cite{BK98}, and is well-known, except that we work with reverse-standard parabolic subgroups instead of the standard ones.
\end{remark}

\begin{remark}
	The second adjunction theorem asserts that $i_{\tilde{\overline{P}}}$ is left adjoint to $r_{\tilde{P}}$. On the level of Hecke modules, it corresponds to $H_\psi^{\pm} \dotimes{H_\psi^{\tilde{M}, +}} (\cdot)$.
\end{remark}

\section{Spherical parts}\label{sec:spherical-parts}
\subsection{Spherical projectors}\label{sec:spherical-projector}
We begin by defining the idempotents $e^{\pm} \in H_\psi^{\pm}$.

\begin{definition}\label{def:spherical-idempotent}
	Let $e^+$ be the map $\tilde{G} \to \End_{\CC}(V_{\check{\tau}_0})$ given by
	\[ e^+(g) = \begin{cases}
		\mes(K_0)^{-1} \check{\tau}_0(g), & g \in \widetilde{K_0} \\
		0, & g \notin \widetilde{K_0}.
	\end{cases}\]
	Let $e^-$ be the map $\tilde{G}\to \End_{\CC}(V_{\check{\tau}_1^-})$ given by
	\[ e^-(g) = \begin{cases}
		\mes(K_1)^{-1} \check{\tau}_1^-(g), & g \in \widetilde{K_1} \\
		0, & g \notin \widetilde{K_0}.
	\end{cases}\]
\end{definition}

In view of Lemma \ref{prop:idempotent-projector}, we have:
\begin{itemize}
	\item $e^\pm$ are idempotents in $H_\psi^{\pm}$,
	\item the action of $e^+$ (resp.\ $e^-$) on $\mathbf{M}_{\tau_0}(\pi)$ (resp.\ $\mathbf{M}_{\tau_1^-}(\pi)$) is a projection operator onto the subspace $\Hom_{\widetilde{K_0}}(\tau_0, \omega_\psi^+)$ (resp.\ $\Hom_{\widetilde{K_1}}(\tau_1^-, \omega_\psi^-)$).
\end{itemize}

Consider $G^{\pm} := \SO(V^\pm)$ and their Iwahori--Hecke algebras $H^{\pm}$ as in \S\ref{sec:Hecke-O}.

\begin{definition}\label{def:K-pm}
	Let $K^{\pm} \subset G^{\pm}(F)$ be the maximal compact subgroups associated to special vertices chosen by removing the white nodes in \eqref{eqn:Gplus-Dynkin} and \eqref{eqn:Gminus-Dynkin}.
\end{definition}

\begin{remark}\label{rem:K-pm}
	The following property will be needed in \S\ref{sec:odd-computations}. In \cite[p.390]{Ca80}, one attaches a root label $q_\alpha \geq 1$ to each affine root $\alpha$ of $G^{\pm}$; these root labels are $\Omega_{\mathrm{aff}}^{\pm}$-invariant. We say $\alpha$ is spherical if it vanishes at the chosen special vertex, in which case one puts $q_{\alpha/2} := q_{\alpha + 1}/q_\alpha$. These quantities appear in the quadratic relations $(T_s + 1)(T_s - q_\alpha) = 0$ of Iwahroi--Hecke algebras, where $\alpha$ is a simple affine root and $s$ is the corresponding reflection.
	\begin{itemize}
		\item Since $G^+$ is split, $q_\alpha = q$ and $q_{\alpha/2} = 1$ for each spherical root $\alpha$ by \cite[3.3 Remark]{Ca80}.
		\item For $G^-$, our choice of special vertex for $G^-$ ensures $q_{\alpha/2} \geq 1$. Indeed, from the quadratic relations in \S\ref{sec:Hecke-O}, we see $q_{\alpha + 1}/q_\alpha \in \{1, q\}$.
	\end{itemize}
\end{remark}

\begin{definition}
	Recall that $\Omega_0 := \lrangle{s_1, \ldots, s_n} \subset \Omega_{\mathrm{aff}}$; under the natural identification, it corresponds to the ``spherical subgroup'' $\Omega^+_0$ of $\Omega^+$.
	
	Similarly, with the notations of \S\ref{sec:Hecke-Mp} we define $\Omega'_0 := \lrangle{s'_2, \ldots, s'_n} \subset \Omega'_{\mathrm{aff}}$; it corresponds to the spherical subgroup $\Omega^-_0$ of $\Omega^-$.
\end{definition}

There is a decomposition $K^{\pm} = \bigsqcup_{w \in \Omega^{\pm}_0} I^{\pm} w I^{\pm}$. Note that $\Omega^{\pm}_0$ is generated by black nodes in \eqref{eqn:Gplus-Dynkin} and \eqref{eqn:Gminus-Dynkin}.

The counterparts of $e^{\pm} \in H_\psi^{\pm}$ for $G^{\pm}$ are
\[ \mes(K^{\pm})^{-1} \mathbf{1}_{K^{\pm}} \in H^{\pm}. \]
They are idempotents whose actions are projections onto $\pi^{K^{\pm}}$, for every $\pi$ in $G^{\pm}\dcate{Mod}$.

\begin{proposition}\label{prop:matching-e}
	We have $\mathrm{TW}\left( \mes(K^{\pm})^{-1} \mathbf{1}_{K^{\pm}}\right) = e^{\pm}$.
\end{proposition}
\begin{proof}
	Consider the $+$ case first. Identify $\Omega_{\mathrm{aff}}$ and $\Omega^+$ and regard $\ell_0$ as a function on $\Omega_{\mathrm{aff}}$.
	
	Decompose $e^+ = \sum_{w \in \Omega_0} c_w T_w$ where $c_w \in \CC$. Fix $w \in \Omega_0$. Since $\mathrm{TW}(T_w^+) = T_w$, it suffices to show
	\[ c_w = \mes(K^+)^{-1} = q^{-\sum_{v \in \Omega_0} \ell_0(v) }. \]
	
	Pick a representative $\dot{w} \in G(F)$ of $w$, and a preimage $\tilde{w}$ of $\dot{w}$. By definition, $c_w T_w$ acts on $\mathbf{M}_{\tau_0}(\pi)$ by sending $\sum_i v_i \otimes e_i$ (where $v_i \in V_\pi$ and $e_i \in V_{\check{\tau}_0}$) to
	\[ \mes(K_0)^{-1} \int_{\tilde{I}\tilde{w}\tilde{I}} \sum_i \pi(g)v_i \otimes \check{\tau}_0(g) e_i \dd g. \]
	Now take $\pi = \omega_\psi^+$, so that $c_w T_w$ acts by $c_w q^{\ell_0(w)}$. Take $\sum_i v_i \otimes e_i$ corresponding to the $\widetilde{K_0}$-equivariant inclusion $\tau_0 \hookrightarrow \omega_\psi^+$. We deduce that
	\[ c_w q^{\ell_0(w)} = \frac{\mes(I\dot{w}I)}{\mes(K_0)}. \]
	Hence
	\[ c_w = q^{\ell(w) - \ell_0(w)} q^{-\sum_{v \in \Omega_0} \ell(v) } \]
	
	An inspection of the formulas in \S\ref{sec:Hecke-O} reveals that $\ell = \ell_0$ on $\Omega_0$. It follows that $c_w = q^{-\sum_{v \in \Omega_0} \ell_0(v)}$, as required.
	
	As for the $-$ case, we replace $\Omega_{\mathrm{aff}}$ (resp.\ $\Omega^+$, $\ell_0$, $\ell$, $\Omega_0$) by $\Omega'_{\mathrm{aff}}$ (resp.\ $\Omega^-$, $\ell_2$, $\ell'$, $\Omega'_0$). In the final step of the arguments, we also replace $\omega_\psi^+$ by $\omega_\psi^-$. The remaining arguments carry over verbatim.
\end{proof}

\subsection{Unramified principal series}
We consider two types of normalized parabolic inductions in $\tilde{G}$. The notation of Definition \ref{def:P1} will be used.
\begin{itemize}
	\item $i_{\tilde{B}^{\leftarrow}}(\chi)$ where $\chi$ is a genuine character of $\tilde{T}$;
	\item $i_{\tilde{P}^1}(\omega_\psi^{\flat, -} \boxtimes \chi)$ where $\chi$ is a genuine character of $\widetilde{T^1}$.
\end{itemize}

In what follows, we assume $\chi$ is unramified in the sense explained in \S\ref{sec:TW-isom}.

Note that $H_\psi^{\tilde{T}, +} = \CC[X_*(T)]$, and in this case $\chi$ can be viewed as a homomorphism $H_\psi^{\tilde{T}, +} \to \CC$.

On the other hand, $H_\psi^{\tilde{M}^1, -} = H_\psi^{\flat, -} \otimes \CC[X_*(T^1)]$, recalling that $H_\psi^{\flat, -}$ is the avatar of $H_\psi^-$ for $\Mp(W^1)$. Therefore $\mathbf{M}_{\tau_1^{\flat, -}}(\omega_\psi^{\flat, -}) \boxtimes \chi$ can be viewed an $H_\psi^{\tilde{M}^1, -}$-module.

Let $\tau_1^{\flat, -}$ (resp.\ $\omega_\psi^{\flat, -}$) be the avatar of $\tau_1^-$ (resp.\ $\omega_\psi^-$) for $\Mp(W^1)$.

\begin{proposition}\label{prop:principal-series-Hecke}
	There is an isomorphism of $H_\psi^+$-modules
	\[\begin{tikzcd}[row sep=tiny]
		\mathbf{M}_{\tau_0}(i_{\tilde{B}^{\leftarrow}}(\chi)) \arrow[r, "\sim"] & \Hom_{H_\psi^{\tilde{T}, +}}(H_\psi^+, \chi) \\
		\sigma \arrow[mapsto, r] & {\left[ h \mapsto (h \cdot \sigma)(\phi)(1_{\tilde{G}}) \right]}
	\end{tikzcd}\]
	where $\phi$ is an arbitrary element of $V_{\tau_0}$ such that $\phi(0, \ldots, 0) = 1$.
	
	In parallel, there is an isomorphism of $H_\psi^-$-modules
	\[\begin{tikzcd}[row sep=tiny]
		\mathbf{M}_{\tau_1^-}(i_{\tilde{P}^1}(\omega_\psi^{\flat, -} \boxtimes \chi)) \arrow[r, "\sim"] & \Hom_{H_\psi^{\tilde{M}^1, -}}\left( H_\psi^-, \mathbf{M}_{\tau_1^{\flat, -}}(\omega_\psi^{\flat, -}) \otimes \chi \right) \\
		\sigma \arrow[mapsto, r] & {\left[ h \mapsto [\overline{\phi} \mapsto (h \cdot \sigma)(\phi)(1_{\tilde{G}})] \right]}
	\end{tikzcd}\]
	where $\overline{\phi}$ is in $V_{\tau_1^{\flat, -}}$ and $\phi$ is in $V_{\tau_1^-}$ such that $\phi(\cdot, \underbracket{0, \ldots, 0}_{n-1}) = \overline{\phi}$.
\end{proposition}
\begin{proof}
	The $+$ case is a special instance of Proposition \ref{prop:Hecke-induction}, by recalling the precise description of $\tau_0 \twoheadrightarrow \tau_{0, U^{\leftarrow}} \simeq \tau_{0, \tilde{T}}$ in Proposition \ref{prop:coinvariant-type}.
	
	The same argument applies to the $-$ case, with only notational differences.
\end{proof}

\begin{lemma}\label{prop:sigma-generator}
	For every unramified character $\chi$ of $\tilde{T}$ (resp.\ $\widetilde{T^1}$), the space $\Hom_{\widetilde{K_0}}(\tau_0, i_{\tilde{B}^{\leftarrow}}(\chi))$ (resp.\ $\Hom_{\widetilde{K_1}}(\tau_1^-, i_{\tilde{P}^1}(\omega_\psi^{\flat, -} \boxtimes \chi))$) is $1$-dimensional. Specifically,
	\begin{itemize}
		\item $\Hom_{\widetilde{K_0}}(\tau_0, i_{\tilde{B}^{\leftarrow}}(\chi))$ is generated by the element
		\[ \sigma(\chi): \phi \mapsto \left[ k \mapsto (\omega_\psi(k)\phi)(0, \ldots, 0) \right], \]
		where $\phi \in V_{\tau_0}$ and $k \in \widetilde{K_0}$;
		
		\item $\Hom_{\widetilde{K_1}}(\tau_1^-, i_{\tilde{P}^1}(\omega_\psi^{\flat, -} \boxtimes \chi))$ is generated by the element
		\[ \sigma'(\chi): \phi \mapsto \left[ k \mapsto (\omega_\psi(k)\phi)(\cdot, 0, \ldots, 0) \right], \]
		where $\phi \in V_{\tau_1^-}$ and $k \in \widetilde{K_1}$. Note that $(\omega_\psi(k)\phi)(\cdot, 0, \ldots, 0)$ belongs to the underlying space of $\omega_\psi^{\flat, -} \boxtimes \chi$ (= that of $\omega_\psi^{\flat, -}$).
	\end{itemize}
	
	Moreover, $\sigma(\chi)$ (resp.\ $\sigma'(\chi)$) is rational in $\chi$.
\end{lemma}
\begin{proof}
	Consider the $+$ case first. The Iwasawa decomposition $G(F) = K_0 B^{\leftarrow}(F)$ implies
	\[ \Hom_{\widetilde{K_0}}(\tau_0, i_{\tilde{B}^{\leftarrow}}(\chi)) \simeq \Hom_{\widetilde{K_0}}\left(\tau_0, \Ind^{\widetilde{K_0}}_{\widetilde{K_0} \cap \tilde{B}^{\leftarrow}}(\chi) \right). \]
	This is in turn isomorphic to $\Hom_{\widetilde{K_0} \cap \tilde{B}^{\leftarrow}}(\tau_0, \chi)$. However
	\begin{equation*}
		K_0 \cap B^{\leftarrow}(F) = \underbracket{(K_0 \cap T(F))}_{= T(\mathfrak{o})} \underbracket{(K_0 \cap U^\leftarrow(F))}_{= U^\leftarrow(\mathfrak{o})},
	\end{equation*}
	hence
	\begin{equation*}
		\widetilde{K_0} \cap \tilde{B}^{\leftarrow} \simeq \underbracket{(\widetilde{K_0} \cap \tilde{T})}_{= T(\mathfrak{o}) \times \bmu_8} U^{\leftarrow}(\mathfrak{o}) \quad \text{(semi-direct product)}.
	\end{equation*}
	Here we split the covering over $U^{\leftarrow}(\mathfrak{o})$ (resp.\ $T(\mathfrak{o})$) using the canonical section over unipotent radicals (resp.\ Schrödinger model). No lattice models for $\omega_\psi$ are involved here, thus everything holds for $p=2$ as well.

	Let $\Xi$ be the genuine character of $\widetilde{K_0} \cap \tilde{B}^{\leftarrow}$ given by projection to $\bmu_8$ in the semi-direct product above. Since $\chi$ is unramified, we are led to $\Hom_{\widetilde{K_0} \cap \tilde{B}^{\leftarrow}}(\tau_0, \Xi)$. By passing to co-invariants and invoking Proposition \ref{prop:coinvariant-U-leftarrow}, we further reduce to
	\[ \Hom_{\widetilde{K_0} \cap \tilde{T}}(\Xi, \Xi) = \CC. \]
	
	The explicit generator $\sigma(\chi)$ matches $1 \in \CC$ under these isomorphisms. Its description is obtained by unraveling constructions: the restriction to $k \in \widetilde{K_0}$ is due to the Iwasawa decomposition, and evaluation at $(0, \ldots, 0)$ comes from Proposition \ref{prop:coinvariant-U-leftarrow}.
	
	The $-$ case follows from a similar reasoning, but the co-invariants under $K_1 \cap U^1(F)$ must be addressed by Proposition \ref{prop:coinvariant-U1} instead. Recall that $M^1 = \Sp(W^1) \times T^1$. In this way one obtains
	\begin{gather*}
		\widetilde{K_1} \cap \widetilde{P^1} = \underbracket{(\widetilde{K}_1 \cap \widetilde{M^1})}_{= \widetilde{K_1^{\flat}} \times T^1(\mathfrak{o})} (K_1 \cap U^1(F)) , \\
		\Hom_{\widetilde{K_1}}\left( \tau_1^-, i_{\tilde{P}^1}(\omega_\psi^{\flat, -} \boxtimes \chi) \right) \simeq \cdots \simeq \Hom_{\widetilde{K_1^\flat} \times T^1(\mathfrak{o})}\left( \tau_1^{\flat, -} \boxtimes \mathbf{1}, \omega_\psi^{\flat, -} \boxtimes \chi \right)
	\end{gather*}
	where
	\begin{itemize}
		\item $K_1^\flat$ is the avatar of $K_1$ for $\Sp(W^1)$,
		\item we split $\widetilde{K_1} \cap \tilde{M}^1$ into $\widetilde{K_1^\flat} \times T^1(\mathfrak{o})$ via Schrödinger model.
	\end{itemize}
	
	Since $\chi$ is unramified, we are finally led to $\Hom_{\widetilde{K_1^\flat}}(\tau_1^{\flat, -}, \omega_\psi^{\flat, -})$. This is described in Proposition \ref{prop:Weil-rep-M} (ii): the $1$-dimensional space generated by inclusion $\tau_1^{\flat, -} \hookrightarrow \omega_\psi^{\flat, -}$.
	
	The generator $\sigma'(\chi)$ is defined to match the inclusion map in $\Hom_{\widetilde{K_1^\flat}}(\tau_1^{\flat, -}, \omega_\psi^{\flat, -})$ through these isomorphisms. The precise description is obtained by unraveling definitions in the same way.
	
	Finally, the rationality of $\sigma(\chi)$ (resp.\ $\sigma'(\chi)$) is immediate from these concrete descriptions, after restricted to $\widetilde{K_0}$ (resp.\ $\widetilde{K_1}$).
\end{proof}

\subsection{Compatibility with Takeda--Wood isomorphism}
The results in the previous subsection have simple and well-known counterparts for $G^{\pm}$. Let $P_{\min}^\pm$ be the minimal reverse-standard parabolic subgroup of $G^{\pm}$ that matches $B^{\leftarrow}$ and $P^1$, respectively.

In particular, letting $M_{\min}^\pm$ be the Levi factor of $P_{\min}^{\pm}$, we have
\[ M_{\min}^+ \simeq \GL(1)^n, \quad M_{\min}^- \simeq \SO(V_1^-) \times \GL(1)^{n-1} \]
where $V_1^-$ is a $3$-dimensional quadratic $F$-vector space with discriminant $1$ and Hasse invariant $-1$. Their Iwahori--Hecke algebras decompose accordingly as tensor products.

Note that $M_{\min}^+$ (resp.\ the $\GL(1)^{n-1}$ in $M_{\min}^-$) is naturally isomorphic to $T$ (resp.\ $T^1$) on the symplectic side. In particular, the unramified characters on them can be matched.

Every character $\chi$ of $M_{\min}^+(F)$ (resp.\ $\GL(1, F)^{n-1}$) can be inflated to $P_{\min}^+(F)$ (resp.\ to $M_{\min}^-(F)$ and then to $P_{\min}^-(F)$). Therefore we have the normalized principal series representations $i_{P_{\min}^\pm}(\chi)$.

On the other hand, every unramified character $\chi$ also determines a character of $H^{M_{\min}^\pm}$ as follows.
\begin{itemize}
	\item In the $+$ case, this is straightforward since $H^{M_{\min}^+} \simeq \CC[X_*(T)]$.
	\item In the $-$ case, on the tensor slot $\CC[X_*(T^1)]$ of $H^{M_{\min}^-}$ it corresponds to $\chi$; on the tensor slot corresponding to $\SO(V_1^-)$ it corresponds to the trivial representation $\mathbf{1}$ of $\SO(V_1^-)$. We abbreviate this character as $\chi$ in what follows.
\end{itemize}

We have the analogue below of Proposition \ref{prop:principal-series-Hecke} for Iwahori--Hecke algebras.

\begin{proposition}\label{prop:principal-series-Hecke-2}
	There is an isomorphism of $H^\pm$-modules
	\[\begin{tikzcd}[row sep=tiny]
		i_{P_{\min}^\pm}(\chi)^{I^{\pm}} \arrow[r, "\sim"] & \Hom_{H^{M_{\min}^\pm}}(H^\pm, \chi) \\
		\sigma \arrow[mapsto, r] & {\left[ h \mapsto (h \cdot \sigma)(1_{G^\pm}) \right]}.
	\end{tikzcd}\]
\end{proposition}

Identify $H^{\pm}$ (resp.\ $H^{M_{\min}^+}$, $H^{M_{\min}^-}$) with $H_\psi^\pm$ (resp.\ $H_\psi^{\tilde{T}, +}$, $H_\psi^{\tilde{M}^1, -}$) by $\mathrm{TW}$. Hence
\begin{equation}\label{eqn:isom-Hecke-principal}\begin{aligned}
	 \Hom_{H_\psi^{\tilde{T}, +}}(H_\psi^+, \chi) & \simeq \Hom_{H^{M_{\min}^+}}(H^+, \chi), \\
	 \Hom_{H_\psi^{\tilde{M}^1, -}}\left( H_\psi^-, \mathbf{M}_{\tau_1^{\flat, -}}\left(\omega_\psi^{\flat, -}\right) \otimes \chi \right) & \simeq \Hom_{H^{M_{\min}^-}}(H^-, \chi).
\end{aligned}\end{equation}
In the second isomorphism above, Proposition \ref{prop:Weil-rep-M} is applied to identify $\mathbf{M}_{\tau_1^{\flat, -}}\left(\omega_\psi^{\flat, -}\right)$ with the trivial character $H^{\SO(V_1^-)} \to \CC$, normalized so that the inclusion $\tau_1^{\flat, -} \hookrightarrow \omega_\psi^{\flat, -}$ matches $1 \in \CC$.

Combining \eqref{eqn:isom-Hecke-principal} with Propositions \ref{prop:principal-series-Hecke} and \ref{prop:principal-series-Hecke-2}, we obtain
\begin{equation}\label{eqn:isom-rep-principal}\begin{aligned}
	\mathbf{M}_{\tau_0}\left( i_{\tilde{B}^{\leftarrow}}(\chi) \right) & \simeq i_{P^+_{\min}}(\chi)^{I^+}, \\
	\mathbf{M}_{\tau_1^-}\left( i_{\tilde{P}^1}(\omega_\psi^{\flat, -} \boxtimes \chi) \right) & \simeq i_{P^-_{\min}}(\chi)^{I^-}.
\end{aligned}\end{equation}

The following is the analogue of Lemma \ref{prop:sigma-generator}. Let $K^\pm \subset G^{\pm}(F)$ be as in \S\ref{sec:spherical-projector}. By \cite[2.2 Corollary]{Ca80}, if $\chi$ is an unramified character of $M_{\min}^+(F)$ (resp.\ $\GL(1, F)^{n-1}$), then
\[ \dim i_{P_{\min}^\pm}(\chi)^{K^\pm} = 1. \]
In fact, it has a canonical generator $\sigma^\pm(\chi)$ whose value at $1_{G^\pm}$ is $1$.

These generators match the $\sigma(\chi)$ and $\sigma'(\chi)$ from Lemma \ref{prop:sigma-generator} in the following sense.

\begin{proposition}\label{prop:s-matching}
	Let $\chi$ be an unramified character of $M_{\min}^+(F) \simeq T(F)$ (resp.\ $\GL(1, F)^{n-1} \simeq T^1(F)$), and identify it with a genuine unramified character of $\tilde{T}$ (resp.\ $\tilde{T}^1$). Denote by $s(\chi)$ (resp.\ $s'(\chi)$) the image of $\sigma(\chi)$ (resp.\ $\sigma'(\chi)$) under the isomorphism in Proposition \ref{prop:principal-series-Hecke}. Similarly we have $s^+(\chi)$ (resp.\ $s^-(\chi)$) under the isomorphism in Proposition \ref{prop:principal-series-Hecke-2}.
	
	The isomorphism \eqref{eqn:isom-Hecke-principal} maps
	\[ s(\chi) \mapsto s^+(\chi) \quad (\text{resp.}\; s'(\chi) \mapsto s^-(\chi)). \]
	In other words, $\sigma(\chi)$ (resp.\ $\sigma'(\chi)$) matches $\sigma^+(\chi)$ (resp.\ $\sigma^-(\chi)$) under \eqref{eqn:isom-rep-principal}.
\end{proposition}
\begin{proof}
	Consider the $+$ case first. Since $e^+ = \mathrm{TW}\left(\mes(K^+)^{-1} \mathbf{1}_{K^+}\right)$ by Proposition \ref{prop:matching-e}, the preimage of $s^+(\chi)$ must be proportional to $s(\chi)$; indeed, they both correspond to elements of the line $\{ \sigma : e^+ \sigma = \sigma \}$ inside $\mathbf{M}_{\tau_0}(i_{\tilde{B}^{\leftarrow}}(\chi))$.
	
	By the matching between spherical projectors, it suffices to show that
	\begin{itemize}
		\item $s(\chi) \in \Hom_{H_\psi^{\tilde{T}, +}}(H_\psi^+, \chi)$ satisfies $s(\chi)(e^+) = 1$, and
		\item $s^+(\chi) \in \Hom_{H^{M^+_{\min}}}(H^+, \chi)$ satisfies $s^+(\chi)\left( \mes(K^+)^{-1} \mathbf{1}_{K^+} \right) = 1$.
	\end{itemize}

	Since $\mes(K^+)^{-1} \mathbf{1}_{K^+} \cdot \sigma^+(\chi) = \sigma^+(\chi)$, the case of $s^+(\chi)$ is immediate from the formula in Proposition \ref{prop:principal-series-Hecke-2} applied to $\sigma = \sigma^+(\chi)$.
	
	Now evaluate $s(\chi)(e^+)$ using the formulas of Proposition \ref{prop:principal-series-Hecke} (with $\sigma = \sigma(\chi)$) and Lemma \ref{prop:sigma-generator}: taking $\phi \in V_{\tau_0}$ with $\phi(0, \ldots, 0) = 1$, we obtain
	\[ (e^+ \cdot \sigma(\chi))(\phi)(1_{\tilde{G}}) = \sigma(\chi)(\phi)(1_{\tilde{G}}) = \phi(0, \ldots, 0) \]
	which also yields $1$, as desired.
	
	The strategy for the $-$ case is the same. It is still easy to deduce $s^-(\chi)\left( \mes(K^-)^{-1} \mathbf{1}_{K^-} \right) = 1$. On the other hand, $s'(\chi)(e^-)$ equals
	\[ \overline{\phi} \mapsto (e^- \cdot \sigma'(\chi))(\phi)(1_{\tilde{G}}) = \sigma'(\chi)(\phi)(1_{\tilde{G}}) = \phi(\cdot, 0, \ldots, 0) = \overline{\phi}. \]
	As elements of the underlying space of $\mathbf{M}_{\tau_1^{\flat, -}}(\omega_\psi^{\flat, -}) \otimes \chi$, the above equals the inclusion map $\tau_1^{\flat, -} \hookrightarrow \omega_\psi^{\flat, -}$. The latter matches the $1$ we obtained on the $G^-$ side; see the discussions after \eqref{eqn:isom-Hecke-principal}.
\end{proof}

\section{Identification of intertwining operators: the even case}\label{sec:int-op-even}
Let $\tilde{G} = \Mp(W)$ as before, with $n := \frac{1}{2} \dim W \geq 1$.

\subsection{Statement of the main result}\label{sec:statement-int-op-even}
To begin with, we describe a canonical way to assign standard representatives $\tilde{w} \in \widetilde{K_0}$ to $w \in \Omega_0$.

Denote the $B^{\leftarrow}$-simple roots $2\epsilon_1, \epsilon_2 - \epsilon_1, \ldots, \epsilon_n - \epsilon_{n-1} \in X^*(T)$ (resp.\ the corresponding reflections) as $\beta_1, \ldots, \beta_n$ (resp.\ $t_1, \ldots, t_n$), in this order.
\begin{itemize}
	\item (Short roots) For $i > 1$, we take the representative of $t_i$ to be
	\[ \tilde{t}_i := \tilde{x}_{\beta_i}(1) \tilde{x}_{-\beta_i}(-1) \tilde{x}_{\beta_i}(1). \]
	\item (Long root) For $t_1$, we take the representative $\tilde{t}_1$ by scaling $\tilde{x}_{\beta_1}(1) \tilde{x}_{-\beta_1}(-1) \tilde{x}_{\beta_1}(1)$, so that in the Schrödinger model $\omega_\psi(\tilde{t}_1)$ is the unitary Fourier transform in the first coordinate.
\end{itemize}

For a general element $w \in \Omega_0$, take any reduced expression $w = t_{i_1} \cdots t_{i_\ell}$ and set
\[ \tilde{w} := \tilde{t}_{i_1} \cdots \tilde{t}_{i_\ell}. \]
This is justified by the following fact.

\begin{lemma}[See {\cite[Proposition 4.2]{GL18}}]
	The representative $\tilde{w}$ is independent of the choice of reduced expressions. 
\end{lemma}

Also note that $\tilde{w} \in \widetilde{K_0}$ by construction.

\begin{definition}\label{def:int-op-even}
	Let $P = MU$ be a parabolic subgroup of $G$ with $P \supset B^{\leftarrow}$, and $w \in \Omega_0$ satisfies ${}^w M = M$. For every $\pi$ in $\tilde{M}\dcate{Mod}$, define $J_w(\pi) \in \Hom_{\tilde{G}}\left( i_{\tilde{P}}(\pi), i_{\tilde{P}}({}^{\tilde{w}} \pi)\right)$ as the composition
	\[ i_{\tilde{P}}(\pi) \xrightarrow{J_{\tilde{P}^w|\tilde{P}}(\pi)} i_{\tilde{P}^w}(\pi) \to i_{\tilde{P}}\left({}^{\tilde{w}} \pi\right). \]
	\begin{itemize}
		\item The first arrow is the standard intertwining operator \cite[\S 2.4]{Li12b}: given $P = MU$ and $P' = MU'$, the operator $J_{\tilde{P}'|\tilde{P}}(\pi)$ maps a section $f$ to
		\[ g \mapsto \int_{U(F) \cap U'(F) \backslash U'(F)} f(ug) \dd u. \]
		\item We fix the Haar measure on $V(F)$ for each unipotent subgroup $V$ involved in $J_{\tilde{P}^w|\tilde{P}}(\pi)$, such that $\mes(V(F) \cap K_0) = 1$.
		\item The second arrow sends a section $f$ to $x \mapsto f(\tilde{w}^{-1} x)$.
	\end{itemize}
	We view $J_w(\pi \otimes \chi)$ is a rational family when $\chi$ ranges over unramified characters of $M(F)$.
\end{definition}

Now apply this to the unramified principal series $i_{\tilde{B}^{\leftarrow}}(\chi)$, where $\chi$ is an unramified character of $T(F)$, identified as a genuine unramified character of $\tilde{T}$. The standard shorthand $w\chi$ for ${}^{\tilde{w}} \chi$ can be used, since the latter does not depend on the choice of representatives.

Observe that $J_w(\chi)$ induces a homomorphism of $H_\psi^+$-modules after applying $\mathbf{M}_{\tau_0}$. By abuse of notation, it is denoted by the same symbol as
\[ J_w(\chi): \mathbf{M}_{\tau_0}\left( i_{\tilde{B}^{\leftarrow}}(\chi)\right) \to \mathbf{M}_{\tau_0}\left( i_{\tilde{B}^{\leftarrow}}(w\chi)\right). \]
Again, this should be viewed as a rational family of operators when $\chi$ varies. By $H_\psi^+$-linearity, $J_w(\chi)$ restricts to a map
\[ \Hom_{\widetilde{K_0}}\left(\tau_0, i_{\tilde{B}^{\leftarrow}}(\chi)\right) \to \Hom_{\widetilde{K_0}}\left(\tau_0,  i_{\tilde{B}^{\leftarrow}}(w\chi)\right) \]
as both sides are cut out by the idempotent $e^+$.

Consider now the $G^+$-side. In what follows, we use $\Omega_{\mathrm{aff}} \simeq \Omega^+_{\mathrm{aff}}$ to identify $\Omega_0$ with $\Omega^+_0$.

With the same convention on Haar measures, replacing $K_0 \subset G(F)$ by $K^+ \subset G^+(F)$, one has the standard intertwining operator
\[ J^+_w(\chi): i_{P_{\min}^+}(\chi)^{I^+} \to i_{P_{\min}^+}(w\chi)^{I^+}. \]
Again, it restricts to $i_{P_{\min}^+}(\chi)^{K^+} \to i_{P_{\min}^+}(w\chi)^{K^+}$.

\begin{definition}\label{def:tw-number}
	For all $w \in \Omega_0$, let $t(w) \in \{0, \ldots, n\}$ be the number of flips in $w$, i.e.\ the number of positive roots $2\epsilon_i$ such that $w(2\epsilon_i)$ is negative.
\end{definition}

\begin{lemma}\label{prop:tw-number}
	Let $w \in \Omega_0$ and take any reduced expression $w = t_{i_1} \cdots t_{i_\ell}$. Then $t(w)$ equals the number of $k$'s such that $i_k = 1$.
\end{lemma}
\begin{proof}
	It is clear that for $w$ to have $t(w)$ flips, the reflection $t_1$ must be used at least $t(w)$ times. On the other hand, every $w \in \Omega_0$ can be expressed as $t_{j_1} \cdots t_{j_h}$, with $t_1$ occurring exactly $t(w)$ times. When we pass to a reduced expression from $t_{j_1} \cdots t_{j_h}$, the number of $t_1$'s do not increase, as one sees from the braid relations, cf.\ the proof of \cite[Proposition 4.2]{GL18}.
\end{proof}

\begin{theorem}\label{prop:comparison-J}
	Let $w \in \Omega_0$. Using Propositions \ref{prop:principal-series-Hecke} and \ref{prop:principal-series-Hecke-2}, we transport $J_w(\chi)$ and $J^+_w(\chi)$ to obtain
	\begin{align*}
		\mathcal{J}_w(\chi): \Hom_{H_\psi^{\tilde{T}, +}}(H_\psi^+, \chi) & \to \Hom_{H_\psi^{\tilde{T}, +}}(H_\psi^+, w\chi), \\
		\mathcal{J}^+_w(\chi): \Hom_{H^{M_{\min}^+}}(H^+, \chi) & \to \Hom_{H^{M_{\min}^+}}(H^+, w\chi),
	\end{align*}
	which are homomorphisms of $H_\psi^+$ and $H^+$-modules, respectively. They can be compared via the identification \eqref{eqn:isom-Hecke-principal}, and this gives
	\[ \mathcal{J}_w(\chi) = |2|_F^{t(w)/2} \cdot \mathcal{J}^+_w(\chi) \]
	as an equality between rational families in $\chi$.
\end{theorem}

Since both sides of $\mathcal{J}_w$ and $\mathcal{J}^+_w$ are irreducible for generic $\chi$, we know that they differ by a rational function in $\chi$, a priori. To determine this rational function, it suffices to compare
\begin{itemize}
	\item the ratio between $\mathcal{J}_w(s(\chi))$ and $s(w\chi)$ (= that between $J_w(\sigma(\chi))$ and $\sigma(\chi)$),
	\item the ratio between $\mathcal{J}^+_w(s^+(\chi))$ and $s^+(w\chi)$ (= that between $J^+_w(\sigma^+(\chi))$ and $\sigma^+(w\chi)$),
\end{itemize}
where $s(\chi)$ and $s^+(\chi)$ (resp.\ $\sigma(\chi)$ and $\sigma^+(\chi)$) are as in Proposition \ref{prop:s-matching} (resp.\ in Lemma \ref{prop:sigma-generator} and before Proposition \ref{prop:s-matching}). All in all, Theorem \ref{prop:comparison-J} will follow from the
\begin{theorem}\label{prop:GK}
	Let $w \in \Omega_0$. Let $c(w, \chi)$ and $d(w, \chi)$ be the rational functions in $\chi$ determined by
	\begin{align*}
		J^+_w(\sigma^+(\chi)) & = c(w, \chi) \sigma^+(w\chi), \\
		J_w(\chi)(\sigma(\chi)) & = d(w, \chi) \sigma(w\chi).
	\end{align*}
	Then $d(w, \chi) = |2|_F^{t(w)/2} c(w, \chi)$.
\end{theorem}

Note that the ratio between $c(w, \chi)$ is expressed by the Gindikin--Karpelevich formula \cite[Theorem 3.1]{Ca80} for $G^+$.
This can be viewed as a variant for $\tilde{G}$ of Gindikin--Karpelevich formula. Its proof will be completed in \S\ref{sec:integral}.

\subsection{Preliminary reductions}
Let $\chi$ be a genuine unramified character of $\tilde{T}$. For each $\phi \in V_{\tau_0}$, the formula in Lemma \ref{prop:sigma-generator} says that $\sigma(\chi)(\phi)$ maps $k \in \widetilde{K_0}$ to $(\omega_\psi(k)\phi)(0, \ldots, 0)$, and is characterized by this. 

Let us write $U := U^{\leftarrow}$. We infer that
\begin{equation}
	J_w(\chi)(\sigma(\chi))(\phi): k \mapsto \int_{U(F) \cap U^w(F) \backslash U^w(F)}  \sigma(\chi)(\phi)(u \tilde{w}^{-1} k) \dd u, \quad k \in \widetilde{K_0},
\end{equation}
for $\chi$ in the convergence range, and this equals $d(w, \chi) (\omega_\psi(k)\phi)(0, \ldots, 0)$.

\begin{lemma}\label{prop:d-formula}
	With the notation above, $d(w, \chi)$ is characterized by
	\[ d(w, \chi) (\omega_\psi(\tilde{w}) \phi)(0, \ldots, 0) = \int_{U(F) \cap U^w(F) \backslash U^w(F)} \sigma(\chi)(\phi)(u) \dd u \]
	for all $\phi \in V_{\tau_0}$, and $\chi$ in the convergence range.
\end{lemma}
\begin{proof}
	Since $\tilde{w} \in \widetilde{K_0}$, one can put $k = \tilde{w}$ in the discussion above.
\end{proof}

Take any reduced expression $w = t_{i_1} \cdots t_{i_\ell}$. For each $1 \leq i \leq n$, let $G_i \subset G$ be the subgroup generated by $T$ and the root subgroups of $\pm\beta_i$. Note that:
\begin{description}
	\item[Long root] If $i = 1$ we have
	\[ \widetilde{G_1} \simeq \Mp(2) \times \text{torus}, \]
	\item[Short roots] If $i > 1$ then
	\[ \widetilde{G_i} \simeq \GL(2, F) \times \text{torus} \times \bmu_8. \]
	Indeed, in this case $G_i$ lies in a Siegel parabolic subgroup, hence $\widetilde{G_i}$ splits via \eqref{eqn:Levi-splitting}.
\end{description}

Since $w = t_{i_1} \cdots t_{i_\ell}$ is a reduced expression, $J_w(\chi)$ decomposes into
\begin{equation}\label{eqn:J-decomp-even}
	J_w(\chi) = J_{i_1}(t_{i_2} \cdots t_{i_\ell} \chi) \cdots J_{i_{\ell - 1}}(t_{i_\ell} \chi) J_{i_\ell}(\chi)
\end{equation}
where $J_k(\cdots)$ is deduced from the standard intertwining operator for $\widetilde{G_k}$ by functoriality.

There is a similar factorization of $J^+_w(\chi)$. Note that the root subgroups involves in $J_w(\chi)$ and $J^+_w(\chi)$ are in natural bijection; the root subgroups in question are actually isomorphic as $\mathfrak{o}$-group schemes, hence our choices of Haar measures are also matched.

\begin{lemma}\label{prop:GK-reduction}
	The assertion in Theorem \ref{prop:GK} holds if it holds for $\tilde{G}= \Mp(2)$ and $w$ being the nontrivial element in $\Omega_0$. 
\end{lemma}
\begin{proof}
	Given $w \in \Omega_0$, decompose $J_w(\chi)$ and $J^+_w(\chi)$ in the manner above. Accordingly, $c(w, \chi)$ (resp.\ $d(w, \chi)$) factorizes into simple ones. Moreover, $c(t_i, \cdot)$ can be computed inside the subgroup $\widetilde{G_i}$ for all $1 \leq i \leq n$. In view of Lemma \ref{prop:tw-number} that deals with the extra term $|2|_F^{t(w)/2}$, it remains to show that:
	\begin{description}
		\item[Short roots] $d(t_i, \cdot) = c(t_i, \cdot)$ when $i > 1$.
		\item[Long root] $d(t_1, \cdot)$ equals its avatar defined for $\Mp(2)$.
	\end{description}

	For the short roots, consider the identity in Lemma \ref{prop:d-formula} with $w = t_i$ and $i > 1$. It follows from the explicit formula \eqref{eqn:Siegel-action} for $\omega_\psi$ that the left hand side is $d(t_i, \chi) \phi(0, \ldots, 0)$. For the right hand side, we have to perform an Iwasawa decomposition to evaluate $\sigma(\chi)(\phi)(u)$; however the $u$ in the integral belongs to $\GL(2, F) \subset \widetilde{G_i}$. The Iwasawa decomposition can also be performed within $\GL(2, F)$; by \eqref{eqn:Siegel-action} once again, the result is $\phi(0, \ldots, 0)$ times the intertwining integral for $\GL(2)$ expressing $c(t_i, \chi)$.
	
	For the long root, we apply Lemma \ref{prop:d-formula} with $w = t_1$ and $\phi = \phi_1 \otimes \cdots \otimes \phi_n$, such that $\phi_j \in \Schw(\mathfrak{o}/(2))$ for all $j$ and $\phi_2(0) = \cdots = \phi_n(0) = 1$. Now everything occurs within the $\Mp(2)$ factor of $\widetilde{G_1}$, leaving $\phi_2, \ldots, \phi_n$ unaffected. Since $\phi_1$ is arbitrary, it follows that $d(t_1, \chi)$ equals its avatar for $\Mp(2)$.
\end{proof}

To prove Theorem \ref{prop:GK}, we are reduced to showing that for $\tilde{G} = \Mp(2)$, i.e.\ $n=1$, and $w = t_1$ being the nontrivial Weyl group element, we have
\[ d(w, \chi) = |2|_F^{\frac{1}{2}} c(w, \chi). \]

Let us identify $G$ with $\SL(2)$. Lemma \ref{prop:d-formula} becomes
\begin{equation}\label{eqn:d-prep}
	d(w, \chi) (\omega_\psi(\tilde{w})\phi)(0) = \int_F (\sigma(\chi)\phi) \twobigmatrix{1}{}{t}{1} \dd t, \quad \phi \in \Schw(\mathfrak{o}/(2)).
\end{equation}

As usual, here $\twomatrix{1}{}{*}{1}$ is lifted to $\tilde{G} = \widetilde{\SL}(2, F)$. To proceed, we employ an explicit Iwasawa decomposition of $\twomatrix{1}{}{t}{1}$ for all $t \in F$.

\begin{itemize}
	\item If $t \in \mathfrak{o}$, then $\twomatrix{1}{}{t}{1} \in K_0$.

	\item If $t \in F \smallsetminus \mathfrak{o}$, on the level of $\SL(2, F)$ we have
	\begin{equation}\label{eqn:matrix-factorization-1}
		\twobigmatrix{1}{}{t}{1} =
		\underbracket{\twobigmatrix{1}{\frac{1}{t}}{}{1}}_{\in U(F)}
		\underbracket{\twobigmatrix{\frac{1}{t}}{}{}{t}}_{\in T(F)}
		\underbracket{\twobigmatrix{}{1}{-1}{} \twobigmatrix{-1}{}{}{-1} 
		\twobigmatrix{1}{\frac{1}{t}}{}{1}}_{\in K_0}.
	\end{equation}
\end{itemize}

Let us lift \eqref{eqn:matrix-factorization-1} to $\tilde{G}$. Note that $\tilde{w}$ has image $\twomatrix{}{1}{-1}{}$ in $\SL(2, F)$.

\begin{lemma}\label{prop:d-Fourier}
	For $n=1$ and the nontrivial Weyl group element $w$, the representative $\tilde{w} \in \tilde{G}$ acts on $\Schw(\mathfrak{o}/(2))$ by $|2|_F^{1/2}$ times the Fourier transform with respect to the counting measure on $\mathfrak{o}/(2)$.
\end{lemma}
\begin{proof}
	By construction, $\omega_\psi(\tilde{w})$ is the unitary Fourier transform. The same is true after restriction to $\Schw(\mathfrak{o}/(2))$.
	
	Hence there exists $a \in \R^{\times}_{> 0}$ such that $\omega_\psi(\tilde{w})$ maps every $\phi \in \Schw(\mathfrak{o}/(2))$ to
	\[ y \mapsto a \sum_{u \in \mathfrak{o}/(2)} \psi(2uy) \phi(u), \quad y \in \mathfrak{o}/(2). \]
	
	Take $\mathbf{1}_{2\mathfrak{o}} \in \Schw(\mathfrak{o}/(2))$, i.e.\ the Dirac function at $2\mathfrak{o} \in \mathfrak{o}/(2)$. It is mapped by $\omega_\psi(\tilde{w})$ to $a \cdot \mathbf{1}_{\mathfrak{o}}$, and then to $a^2 |\mathfrak{o}/(2)| \mathbf{1}_{2\mathfrak{o}}$. One determines $a$ by noting $|\mathfrak{o}/(2)| = |2|_F^{-1}$.
\end{proof}

We record a standard by-product of the result above.

\begin{corollary}\label{prop:Fourier-o}
	Let $\mu$ be the self-dual Haar measure on $F$. Then $\mu(\mathfrak{o}) = |2|_F^{-1/2}$.
\end{corollary}
\begin{proof}
	By Lemma \ref{prop:d-Fourier},
	\[ \mu(\mathfrak{o}) = (\omega_\psi(\tilde{w}) \mathbf{1}_{\mathfrak{o}})(0) = |2|_F^{1/2} \sum_{x \in \mathfrak{o}/(2)} 1 = |2|_F^{1/2} |2|_F^{-1} \]
	since the unitary Fourier transform is defined via $\mu$.
\end{proof}

\begin{lemma}\label{prop:d-lifting}
	Assume $n=1$. Denote by $s: \twomatrix{*}{}{}{*} \to \tilde{G}$ the section \eqref{eqn:Levi-splitting} furnished by Schrödinger model. Let $t \in F \smallsetminus \mathfrak{o}$, and define $\zeta \in \bmu_8$ to be the element such that the identity
	\[ \tilde{x}_{-2\epsilon_1}(t) = \zeta \cdot	\tilde{x}_{2\epsilon_1}\left(\frac{1}{t}\right)
	s\twobigmatrix{\frac{1}{t}}{}{}{t}
	\tilde{w}
	s\twobigmatrix{-1}{}{}{-1}
	\tilde{x}_{2\epsilon_1}\left( \frac{1}{t} \right) \]
	holds in $\tilde{G}$; note that it lifts the equation \eqref{eqn:matrix-factorization-1} in $G(F)$. Then $\zeta = \gamma_\psi(-2t)$.
\end{lemma}
\begin{proof}
	Compare the effects of both sides on $\mathbf{1}_{2\mathfrak{o}}$. The formulas \eqref{eqn:root-action} and \eqref{eqn:Siegel-action} imply
	\[ \omega_\psi\left(\tilde{x}_{2\epsilon_1}\left( \frac{1}{t} \right)\right) \mathbf{1}_{2\mathfrak{o}} = \mathbf{1}_{2\mathfrak{o}}, \quad s\twobigmatrix{-1}{}{}{-1} \mathbf{1}_{2\mathfrak{o}} = \mathbf{1}_{2\mathfrak{o}}, \]
	as well as
	\[ \omega_\psi(\tilde{w})(\mathbf{1}_{2\mathfrak{o}}) \in \R^{\times}_{> 0} \mathbf{1}_{\mathfrak{o}}, \quad \omega_\psi\left(s\twobigmatrix{\frac{1}{t}}{}{}{t}\right) \mathbf{1}_{\mathfrak{o}} \in \R^{\times}_{> 0} \mathbf{1}_{t\mathfrak{o}}. \]
	Finally $\omega_\psi\left( \tilde{x}_{2\epsilon_1}\left(\frac{1}{t}\right)\right) \mathbf{1}_{t\mathfrak{o}}$ is the function
	\[ x \mapsto \psi\left( \frac{x^2}{t} \right) \mathbf{1}_{t\mathfrak{o}}(x), \quad x \in F. \]
	
	On the other hand, $\omega_\psi(\tilde{x}_{-2\epsilon_1}(t)) = \omega_\psi(\tilde{w})^{-1} \omega_\psi(\tilde{x}_{2\epsilon_1}(-t)) \omega_\psi(\tilde{w})$ sends $\mathbf{1}_{2\mathfrak{o}}$ to
	\[ \R^{\times}_{> 0} \omega_\psi(\tilde{w})^{-1} \left[ x \mapsto \psi(-tx^2) \mathbf{1}_{\mathfrak{o}}(x) \right]. \]
	
	Evaluate the function above (an inverse Fourier transform) at $0$ to obtain $\R^\times_{> 0} \cdot \int_{\mathfrak{o}} \psi(-ty^2) \dd y$. We claim that this equals $\gamma_\psi(-2t)$ up to $\R^{\times}_{> 0}$. Indeed, since $\mathfrak{o} \supset \frac{2}{t}\mathfrak{o}$ on which $y \mapsto \psi(-ty^2)$ is identically $1$, one can apply \cite[\S 27]{Weil64} to evaluate this integral to get $\R_{> 0}^{\times} \gamma_\psi(-2t)$.
	
	Since $\gamma_\psi(-2t) \in \bmu_8$, we may now compare the two sides of the assertion applied to $\mathbf{1}_{2\mathfrak{o}}$, by evaluating them at $0$, to conclude $\zeta = \gamma_\psi(-2t)$.
\end{proof}

The assertion above can also be proved by a computation via Kubota's cocycle. For the relation between Kubota's cocycle and Schrödinger's model, see also the discussions before \cite[Proposition 9.1.1]{Li20}.

\begin{lemma}\label{prop:rank-one-integrals}
	Assume $n=1$ and let $w$ be the nontrivial element of the Weyl group. Set
	\[ z := \chi\twobigmatrix{\varpi}{}{}{\varpi^{-1}} = \chi\left(\check{\beta}_1(\varpi)\right). \]
	
	When $|z| < 1$, one can expand $d(w, \chi)$ into a convergence sum $\sum_{k=0}^\infty r_k z^k$. The sequence $(r_k)_{k=0}^\infty$ is characterized by the property that
	\begin{align*}
		\int_{\mathfrak{o}} \omega_\psi\twobigmatrix{1}{}{t}{1} (\phi)(0) \dd t & = r_0 \cdot (\omega_\psi(\tilde{w})\phi)(0), \\
		|\varpi|_F^k \int_{\varpi^{-k} \mathfrak{o}^{\times}} \gamma_\psi(-2t) \omega_\psi(\tilde{w}) \omega_\psi\twobigmatrix{1}{\frac{1}{t}}{}{1}(\phi)(0) \dd t & = r_k \cdot (\omega_\psi(\tilde{w})\phi)(0), \quad k \geq 1
	\end{align*}
	for all $\phi \in \Schw(\mathfrak{o}/(2))$.
\end{lemma}
\begin{proof}
	Assume $|z| < 1$. First of all, we prove that for all $\phi \in \Schw(\mathfrak{o}/(2))$,
	\begin{multline}\label{eqn:rank-one-integrals}
		d(w, \chi) \underbracket{|2|_F^{\frac{1}{2}} \sum_{x \in \mathfrak{o}/(2)} \phi(x)}_{= (\omega_\psi(\tilde{w})\phi)(0)} = \int_{\mathfrak{o}} \omega_\psi\twobigmatrix{1}{}{t}{1} (\phi)(0) \dd t \\
		+ \sum_{k=1}^\infty z^k |\varpi|_F^k \int_{\varpi^{-k} \mathfrak{o}^{\times}} \gamma_\psi(-2t) \omega_\psi(\tilde{w}) \omega_\psi\twobigmatrix{1}{\frac{1}{t}}{}{1}(\phi)(0) \dd t .
	\end{multline}
	The sum is convergent and the Haar measure on $F$ is chosen so that $\mes(\mathfrak{o}) = 1$.
	
	Indeed, this follows by first evaluating the left hand side of \eqref{eqn:d-prep} by Lemma \ref{prop:d-Fourier}, and then the right hand by Lemma \ref{prop:d-lifting}, observing that $s\twomatrix{-1}{}{}{-1}$ acts trivially on $\Schw(\mathfrak{o}/(2))$ and $\sigma(\chi)$ can be accessed via Lemma \ref{prop:sigma-generator}. The factor $|\varpi|_F^k$ comes from $\delta_{B^{\leftarrow}}^{1/2}$. The convergence when $|z| < 1$ is a well-known fact about standard intertwining operators.
	
	Since $J_w(\chi)$ is known to be rational in $\chi$ and converges when $|z| < 1$, by expanding both sides of \eqref{eqn:rank-one-integrals}, we obtain the expansion $d(w, \chi) = \sum_{k=0}^\infty r_k z^k$ with the required property. This property characterizes $(r_k)_{k=0}^\infty$ by taking $\phi$ such that $(\omega_\psi(\tilde{w})\phi)(0) \neq 0$.
\end{proof}

\subsection{Evaluation of \texorpdfstring{$p$}{p}-adic integrals}\label{sec:integral}
The goal here is to evaluate the integrals in Lemma \ref{prop:rank-one-integrals}, with suitable choices of $\phi$, in order to determine $d(w, \chi)$.

\begin{lemma}\label{prop:main-integral-0}
	With the notation of Lemma \ref{prop:rank-one-integrals}, we have $r_0 = |2|_F^{1/2}$.
\end{lemma}
\begin{proof}
	Taking $\phi := \mathbf{1}_{\mathfrak{o}}$, we will show that
	\[ \int_{\mathfrak{o}} \omega_\psi\twobigmatrix{1}{}{t}{1} (\mathbf{1}_{\mathfrak{o}})(x) \dd t = \mathbf{1}_{\mathfrak{o}}(x) \]
	for all $x$. The assertion then follows by taking $x=0$ and noting that $|2|_F^{1/2} \sum_{y \in \mathfrak{o}/(2)} \mathbf{1}_{\mathfrak{o}}(y) = |2|_F^{1/2} |\mathfrak{o}/(2)| = |2|_F^{-1/2}$, implying $r_0 = |2|_F^{1/2}$.
	
	Indeed, we have $\omega_\psi(\tilde{w})\mathbf{1}_{\mathfrak{o}} = c \mathbf{1}_{2\mathfrak{o}}$ for some $c > 0$. For all $t \in \mathfrak{o}$, the formulas for $\omega_\psi$ imply
	\begin{align*}
		\omega_\psi\twobigmatrix{1}{}{t}{1} \mathbf{1}_{\mathfrak{o}} & = \omega_\psi(\tilde{w})^{-1} \omega_\psi\twobigmatrix{1}{-t}{}{1} \omega_\psi(\tilde{w}) \mathbf{1}_{\mathfrak{o}} \\
		& = c \omega_\psi(\tilde{w})^{-1} \omega_\psi\twobigmatrix{1}{-t}{}{1} \mathbf{1}_{2\mathfrak{o}} \\
		& = c \omega_\psi(\tilde{w})^{-1} \mathbf{1}_{2\mathfrak{o}} = \mathbf{1}_{\mathfrak{o}}.
	\end{align*}
	Since $\mes(\mathfrak{o}) = 1$, the claim follows by integrating over $t \in \mathfrak{o}$.
\end{proof}

\begin{lemma}\label{prop:main-integral}
	For all $k \geq 1$, we have
	\[ \int_{\mathfrak{o}^{\times}} \gamma_\psi\left( 2\varpi^{-k} a\right) \dd a = \begin{cases}
		(1 - q^{-1}) |2|_F^{\frac{1}{2}}, & k \;\text{is even} \\
		0, & k \;\text{is odd}
	\end{cases}\]
	where the measure on $\mathfrak{o}^{\times}$ is induced from the Haar measure on $F$ with $\mes(\mathfrak{o}) = 1$.
\end{lemma}
\begin{proof}
	Since $\gamma_\psi(b)$ depends only on $b \bmod F^{\times 2}$, henceforth we will assume $k \in \{0, 1\}$.
	
	Since $a$ is constrained in the compact set $\mathfrak{o}^{\times}$, for $r \gg 0$ (uniformly in $a$ and $k$), we have
	\[ \gamma_\psi\left(2\varpi^{-k} a\right) = C_k \int_{\varpi^{-r} \mathfrak{o}} \psi\left( \varpi^{-k} a x^2 \right) \dd x, \]
	where $C_k \in \R_{> 0}$ and the Haar measure satisfies $\mes(\mathfrak{o}) = 1$. Moreover,
	\begin{equation}\label{eqn:p-adic-integral-0}
		C_0 = |2|_F^{-\frac{1}{2}}.
	\end{equation}

	Indeed, we actually have
	\[ \gamma_\psi\left(2 \varpi^{-k} a \right) = \int_{\varpi^{-r} \mathfrak{o}} \psi\left( \varpi^{-k} a x^2 \right) \dd\mu_{\varpi^{-k} a}(x) \]
	where $\mu_{\varpi^{-k} a}$ is the self-dual Haar measure on $F$ relative to $\psi'(x) = \psi(\varpi^{-k} ax)$. When $k=0$, this is the same as the self-dual Haar measure relative to $\psi$. One takes $C_k := \mu_{\varpi^{-k} a}(\mathfrak{o})$ and applies Corollary \ref{prop:Fourier-o} to reach \eqref{eqn:p-adic-integral-0}.
	
	Hence
	\begin{equation}\label{eqn:p-adic-integral-1}
		\int_{\mathfrak{o}^{\times}} \gamma_\psi\left( 2\varpi^{-k} a\right) \dd a = C_k \int_{\varpi^{-r} \mathfrak{o} \smallsetminus \{0\}} \int_{\mathfrak{o}^{\times}} \psi\left( \varpi^{-k} a x^2 \right) \dd a \dd x.
	\end{equation}
	
	Fix $x \in F \smallsetminus \{0\}$. Let $m \in \Z_{\geq 1}$, we claim that when $2m \geq 2e + k - 2\mathrm{val}_F(x)$, the map
	\[ t \mapsto \dfrac{\psi(\varpi^{-k} ax^2 t)}{\psi\left( \varpi^{-k} ax^2 \right)}, \quad t \in 1 + \varpi^m \mathfrak{o} \]
	is a multiplicative character. Indeed, put $c := \varpi^{-k} ax^2$ and set $\psi_c(t) = \psi(ct)$. The additive character $\psi_c$ is trivial on $\{t: \mathrm{val}_F(t) \geq 2e + k - 2\mathrm{val}_F(x)\}$, but nontrivial on $\{t: \mathrm{val}_F(t) \geq 2e + k - 2\mathrm{val}_F(x) - 1\}$. For all $x, y \in \mathfrak{o}$, we deduce
	\begin{multline*}
		\dfrac{\psi_c((1 + \varpi^m x)(1 + \varpi^m y))}{\psi_c(1)} = \dfrac{\psi_c(1 + \varpi^m(x+y) + \varpi^{2m}xy)}{\psi_c(1)} \\
		= \psi_c(\varpi^m(x+y)) = \psi_c(\varpi^m x) \psi_c(\varpi^m y) = \dfrac{\psi_c(1 + \varpi^m x) \psi_c(1 + \varpi^m y)}{\psi_c(1) \psi_c(1)}.
	\end{multline*}
	The claim follows at once.
	
	On the other hand, those $m$ for which $t \mapsto \frac{\psi_c(t)}{\psi_c(1)}$ is trivial on $1 + \varpi^m \mathfrak{o}$ is characterized by $m \geq 2e + k - 2\mathrm{val}_F(x)$.
	
	Returning to \eqref{eqn:p-adic-integral-1}, when $m \geq 1$ is given, we decompose
	\[ \mathfrak{o}^{\times} = \bigsqcup_{\overline{a} \in (\mathfrak{o}/(\varpi^m))^{\times}} a(1 + \varpi^m \mathfrak{o}), \quad a: \;\text{representative of}\; \overline{a}. \]
	
	When $x$ is kept fixed, the inner integral in \eqref{eqn:p-adic-integral-1} is thus
	\begin{equation}\label{eqn:p-adic-integral-2}
		\sum_{\overline{a} \in (\mathfrak{o}/(\varpi^m))^{\times}} \int_{1 + \varpi^m \mathfrak{o}} \psi(\varpi^{-k} a x^2 t) \dd t.
	\end{equation}
	
	Consider the case $k=1$ first.
	\begin{itemize}
		\item If there exists $m \geq 1$ such that $m < 2e + 1 - 2\mathrm{val}_F(x) \leq 2m$, namely if $2e + 1 -2\mathrm{val}_F(x) = 3, 5, 7, \ldots$, then the discussions above implies that \eqref{eqn:p-adic-integral-2} is zero.
		
		\item If $2e + 1 - 2\mathrm{val}_F(x) = 1$, then we take $m=1$ in \eqref{eqn:p-adic-integral-2} to get
		\begin{multline*}
			\sum_{\overline{a} \in (\mathfrak{o}/(\varpi))^{\times}} \psi( \underbracket{\varpi^{-1} a x^2}_{\mathrm{val}_F = 2e - 1} ) \underbracket{\mes(1 + \varpi \mathfrak{o})}_{= |\varpi|_F = q^{-1}} \\
			= q^{-1} \sum_{\overline{a} \in \mathfrak{o}/(\varpi)} \psi( \varpi^{-1} a x^2) - q^{-1} = -q^{-1}.
		\end{multline*}
		
		\item If $2e + 1 -2\mathrm{val}_F(x) \leq -1$ (equivalently, $\leq 0$) then $\psi(\varpi^{-1} a x^2) = 1$ for all $a \in \mathfrak{o}^{\times}$, hence
		\[ \int_{\mathfrak{o}^{\times}} \psi\left( \varpi^{-1} a x^2 \right) \dd a = \mes(\mathfrak{o}^{\times}) = 1 - q^{-1}. \]
	\end{itemize}
	
	We can thus integrate $x$ over the subset $\mathrm{val}_F \geq e$. The net result in the case $k=1$ is
	\begin{align*}
		\int_{\mathfrak{o}^{\times}} \gamma_\psi\left( 2\varpi^{-1} a \right) \dd a & = \int_{\varpi^{e+1} \mathfrak{o}} (1 - q^{-1}) \dd x - \int_{\varpi^e \mathfrak{o}^{\times}} q^{-1} \dd x \\
		& = (1 - q^{-1}) q^{-e-1} - q^{-1} q^{-e}(1 - q^{-1}) = 0.
	\end{align*}
	
	Next, consider the case $k=0$.
	\begin{itemize}
		\item If $2e - 2\mathrm{val}_F(x) = 2, 4, 6, \ldots$ then there exists $m \in \Z_{\geq 1}$ such that
		\[ m < 2e - 2\mathrm{val}_F(x) \leq 2m. \]
		The earlier discussion implies that \eqref{eqn:p-adic-integral-2} is zero.
		
		\item If $2e - 2\mathrm{val}_F(x) \leq 0$, then $\psi(ax^2) = 1$ for all $a \in \mathfrak{o}^{\times}$, hence
		\[ \int_{\mathfrak{o}^{\times}} \psi\left( a x^2 \right) \dd a = \mes(\mathfrak{o}^{\times}) = 1 - q^{-1}. \]
	\end{itemize}
	
	Therefore we can integrate $x$ over the subset $\mathrm{val}_F \geq e$. The net result is
	\[ \int_{\mathfrak{o}^{\times}} \gamma_\psi(2a) \dd a = q^{-e}(1 - q^{-1}). \]
	
	To conclude, it remains to notice that $q^{-e} = |2|_F$, and we have to multiply $\int_{\mathfrak{o}^{\times}} \gamma_\psi(2a) \dd a$ by $C_0 = |2|_F^{-\frac{1}{2}}$; see \eqref{eqn:p-adic-integral-0}.
\end{proof}

\begin{lemma}\label{prop:main-integral-1}
	With the notation of Lemma \ref{prop:rank-one-integrals}, for all $k \geq 1$ we have
	\[ r_k = \begin{cases}
		(1 - q^{-1}) |2|_F^{1/2}, & k\;\text{is even} \\
		0, & k\;\text{is odd}.
	\end{cases}\]
\end{lemma}
\begin{proof}
	Set $\phi := \mathbf{1}_{2\mathfrak{o}}$. Since $|2|_F^{1/2} \sum_{y \in \mathfrak{o}/(2)} \phi(y) = |2|_F^{1/2}$, it suffices to show
	\[ |\varpi|_F^k \int_{\varpi^{-k} \mathfrak{o}^{\times}} \gamma_\psi(-2t) \omega_\psi(\tilde{w}) \omega_\psi\twobigmatrix{1}{\frac{1}{t}}{}{1}(\phi)(0) \dd t = \begin{cases}
		(1 - q^{-1}) |2|_F, & k \;\text{is even} \\
		0, & k \;\text{is odd.}
	\end{cases}. \]
	
	From \eqref{eqn:root-action} and $|t|_F \geq 1$, we have
	\[ \omega_\psi(\tilde{w}) \omega_\psi\twobigmatrix{1}{\frac{1}{t}}{}{1}(\phi) = \omega_\psi(\tilde{w})(\phi), \]
	and its value at $0$ is $|2|_F^{1/2}$ by Lemma \ref{prop:d-Fourier}. It remains to apply Lemma \ref{prop:main-integral} after a change of variables.
\end{proof}

We are ready to finish the proof of the Gindikin--Karpelevich formula for $\tilde{G}$ in the even case.

\begin{proof}[Proof of Theorem \ref{prop:GK}]
	By Lemma \ref{prop:GK-reduction}, it suffices to show $d(w, \chi) = |2|_F^{\frac{1}{2}} c(w, \chi)$ when $n=1$ and $w$ is the nontrivial element of the Weyl group. From Lemmas \ref{prop:main-integral-0} and \ref{prop:main-integral-1}, we obtain
	\[ d(w, \chi) = \sum_{k=0}^{\infty} r_k z^k = |2|_F^{1/2} \left( 1 + (1 - q^{-1}) \sum_{h=1}^\infty z^{2h} \right). \]
	
	On the other hand, the Gindikin--Karpelevich formula \cite[Theorem 3.1]{Ca80} for $\SO(V^+) \simeq \mathrm{PGL}(2)$ yields
	\[ c(w, \chi) = \frac{1 - q^{-1} z^2}{1 - z^2} = 1 + (1 - q^{-1}) \sum_{h=1}^\infty z^{2h}; \]
	the square here is due to rescaling of (co)root length between $\SL(2)$ and $\mathrm{PGL}(2)$. The desired equality follows.
\end{proof}

\section{Identification of intertwining operators: the odd case}\label{sec:int-op-odd}
We may assume $n \geq 2$ in the odd case.

\subsection{Standard intertwining operators}\label{sec:statement-int-op-odd}
We begin by constructing representatives $\mathring{w} \in \widetilde{K_1}$ for elements $w \in \Omega'_0$. Although $\Omega'_0 \subset \Omega_0$, they will differ slightly from the representatives defined in the even case.

Denote the $B^{\leftarrow}$-positive roots $2\epsilon_2, \ldots, \epsilon_n - \epsilon_{n-1}$ as $\beta'_2, \ldots, \beta'_n$, in this order. Denote the corresponding reflections by $t'_2, \ldots, t'_n$. Then $t'_2, \ldots, t'_n$ generate $\Omega'_0$, giving rise to a Coxeter group of type $\mathrm{C}_{n-1}$.

\begin{description}
	\item[Short roots] For $i > 2$, we take the representative of $t'_i$ as
	\[ \mathring{t}'_i := \tilde{x}_{\beta'_i}(1) \tilde{x}_{-\beta'_i}(-1) \tilde{x}_{\beta'_i}(1). \]
	\item[Long root] For $t'_2$, we rescale $\tilde{x}_{\beta'_2}(1) \tilde{x}_{-\beta'_2}(-1) \tilde{x}_{\beta'_2}(1)$, so that in the Schrödinger model $\omega_\psi(\mathring{t}'_2)$ is the unitary Fourier transform in the second coordinate.
\end{description}

For a general element $w \in \Omega'_0$, we define $\mathring{w}$ by taking any reduced expression in terms of $t'_2, \ldots, t'_n$. As in the even case, $\mathring{w}$ is independent of reduced expressions by \cite[Proposition 4.2]{GL18}. The arguments are exactly the same as in the even case: one can work within the metaplectic group associated with the symplectic subspace $\bigoplus_{i=2}^n (Fe_j \oplus Ff_j)$ of $W$.

Observe that $K_1$ contains the avatar of $K_0$ for $\Sp\left(\bigoplus_{i=2}^n (Fe_j \oplus Ff_j) \right)$, hence $\mathring{w} \in \widetilde{K_1}$.

Some combinatorial preparations are in order.

\begin{lemma}\label{prop:reduced-transition}
	If $t'_{i_1} \cdots t'_{i_\ell}$ is a reduced expression for $w \in \Omega'_0$, then a reduced expression of $w$ as an element of $\Omega_0$ can be obtained as follows.
	\begin{enumerate}
		\item Leave the terms $t'_{i_j}$ with $i_j > 2$ intact, in which case $t'_{i_j} = t_{i_j}$.
		\item Replace each occurrence of $t'_2$ by $t_2 t_1 t_2$. 
	\end{enumerate}
\end{lemma}
\begin{proof}
	Evidently, $t'_2 = t_2 t_1 t_2$. The point is to show the resulting expression is reduced in $\Omega_0$. The length $k$ of $w$ counted in $(\Omega'_0, \{t'_i\}_{i \geq 2})$ is the number of the $B^{\leftarrow}$-positive roots
	\[ 2\epsilon_i, \quad \epsilon_j \pm \epsilon_l \quad (2 \leq i \leq n, \quad 2 \leq l < j \leq n) \]
	that are mapped to $B^{\leftarrow}$-negative roots by $w$. Viewed inside $(\Omega_0, (t_i)_{i \geq 1})$, the extra $B^{\leftarrow}$-positive roots to be considered are, a priori,
	\[ 2\epsilon_1, \quad \epsilon_j \pm \epsilon_1 \quad (2 \leq j \leq n). \]
	
	If $w$ maps a root $2\epsilon_j$ into a negative one, where $2 \leq j \leq n$, then it maps both $\epsilon_j + \epsilon_1$ and $\epsilon_j - \epsilon_1$ into negative ones. These are exactly the extra sign-changes in $\Omega_0$.
	
	On the other hand, Lemma \ref{prop:tw-number} (applied to type $\mathrm{C}_{n-1}$) shows that the number of indices $2 \leq j \leq n$ such that $w(2\epsilon_j)$ is negative equals the number of occurrences of $t'_2$ in $w = t'_{i_1} \cdots t'_{i_\ell}$. Hence replacing $t'_2$ by $t_2 t_1 t_2$ yields an expression with the desired length in $\Omega_0$.
\end{proof}

\begin{proposition}\label{prop:comparison-representatives}
	Let $w \in \Omega'_0$.
	\begin{enumerate}[(i)]
		\item The number $t(w)$ (Definition \ref{def:tw-number}) equals the number of occurrences of $t'_2$ in any reduced expression $w = t'_{i_1} \cdots t'_{i_\ell}$.
		
		\item The representatives $\tilde{w}$ and $\mathring{w}$ are related by
		\[ \tilde{w} = x_1 \mathring{w} = \mathring{w} x_2 \]
		where $x_i$ is an element of $T(F) \simeq (F^{\times})^n$, viewed as an element of $\tilde{T}$ via \eqref{eqn:Levi-splitting}, such that
		\begin{itemize}
			\item all the $n$ coordinates of $x_i$ are in $\{\pm 1\} \subset F^{\times}$;
			\item the first coordinate of $x_i$ is $(-1)^{t(w)}$.
		\end{itemize}
		
		\item For every parabolic $Q = M^1 V \in \mathcal{P}(M^1)$, denote by $\Sigma_Q^{\mathrm{red}} \subset X^*(T^1)$ the set of reduced roots in $V$. For $Q, Q'$ as above, put $d(Q, Q') := \left| \Sigma_Q^{\mathrm{red}} \cap \Sigma_{Q'}^{\mathrm{red}} \right|$.
		
		If $v, w \in \Omega'_0$ satisfy $\ell'(vw) = \ell'(v) + \ell'(w)$, then $d({}^{vw} P, P) = d({}^{wv} P, {}^v P) +  d({}^v P, P)$.
	\end{enumerate}
\end{proposition}
\begin{proof}
	By Lemma \ref{prop:reduced-transition}, the number of occurrences of $t'_2$ equals that of $t_1$ in any reduced expression of $w$ taken in $\Omega_0$. The latter number equals $t(w)$ by Lemma \ref{prop:tw-number}. This proves (i).
	
	As for (ii), we begin by comparing $\tilde{t}_2 \tilde{t}_1 \tilde{t}_2$ and $\mathring{t}'_2$. This can be done within $\Mp(4)$. Observe that $\tilde{t}_2$ belongs to the Siegel parabolic $\GL(2, F)$ via \eqref{eqn:Levi-splitting}. Within $\GL(2, F)$ we have
	\[ \tilde{t}_2 = \begin{pmatrix} & 1 \\ -1 & \end{pmatrix} = \tilde{t}_2^{-1} \begin{pmatrix} -1 & \\ & -1 \end{pmatrix}.\]
	Hence in $\Mp(4)$ we get
	\[ \tilde{t}_2 \tilde{t}_1 \tilde{t}_2 = \tilde{t}_2 \tilde{t}_1 \tilde{t}_2^{-1} \cdot \zeta = \zeta \cdot \tilde{t}_2 \tilde{t}_1 \tilde{t}_2^{-1}, \]
	where $\zeta \in \Mp(4)$ is the canonical lifting of $-1 \in \Sp(4, F)$ acting by $\pm \identity$ on $\omega_\psi^{\pm}$. Moreover, $\tilde{t}_2 \tilde{t}_1 \tilde{t}_2^{-1}$ acts as the unitary Fourier transform on the second coordinate via $\omega_\psi$, i.e.\ as $\mathring{t}'_2$. All these follow by a quick computation in Schrödinger's model.
	
	For each occurrence of $t'_2$ in a given reduced expression of $w \in \Omega'_0$, we pass from $\mathring{t}'_2$ to $\tilde{t}_2 \tilde{t}_1 \tilde{t}_2$ in this way, giving rise to a $\zeta$ on the left or right. When we commute $\zeta$ past $\tilde{t}_2, \ldots, \tilde{t}_n$ in the reduced expression furnished by Lemma \ref{prop:reduced-transition}, the second $-1$ in $\zeta$ might get moved to the slots $2, \ldots, n$, whilst the first $-1$ remain intact.
	
	The roots in $U^1$ and their behavior under $\Omega'_0$ are known. The assertion (iii) follows easily: noting that the reduced roots in $U^1$ are $\epsilon_2, \epsilon_3 - \epsilon_2, \ldots, \epsilon_n - \epsilon_{n-1} \in X^*(T_1)$, the computation reduces to type $\mathrm{C}_{n-1}$.
\end{proof}

Now recall that by Definition \ref{def:P1}, we have ${}^w M^1 = M^1$ for all $w \in \Omega'_0$.

\begin{lemma}\label{prop:w-action-T1}
	Let $w \in \Omega'_0$. The actions of $\tilde{w}$ and $\mathring{w}$ on $\tilde{M}^1 = \Mp(W^1) \times T^1(F)$ are the same. The action is trivial on the $\Mp(W^1)$ factor, and equals the evident action of $\Omega'_0$ on $T^1(F)$.
\end{lemma}
\begin{proof}
	The assertion is clear for $\mathring{w}$, which is constructed within the metaplectic group associated with $\bigoplus_{i=2}^n (Fe_j \oplus Ff_j)$. The case for $\tilde{w}$ follows from Proposition \ref{prop:comparison-representatives} (ii).
\end{proof}

\begin{definition}\label{def:int-op-odd}
	Let $P = MU$ be a parabolic subgroup of $G$ with $P \supset P^1$, and $w \in \Omega'_0$ satisfies ${}^w M = M$. For every $\pi$ in $\tilde{M}\dcate{Mod}$, by imitating Definition \ref{def:int-op-even}, the intertwining operator $J'_w(\pi)$ is the composition
	\[ i_{\tilde{P}}(\pi) \xrightarrow{J'_{\tilde{P}^w | \tilde{P}}(\pi)} i_{\tilde{P}^w}(\pi) \to i_{\tilde{P}}\left({}^{\mathring{w}} \pi\right). \]	
	The differences here are that
	\begin{itemize}
		\item $J'_{\tilde{P}^w | \tilde{P}}(\pi)$ is defined by fixing the Haar measure on $V(F)$ such that $\mes(V(F) \cap K_1) = 1$, for each unipotent subgroup $V$ involved in the intertwining integral;
		\item the map $i_{\tilde{P}^w}(\pi) \to i_{\tilde{P}}\left( {}^{\mathring{w}} \pi \right)$ is defined using the representative $\mathring{w}$, instead of $\tilde{w}$.
	\end{itemize}
	The choice of measures is legitimate since $K_1$ is in good position relative to $M^1$ and $\mathring{w} \in \widetilde{K_1}$.

	In particular, taking $P = P^1$ and $\pi = \omega_\psi^{\flat, -} \boxtimes \chi$ where $\omega_\psi^{\flat}$ (resp.\ $\chi$) is the Weil representation of $\Mp(W^1)$ (resp.\ an unramified character of $T^1(F)$), for every $w \in \Omega'_0$ we have by Lemma \ref{prop:w-action-T1} that
	\[ J'_w\left(\omega_\psi^{\flat, -} \boxtimes \chi\right): i_{\tilde{P}^1}\left(\omega_\psi^{\flat, -} \boxtimes \chi\right) \to i_{\tilde{P}^1}\left(\omega_\psi^{\flat, -} \boxtimes w\chi\right). \]
\end{definition}

Next, $J'_w(\omega_\psi^{\flat, -} \boxtimes \chi)$ induces a homomorphism of $H_\psi^-$-modules after applying $\mathbf{M}_{\tau_1^-}$. This homomorphism is denoted as
\[ J'_w(\chi): \mathbf{M}_{\tau_1^-}\left( i_{\tilde{P}^1}(\omega_\psi^{\flat, -} \boxtimes \chi)\right) \to \mathbf{M}_{\tau_1^-}\left( i_{\tilde{P}^1}(\omega_\psi^{\flat, -} \boxtimes w\chi)\right). \]
By $H_\psi^-$-linearity, $J'_w(\chi)$ restricts to
\[ \Hom_{\widetilde{K_1}}\left(\tau_1^-, i_{\tilde{P}^1}(\omega_\psi^{\flat, -} \boxtimes \chi)\right) \to \Hom_{\widetilde{K_1}}\left(\tau_1^-,  i_{\tilde{P}^1}(\omega_\psi^{\flat, -} \boxtimes w\chi)\right) \]
as both sides are cut out by the idempotent $e^-$.

Consider now the $G^-$-side. We use $\Omega'_{\mathrm{aff}} \simeq \Omega^-_{\mathrm{aff}}$ to identify $\Omega'_0$ with the spherical subgroup $\Omega^-_0$ of $\Omega^-_{\mathrm{aff}}$.

With the same convention on Haar measures, replacing $K_1 \subset G(F)$ by $K^- \subset G^-(F)$ (see \S\ref{sec:spherical-projector}), one has the standard intertwining operator
\[ J^-_w(\chi): i_{P_{\min}^-}(\chi)^{I^-} \to i_{P_{\min}^-}(w\chi)^{I^-}. \]
It restricts to $i_{P_{\min}^-}(\chi)^{K^-} \to i_{P_{\min}^-}(w\chi)^{K^-}$.

\subsection{Towards the main result}
The notation $\chi$ will stand for an unramified character of $T^1(F)$, identified with the corresponding character $\CC[X_*(T^1)] \to \CC$. The following is the counterpart of Theorem \ref{prop:comparison-J} in the odd case.

\begin{theorem}\label{prop:comparison-J-odd}
	Let $w \in \Omega'_0$. Using Propositions \ref{prop:principal-series-Hecke} and \ref{prop:principal-series-Hecke-2}, we transport $J'_w(\chi)$ and $J^-_w(\chi)$ to obtain
	\begin{align*}
		\mathcal{J}'_w(\chi): \Hom_{H_\psi^{\tilde{M}^1, -}}\left( H_\psi^-, \mathbf{M}_{\tau_1^{\flat, -}}(\omega_\psi^{\flat, -}) \otimes \chi \right) & \to \Hom_{H_\psi^{\tilde{M}^1, -}}\left( H_\psi^-, \mathbf{M}_{\tau_1^{\flat, -}}(\omega_\psi^{\flat, -}) \otimes w\chi \right), \\
		\mathcal{J}^-_w(\chi): \Hom_{H^{M_{\min}^-}}(H^-, \chi) & \to \Hom_{H^{M_{\min}^-}}(H^-, w\chi),
	\end{align*}
	which are homomorphisms of $H_\psi^-$ and $H^-$-modules, respectively. They can be compared via the identification \eqref{eqn:isom-Hecke-principal}, and this gives
	\[ \mathcal{J}'_w(\chi) = |2|_F^{t(w)/2} \cdot \mathcal{J}^-_w(\chi) \]
	as an equality between rational families in $\chi$.
\end{theorem}

We will appeal to the same reduction steps as in the even case. The operators $\mathcal{J}'_w(\chi)$ and $\mathcal{J}^-_w(\chi)$ preserve spherical parts. Since the canonical generators $s'(\chi)$ and $s^-(\chi)$ (resp.\ $\sigma'(\chi)$ and $\sigma^-(\chi)$) match under \eqref{eqn:isom-Hecke-principal} (resp.\ \eqref{eqn:isom-rep-principal}) by Proposition \ref{prop:s-matching}, we are reduced to the following variant of Gindikin--Karpelevich formula, which is the odd counterpart of Theorem \ref{prop:GK}.

\begin{theorem}\label{prop:GK-odd}
	Let $w \in \Omega'_0$. Let $c'(w, \chi)$ and $d'(w, \chi)$ be the rational functions characterized by
	\begin{align*}
		J^-_w(\sigma^-(\chi)) & = c'(w, \chi) \sigma^-(w\chi), \\
		J'_w(\chi)(\sigma(\chi)) & = d'(w, \chi) \sigma'(w\chi).
	\end{align*}
	Then $d'(w, \chi) = |2|_F^{t(w)/2} c'(w, \chi)$.
\end{theorem}

In the results above,the intertwining operators $J'_w(\omega_\psi^{\flat, -} \boxtimes \chi)$ are defined using the representative $\mathring{w}$ and the Haar measures normalized via $K_1$. We give a variant below for the representatives $\tilde{w}$ and the Haar measures normalized via $K_0$, as in Definition \ref{def:int-op-even}.

\begin{theorem}\label{prop:comparison-J-odd-2}
	Let $w \in \Omega'_0$. Define
	\[ \mathcal{J}_w(\chi): \Hom_{H_\psi^{\tilde{M}^1, -}}\left( H_\psi^-, \mathbf{M}_{\tau_1^{\flat, -}}(\omega_\psi^{\flat, -}) \otimes \chi \right) \to \Hom_{H_\psi^{\tilde{M}^1, -}}\left( H_\psi^-, \mathbf{M}_{\tau_1^{\flat, -}}(\omega_\psi^{\flat, -}) \otimes w\chi \right) \]
	by using the intertwining operator $J_w\left(\omega_\psi^{\flat, -} \boxtimes \chi\right)$ defined as in Definition \ref{def:int-op-even}. Then
	\[ \mathcal{J}_w(\chi) = (-q^{-1})^{t(w)} |2|_F^{t(w)/2} \mathcal{J}^-_w(\chi). \]
\end{theorem}

The proofs of Theorems \ref{prop:GK-odd} and \ref{prop:comparison-J-odd-2} will be given at the end of \S\ref{sec:odd-computations}. As a preparation, we present an analogue of Lemma \ref{prop:d-formula}. The proof is the same.

\begin{lemma}\label{prop:d-formula-odd}
	With the notation above, $d'(w, \chi)$ is characterized by the following equality in the underlying space of $\omega_\psi^{\flat, -}$:
	\[ d'(w, \chi) (\omega_\psi(\tilde{w}) \phi)(\cdot, 0, \ldots, 0) = \int_{U^1(F) \cap U^{1, w}(F) \backslash U^{1, w}(F)} \sigma'(\chi)(\phi)(u) \dd u \]
	for all $\phi \in V_{\tau_1^-}$.
\end{lemma}

Next, we take a reduced expression $w = t'_{i_1} \cdots t'_{i_\ell}$. Then
\[ J'_w(\chi) = J'_{i_1}(t'_{i_2} \cdots t'_{i_\ell} \chi) \cdots J'_{i_{\ell - 1}}(\chi) \]
as in \eqref{eqn:J-decomp-even}; here we make use of Proposition \ref{prop:comparison-representatives} (iii) to factorize standard intertwining operators.

There is a similar decomposition for $J^-_w(\chi)$, and below is the odd counterpart of Lemma \ref{prop:GK-reduction}.

\begin{lemma}\label{prop:GK-reduction-odd}
	The assertion in Theorem \ref{prop:GK-odd} holds true if it holds for $\tilde{G} = \Mp(4)$ (i.e.\ the case $n=2$) and $w$ being the nontrivial element in $\Omega'_0$ (i.e.\ $w = t'_2$).
\end{lemma}
\begin{proof}
	As in the even case, one easily reduces to the case when $w = t'_i$ and computes in $\widetilde{G'_i}$, where $G_i \subset G$ is the subgroup generated by $M^1$ and the root subgroups for the roots $\pm\alpha \in X^*(T)$ such that $\alpha > 0$ and $t'_i \alpha < 0$ relative to $B^{\leftarrow}$.
	\begin{description}
		\item[Long root] If $i=2$, we have
		\[ \widetilde{G'_2} = \Mp(4) \times T^1(F) \]
		where $\Mp(4)$ is associated with the symplectic subspace $Fe_1 \oplus Fe_2 \oplus Ff_2 \oplus Ff_1$, hence contains $\tilde{M}^1$. Cf.\ the proof of Proposition \ref{prop:comparison-representatives}.
		\item[Short roots] If $i > 2$, we have
		\[ \widetilde{G'_i} = \tilde{M}^1 \times \GL(2, F) \times \text{torus} \]
		where $\GL(2, F)$ is embedded in the Siegel Levi subgroup of $G$, and is associated with $Fe_i \oplus Fe_{i-1}$ and $Ff_i \oplus Ff_{i-1}$.
	\end{description}
	
	When $i > 2$, all computations occur in the Siegel Levi subgroup
	\[ \GL(n-1, F) \stackrel{\text{\eqref{eqn:Levi-splitting}}}{\hookrightarrow} \Mp\left( \bigoplus_{i=2}^n Fe_i \oplus Ff_i \right), \]
	without affecting the first slot of the inducing representation $\omega_\psi^{\flat, -} \boxtimes \chi$. Hence the resulting $d'(t_i, \cdot)$ matches $c'(t_i, \cdot)$. Note that the Haar measures on unipotent radicals here are the same as in the even case: it boils down to the standard $\mathfrak{o}$-structure of $\GL(n-1)$.
	
	When $i = 2$, all computations occur in the factor $\Mp(4)$. The other slots of the inducing representation $\omega_\psi^{\flat, -} \boxtimes \chi$ are unaffected. In view of Proposition \ref{prop:comparison-representatives} (i) that deals with $t(w)$, the problem thus reduces to $\Mp(4)$.
\end{proof}

Denote by $\overline{U^1}$ the unipotent radical of the opposite $\overline{P^1}$. Below is the odd counterpart of Lemma \ref{prop:rank-one-integrals}.

\begin{lemma}\label{prop:rank-one-integrals-odd}
	Assume $n=2$ and $w = t'_2 \in \Omega'_0$. Set
	\[ z := \chi\left( \check{\beta}'_2(\varpi) \right). \]
	
	When $|z| < 1$, one can expand $d'(w, \chi)$ into a convergent sum $\sum_{k=0}^\infty r'_k z^k$. The sequence $(r'_k)_{k=0}^\infty$ is characterized by the equalities below in the underlying space of $\omega_\psi^{\flat, -}$:
	\begin{align*}
		\int_{\overline{U^1}(F) \cap K_1} \omega_\psi(v)\phi(\cdot, 0) \dd v & = r'_0 \cdot \omega_\psi(\mathring{w})\phi(\cdot, 0), \\
		\delta_{P^1}(\check{\beta}'_2(\varpi))^{k/2} \int_{\mathcal{C}_k} \omega_\psi^{\flat, -}(s(v))\left( \omega_\psi(r(v))\phi(\cdot, 0)\right) \dd v & = r'_k \cdot \omega_\psi(\mathring{w})\phi(\cdot, 0), \quad k \geq 1.
	\end{align*}
	for all $\phi \in \Schw^-(\mathfrak{o}/(2\varpi)) \otimes \Schw(\mathfrak{o}/(2))$. Here we define
	\[ \mathcal{C}_k := \left\{\begin{array}{r|l}
		v \in \overline{U^1}(F) & v = u s \check{\beta}'_2(\varpi^k) r, \\
		& u \in U^1(F), \; s \in \Mp(W^1), \; r \in \widetilde{K_1}
	\end{array}\right\}, \quad k \in \Z, \]
	and $(s(v), r(v))$ in the integral means any choice of $(s, r)$ in such a factorization of $v \in \mathcal{C}_k$.
\end{lemma}
\begin{proof}
	This is a crude variant of Lemma \ref{prop:rank-one-integrals} in the odd setting. First of all, in order to ensure the convergence of intertwining integral in Lemma \ref{prop:d-formula-odd} when $|z| < 1$, recall that $\omega_\psi^{\flat, -}$ is supercuspidal.
	
	We claim that for all $\phi \in \Schw^-(\mathfrak{o}/(2\varpi)) \otimes \Schw(\mathfrak{o}/(2))$, Lemma \ref{prop:d-formula-odd} yields an identity in the underlying space of $\omega_\psi^{\flat, -}$:
	\begin{multline}\label{eqn:rank-one-integrals-odd}
		d'(w, \chi) \omega_\psi(\mathring{w})(\phi)(\cdot, 0) =
		\int_{\overline{U^1}(F) \cap K_1} \omega_\psi(v)(\phi)(\cdot, 0) \dd v \\
		+ \sum_{k \geq 1} z^k \delta_{P^1}(\check{\beta}'_2(\varpi))^{k/2} \int_{\mathcal{C}_k} \omega_\psi^{\flat, -}(s(v)) \left( \omega_\psi(r(v))\phi(\cdot, 0) \right) \dd v.
	\end{multline}
	
	Indeed, we will see in \S\ref{sec:odd-computations} that $U^{1, w} = \overline{U^1}$, hence the intertwining integral is over $\overline{U^1}(F)$, Iwasawa-decomposed into disjoint cells $\mathcal{C}_k$. For $k=0$, note that since $K_1$ is in good position relative to $M^1$, we have
	\[ \left(\overline{U^1}(F) M^1(F) U^1(F)\right) \cap K_1 = (\overline{U^1}(F) \cap K_1) (M^1(F) \cap K_1) (U^1(F) \cap K_1), \]
	and it follows readily that $\mathcal{C}_0 = \overline{U^1}(F) \cap K_1$.
	
	Hence the integral over $\mathcal{C}_0$ is simply
	\[ \int_{\overline{U^1}(F) \cap K_1} \omega_\psi(v)(\phi)(\cdot, 0) \dd v. \]

	Over $\mathcal{C}_k$ with $k \geq 1$, the Iwasawa decomposition yields $z^k \delta_{P^1}(\check{\beta}'_2(\varpi))^{k/2}$ and the action of $\omega_\psi^{\flat, -}(s(v))$. The remaining contribution from $\sigma'(\chi)(\phi)(r(v))$ is dealt with using Lemma \ref{prop:sigma-generator}.

	On the other hand, we will show $k < 0 \implies \mathcal{C}_k = \emptyset$ in Lemma \ref{prop:k-elimination}. This justifies \eqref{eqn:rank-one-integrals-odd}.
	
	The rest follows by expanding $d'(w, \chi)$ into a Laurent series $\sum_{k \gg -\infty} r'_k z^k$ and comparing the coefficients in \eqref{eqn:rank-one-integrals-odd}. We have seen that $k < 0 \implies r'_k = 0$. To see that the equations characterize $(r'_k)_{k=0}^\infty$, choose $\phi$ such that $\omega_\psi(\mathring{w})\phi(\cdot, 0)$ is nonzero.
\end{proof}

\subsection{Computations}\label{sec:odd-computations}
Hereafter, we assume $n=2$ and let $w = t'_2 \in \Omega'_0$.

Express the elements of $\Sp(4)$ in the ordered basis $e_2, e_1, f_1, f_2$. The affine $3$-space is isomorphic to $U^1$ through the map $\nu$ below:
\[ \nu(u, v, t) = \begin{pmatrix}
	1 & u & v & t \\
	0 & 1 & 0 & v \\
	0 & 0 & 1 & -u \\
	0 & 0 & 0 & 1
\end{pmatrix}.\]

We have
\begin{gather*}
	x_{\epsilon_2 - \epsilon_1}(u) = \nu(u, 0, 0), \quad x_{\epsilon_2 + \epsilon_1}(v) = \nu(0, v, 0), \quad x_{2\epsilon_2}(t) = \nu(0, 0, t), \\
	\check{\beta}'_2(a) = \begin{pmatrix}
		a & 0 & 0 & 0 \\
		0 & 1 & 0 & 0 \\
		0 & 0 & 1 & 0 \\
		0 & 0 & 0 & a^{-1}
	\end{pmatrix}, \quad a \in F^{\times}.
\end{gather*}

In fact, $\nu$ makes $U^1$ isomorphic to the Heisenberg group attached to the symplectic vector space $Fe_1 \oplus Ff_1$. Therefore, the Haar measure on $U^1(F)$ can be given as
\[ \nu_*\left(\dd u \dd v \dd t \right), \quad \dd u, \dd v, \dd u: \text{Haar measures on}\; F. \]

Since $(\epsilon_2 \pm \epsilon_1)(z_1) = \pm\frac{1}{2}$ and $2\epsilon_2(z_1) = 0$, the Haar measure on $U^1(F)$ giving $\mes(U^1(F) \cap K_1) = 1$ is given by $\nu_*\left(\dd u \dd v \dd t \right)$ where
\begin{align*}
	\dd u: & \mes((\varpi)) = 1, \\
	\dd v, \; \dd t: & \mes(\mathfrak{o}) = 1.
\end{align*}
Cf.\ the proof of Proposition \ref{prop:coinvariant-U1}. Similarly, $\overline{U^1}$ is described by
\[ \nu^-(u, v, t) = \begin{pmatrix}
	1 & 0 & 0 & 0 \\
	v & 1 & 0 & 0 \\
	u & 0 & 1 & 0 \\
	t & u & -v & 1
\end{pmatrix}. \]

The element $w = t'_2$ has the following representative in $K_1$
\[ \dot{w} := \begin{pmatrix}
	0 & 0 & 0 & -1 \\
	0 & 1 & 0 & 0 \\
	0 & 0 & 1 & 0 \\
	1 & 0 & 0 & 0
\end{pmatrix} \]
and $\mathring{w} = \mathring{t}'_2$ is a preimage of $\dot{w}$. It follows that
\begin{equation}\label{eqn:w-permute-U1}
	w \overline{U^1} w^{-1} = U^1, \quad \dot{w}(\overline{U^1}(F) \cap K_1)\dot{w}^{-1} = U^1(F) \cap K_1.
\end{equation}

\begin{lemma}\label{prop:k-elimination}
	Let $k \in \Z$. If there exist $u \in U^1(F)$, $s \in \Sp(W^1)$ and $r \in K_1$ such that $u s \check{\beta}'_2(\varpi^k) r \in \overline{U^1}(F)$, then $k \geq 0$.
\end{lemma}
\begin{proof}
	Since $K_1$ is the stabilizer of the lattice $\mathcal{L}_1$, by unraveling definitions we see
	\[ K_1 = \begin{pmatrix}
		\mathfrak{o} & \mathfrak{o} & (\varpi^{-1}) & \mathfrak{o} \\
		\mathfrak{o} & \mathfrak{o} & (\varpi^{-1}) & \mathfrak{o} \\
		(\varpi) & (\varpi) & \mathfrak{o} & (\varpi) \\
		\mathfrak{o} & \mathfrak{o} & (\varpi^{-1}) & \mathfrak{o}
	\end{pmatrix} \cap \Sp(4, F). \]
	
	If $u s \check{\beta}'_2(\varpi^k) r \in \overline{U^1}(F)$, we have
	\[\begin{pmatrix}
		\varpi^k & * & * & * \\
		0 & * & * & * \\
		0 & * & * & * \\
		0 & 0 & 0 & \varpi^{-k}
	\end{pmatrix} r = \begin{pmatrix}
		1 & 0 & 0 & 0 \\
		* & 1 & 0 & 0 \\
		* & 0 & 1 & 0 \\
		* & * & * & 1
	\end{pmatrix}.\]
	Let $r_{44} \in \mathfrak{o}$ be the $(4, 4)$-entry of $r$. It follows that $\varpi^{-k} r_{44} = 1$, hence $k \geq 0$.
\end{proof}

\begin{lemma}\label{prop:main-integral-odd}
	With the notation of Lemma \ref{prop:rank-one-integrals-odd}, we have $r'_0 = |2|_F^{1/2}$.
\end{lemma}
\begin{proof}
	Take $\phi = \phi_1 \otimes \mathbf{1}_{\mathfrak{o}}$ where $\phi_1 \in \Schw^-(\mathfrak{o}/(2\varpi))$ is arbitrary. The goal is to show that
	\begin{equation*}
		\int_{\overline{U^1}(F) \cap K_1} \omega_\psi(v)(\phi)(\cdot, 0) \dd v = |2|_F^{1/2} \omega_\psi(\mathring{w})(\phi)(\cdot, 0).
	\end{equation*}

	For the right hand side, observe that $\omega_\psi(\mathring{w})$ is the unitary Fourier transform in the second variable. By Lemma \ref{prop:d-Fourier}, we get $\phi_1$ times
	\[ |2|_F \sum_{y \in \mathfrak{o}/(2)} \mathbf{1}_{\mathfrak{o}}(y) = |2|_F \cdot |\mathfrak{o}/(2)| = |2|_F |2|_F^{-1} = 1. \]

	Consider the left hand side. We know that $\omega_\psi(\mathring{w})(\phi_1 \otimes \mathbf{1}_{\mathfrak{o}}) = c \phi_1 \otimes \mathbf{1}_{2\mathfrak{o}}$ for some $c \in \R_{> 0}$. Make the change of variable $\nu = \mathring{w} v \mathring{w}^{-1}$. In view of \eqref{eqn:w-permute-U1}, the left hand side equals
	\begin{equation*}
		\left( c \int_{U^1(F) \cap K_1} \omega_\psi(\mathring{w})^{-1} \omega_\psi(\nu)(\phi_1 \otimes \mathbf{1}_{2\mathfrak{o}}) \dd \nu \right) (\cdot, 0).
	\end{equation*}
	
	To carry out the integral in parentheses, use the multiplication in Heisenberg group to express
	\[ \nu = \nu(u, v, t) = x_{\epsilon_2 - \epsilon_1}(u) x_{\epsilon_2 + \epsilon_1}(v) x_{2\epsilon_2}(t - uv), \]
	lift this to $\tilde{G}$, and then integrate in the order $\int_{u \in (\varpi)} \int_{v \in \mathfrak{o}} \int_{t \in \mathfrak{o}}$.
	
	We claim that given $(u, v)$, the innermost integral yields $\phi_1 \otimes \mathbf{1}_{2\mathfrak{o}}$. Indeed, it equals
	\[ \int_{\mathfrak{o}} \omega_\psi(\tilde{x}_{2\epsilon_2}(t))(\phi_1 \otimes \mathbf{1}_{2\mathfrak{o}}) \dd t; \]
	the action of $\omega_\psi(\tilde{x}_{2\epsilon_2}(t))$ only affects the second slot $\mathbf{1}_{2\mathfrak{o}}$, and is trivial by \eqref{eqn:root-action}. The claim follows.
	
	Next, let $y = (y_1, y_2) \in \mathfrak{o}^2 \simeq \mathfrak{o}f_1 \oplus \mathfrak{o}f_2$. For all $v \in \mathfrak{o}$, we have
	\[ \omega_\psi\left( \tilde{x}_{\epsilon_2 + \epsilon_1}(v) \right)(\phi_1 \otimes \mathbf{1}_{2\mathfrak{o}})(y) = \psi(2v y_1 y_2) (\phi_1 \otimes \mathbf{1}_{2\mathfrak{o}})(y). \]
	This is zero if $y_2 \notin 2\mathfrak{o}$. If $y_2 \in 2\mathfrak{o}$ then $\psi(2v y_1 y_2) = 1$. Therefore
	\[ \omega_\psi\left( \tilde{x}_{\epsilon_2 + \epsilon_1}(v) \right)(\phi_1 \otimes \mathbf{1}_{2\mathfrak{o}}) = \phi_1 \otimes \mathbf{1}_{2\mathfrak{o}}. \]
	
	Next, for all $u \in (\varpi)$ we have
	\[ \omega_\psi\left( \tilde{x}_{\epsilon_2 - \epsilon_1}(u) \right)(\phi_1 \otimes \mathbf{1}_{2\mathfrak{o}})(y) = (\phi_1 \otimes \mathbf{1}_{2\mathfrak{o}})(y + uy_2 f_1) = \phi_1(y_1 + uy_2) \mathbf{1}_{2\mathfrak{o}}(y_2). \]
	This is zero unless $y_2 \in 2\mathfrak{o}$, in which case we get $\phi_1(y)$ since $uy_2 \in (2\varpi)$. Therefore
	\[ \omega_\psi\left( \tilde{x}_{\epsilon_2 - \epsilon_1}(u) \right)(\phi_1 \otimes \mathbf{1}_{2\mathfrak{o}}) = \phi_1 \otimes \mathbf{1}_{2\mathfrak{o}}. \]
	
	After integration, we are left with $c \omega_\psi(\mathring{w})^{-1} (\phi_1 \otimes \mathbf{1}_{2\mathfrak{o}}) = \phi_1 \otimes \mathbf{1}_{\mathfrak{o}}$. Restriction to $(\cdot, 0)$ yields $\phi_1$, as desired.
\end{proof}

Recall that in the case $n=2$, we have $\tilde{G} = \Mp(4)$ and $G^- = \SO(V^-)$ where $\dim V^- = 5$.

\begin{lemma}\label{prop:GK-odd-Mp4}
	We have $d'(w, \chi) = |2|_F^{1/2} c'(w, \chi)$ as rational functions in $\chi$.
\end{lemma}
\begin{proof}
	Introduce $z = \chi\left( \check{\beta}'_2(\varpi) \right)$ as before. Note that the range of convergence $|z| < 1$ is the same for both $J'_w(\chi)$ and $J^-_w(\chi)$. Define
	\[ \rho(z) := d'(w, \chi) c'(w, \chi)^{-1}. \]
	By the rationality of intertwining operators (see \cite[Théorème 2.4.1]{Li12b}), this is a rational function in $z$, not identically zero. We claim that
	\begin{equation}\label{eqn:GK-odd-Mp4}
		\rho \;\text{has neither zero nor pole when}\; z \neq 0.
	\end{equation}
	
	The Gindikin--Karpelevich formula \cite[Theorem 3.1]{Ca80} for $G^-$ implies that $c'(w, \chi)$ has neither poles nor zeros in the region $|z| < 1$. Indeed, by Remark \ref{rem:K-pm}, the factors $q_{\alpha/2}^{1/2} q_\alpha$ and $q_\alpha$ in that formula are always $\geq 1$.
	
	Since $J^-_w(\chi)(\sigma^-(\chi)) = c'(w, \chi) \sigma^-(w\chi)$, by the description of Langlands quotients as the image of $J^-_w(\chi)$ (see eg.\ \cite[Corollary 3.2]{Kon03}), we infer that the spherical part of $i_{P^-_{\min}}(\chi)$ survives in the unique irreducible quotient thereof, i.e.\ the Langlands quotient, when $|z| < 1$.
	
	The statement above translates into Hecke modules: the spherical part of the $H^-$-module $i_{P^-_{\min}}(\chi)^{I^-}$ survives in its unique irreducible quotient when $|z| < 1$.
	
	On the side of $\tilde{G}$, the basic theory of intertwining operators and Langlands quotients remains valid, as explained in \cite{Li12b, Luo20}. Hence the spherical part of $\mathbf{M}_{\tau_1^-}\left(i_{\tilde{P}^1}(\omega_\psi^{\flat, -} \boxtimes \chi)\right)$ survives in the unique irreducible quotient when $|z| < 1$, by passing to Hecke modules and recalling that all these Hecke modules together with their spherical parts match well under $\mathrm{TW}$.
	
	As before, this unique irreducible quotient is the image of $J'_w(\chi)$. Since $J'_w(\chi)(\sigma'(\chi)) = d'(w, \chi) \sigma'(w\chi)$, we conclude that $d'(w, \chi)$ has neither zero nor pole when $|z| < 1$. Consequently, $\rho(z)$ has neither zero nor pole in the region $|z| < 1$.
	
	Next, consider the region $|z| > 1$. The endomorphism $J_w(w\chi) J_w(\chi)$ can be identified with a rational function in $\chi$, and it equals the inverse of Harish-Chandra $\mu$-function of $i_{\tilde{P}^1}(\omega_\psi^{\flat, -} \boxtimes \chi)$, up to some positive constant depending on Haar measures. The same holds on the side of $G^-$. By combining
	\begin{itemize}
		\item the isomorphism $\mathrm{TW}$ and its avatars for Levi subgroups that preserve Hilbert structures (see \cite[\S 2]{TW18}),
		\item Proposition \ref{prop:ri-compatibility}, and
		\item the general formalism in \cite[Theorem B]{BHK11},
	\end{itemize}
	we know that the $\mu$-functions for $i_{\tilde{P}^1}(\omega_\psi^{\flat, -} \boxtimes \chi)$ and $i_{P^-_{\min}}(\chi)$ match up to an explicit positive constant.
	
	Since $w\chi$ falls in the convergence range, it follows from the previous case that $\rho(z)$ has neither zero nor poles in the region $|z| > 1$.
	
	Now consider the unitary range $|z| = 1$. By the general theory (see eg.\ \cite[Corollaire IV.1.2]{Wa03}), $J'_w(\chi)$ is nonzero, and the order of pole is $\leq 1$; moreover, if $w\chi \neq \chi$ then $J'_w(\chi)$ has no pole there.
	
	Before dealing with the remaining case $w\chi = \chi$, recall that $J'_w(\omega_\psi^{\flat, -} \boxtimes \chi)$ equals a normalizing factor $r(\chi)$ times an intertwining operator $R(\chi)$, the latter being the ``normalized intertwining operator''. Both are rational in $\chi$, the operator $R(\chi)$ is unitary when $|z| = 1$, and the factor $r(\chi)$ can be constructed from Harish-Chandra $\mu$-function; we refer to \cite[\S 3.3]{Li12b} or \S\ref{sec:normalization} for details.
	
	Thus the possible pole of $d'(w, \chi)$ when $|z|=1$ and $w\chi = \chi$ is accounted by that of $r(\chi)$. We infer that, in this case
	\begin{equation*}
		d'(w, \chi) = \infty \iff r(\chi) = \infty \\
		\iff \mu(\chi) = 0.
	\end{equation*}
	See \textit{loc.\ cit.}. More precisely, the second equivalence is based on the fact that $\mu(\chi)^{-1} = |r(\chi)|^2$ up to a positive constant. One can also relate this to the reducibility of $i_{\tilde{P}^1}(\omega_\psi^{\flat, -} \boxtimes \chi)$ by Knapp--Stein theory (see \cite[Theorem 2.4]{Luo20} and \cite[p.832]{Li12b}).
	
	Of course, the same story holds for $G^-$ and $i_{P^-_{\min}}(\chi)$ when $|z|=1$. Since
	\begin{itemize}
		\item the action of $w$ on $\chi$ is the same for both sides, and
		\item the inducing data that match under $\mathrm{TW}$ have the same $\mu$-functions and reducibilities,
	\end{itemize}
	we conclude that $\rho(z)$ has neither zero nor pole on the circle $|z| = 1$.

	Summing up, \eqref{eqn:GK-odd-Mp4} holds true. Hence $\rho(z) = b z^h$ for some $h \in \Z$ and $b \in \CC^{\times}$. By Lemma \ref{prop:rank-one-integrals-odd},
	\[ d'(w, \chi) = |2|_F^{1/2} (1 + \text{higher terms in}\; z), \]
	whilst by \cite[Theorem 3.1]{Ca80},
	\[ c'(w, \chi) = 1 + \text{higher terms in}\; z. \]
	
	Hence $\rho(z) = |2|_F^{1/2}$ as desired.
\end{proof}

\begin{proof}[Proof of Theorem \ref{prop:GK-odd}]
	The case $n=2$ and $w = t'_2$ has been settled by Lemma \ref{prop:GK-odd-Mp4}. This implies the general case by Lemma \ref{prop:GK-reduction-odd}.
\end{proof}

This also completes the proof of Theorem \ref{prop:comparison-J-odd}.

\begin{proof}[Proof of Theorem \ref{prop:comparison-J-odd-2}]
	Given Theorem \ref{prop:comparison-J-odd}, it suffices to show
	\[ J_w(\omega_\psi^{\flat, -} \boxtimes \chi) = (-q^{-1})^{t(w)} J'_w(\omega_\psi^{\flat, -} \boxtimes \chi). \]
	
	The only differences between them come from
	\begin{enumerate}[(a)]
		\item different representatives $\tilde{w}$ and $\mathring{w}$ of $w \in \Omega'_0$;
		\item different Haar measures on unipotent subgroups involved in $J_{\tilde{P}^{1, w} | \tilde{P}^1}(\omega_\psi^{\flat, -} \boxtimes \chi)$.
	\end{enumerate}
	
	Concerning (a), by Proposition \ref{prop:comparison-representatives} (ii), $J_w(\omega_\psi^{\flat, -} \boxtimes \chi)$ and $J'_w(\omega_\psi^{\flat, -} \boxtimes \chi)$ differ by
	\begin{align*}
		i_{\tilde{P}^1}\left( \omega_\psi^{\flat, -} \boxtimes \chi \right) & \rightiso i_{\tilde{P}^1}\left( \omega_\psi^{\flat, -} \boxtimes \chi \right) \\
		f & \mapsto \left[ x \mapsto f(\xi x) \right],
	\end{align*}
	where $\xi \in T(F) \simeq (F^{\times})^n$ has all coordinates in $\{\pm 1\} \subset F^{\times}$, and the first coordinate equals $(-1)^{t(w)}$. The first coordinate acting by $\omega_\psi^{\flat, -}$ yields $(-1)^{t(w)}$; the remaining $\pm 1$ have no effect since $\chi$ is unramified.
	
	As to (b), factorize $J_{\tilde{P}^{1, w} | \tilde{P}^1}(\omega_\psi^{\flat, -} \boxtimes \chi)$ and reduce to the cases $w = t'_i$ as in Lemma \ref{prop:GK-reduction-odd}, where $2 \leq i \leq n$. For $i > 2$, it has been observed in the proof of Lemma \ref{prop:GK-reduction-odd} that the choices of Haar measures are the same in the even and odd cases.
	
	For $i = 2$, an inspection of the proof of Lemma \ref{prop:main-integral-odd} reveals that the only difference occurs in the subgroup $x_{\epsilon_2 - \epsilon_1}(F) \simeq F$ of $U^1(F)$. In the even case $\mes(\mathfrak{o}) = 1$. In the odd case $\mes(\mathfrak{o}) = |\varpi|^{-1} \mes((\varpi)) = q$.
	
	Since $t'_2$ occurs $t(w)$ times in any reduced expression of $w$, they give the factor $(-q^{-1})^{t(w)}$. 
\end{proof}

\section{Applications}\label{sec:applications}
\subsection{Preservation of Aubert involution}\label{sec:Aubert}
Our brief review of the duality involution of Aubert follows \cite[\S A.2]{KMSW}. Set $\tilde{G} := \Mp(W)$ and $G^{\pm} := \SO(V^\pm)$ as before. Let $\mathrm{Groth}(\tilde{G})$ (resp.\ $\mathrm{Groth}(G^\pm)$) be the Grothendieck group of the abelian subcategory of $\tilde{G}\dcate{Mod}$ (resp.\ $G^\pm\dcate{Mod}$) consisting of objects of finite length. On the level of Grothendieck groups, the Aubert involutions are the endomorphisms below of $\mathrm{Groth}(\tilde{G})$, $\mathrm{Groth}(G^\pm)$, respectively.
\begin{align*}
	\mathbf{D}^{\tilde{G}} & := \sum_{\substack{P \supset B^{\leftarrow} \\ \text{parabolic}}} (-1)^{\dim A_{B^{\leftarrow}}/A_P} \; i_{\tilde{P}} r_{\tilde{P}}, \\
	\mathbf{D}^{G^\pm} & := \sum_{\substack{P^\pm \supset P_{\min}^\pm \\ \text{parabolic}}} (-1)^{\dim A_{P_{\min}^\pm}/A_{P^\pm}} \; i_{P^\pm} r_{P^\pm}.
\end{align*}

Denote by $[\pi]$ the image of $\pi$ in the Grothendieck group. It is known that if $\pi$ is irreducible in $\tilde{G}\dcate{Mod}$, then $\beta(\pi)\mathbf{D}^{\tilde{G}}[\pi] = [\hat{\pi}]$ for some irreducible $\hat{\pi}$. Here
\[ \beta(\pi) := (-1)^{\dim A_{B^{\leftarrow}}/A_{M_\pi}} \]
where $M_\pi$ is a standard Levi subgroup relative to $B^{\leftarrow}$ such that the cuspidal support of $\pi$ has the form $[M_\pi, \pi']$. Ditto for $\mathbf{D}^{G^\pm}$ upon replacing $B^{\leftarrow}$ by $P_{\min}^{\pm}$.

The involutions above preserve the Grothendieck groups of each Bernstein block. We will compare them via the equivalences
\begin{equation*}
	\mathcal{G}_\psi^{\pm} \simeq H_\psi^{\pm}\dcate{Mod} \stackrel{\mathrm{TW}^*}{\simeq} H^{\pm}\dcate{Mod} \simeq \mathcal{G}^{\pm}.
\end{equation*}

\begin{corollary}\label{prop:Aubert-involution-1}
	Let $\pi$ be an object of $\mathcal{G}_\psi^{\pm}$ of finite length. Let $\sigma$ be an object of $\mathcal{G}^\pm$ which matches $\pi$ via $\mathcal{G}_\psi^{\pm} \simeq \mathcal{G}^{\pm}$. Then $\mathbf{D}^{\tilde{G}}[\pi]$ matches $\pm \mathbf{D}^{G^\pm}[\sigma]$ under the induced isomorphism between Grothendieck groups.
\end{corollary}
\begin{proof}
	In the case of $\mathcal{G}_\psi^+$, recall that the parabolic subgroups $P \supset B^{\leftarrow}$ and $P^+ \supset P_{\min}^+$ are in natural bijection, say by matching the patterns of $\GL(n_i)$'s. Also, $\dim A_P = \dim A_{P^+}$.
	
	There is a variant of $\mathcal{G}_\psi^{\pm} \simeq \mathcal{G}^{\pm}$ for Levi subgroups. Since the functors $r_{\tilde{P}}$ (resp.\ $i_{\tilde{P}}$) match $r_{P^+}$ (resp.\ $i_{P^+}$) by Proposition \ref{prop:ri-compatibility} under these correspondences, the duality operators also match.
	
	As for the case of $\mathcal{G}_\psi^-$, only the parabolic subgroups $P = MU \supset B^{\leftarrow}$ with $M = \Sp(W^{\flat}) \times \prod_i \GL(n_i)$ such that $\dim W^{\flat} \geq 2$ are in bijection with $P^- \supset P_{\min}^-$. However, $r_{\tilde{P}}|_{\mathcal{G}_\psi^-}$ is nonzero only for such $P$. We still have $\dim A_P = \dim A_{P^-}$, whilst $\dim A_{B^{\leftarrow}} = \dim A_{P_{\min}^-} + 1$. This suffices to conclude the proof.
\end{proof}

\begin{remark}
	Assume $\pi$ is irreducible. One readily verifies that
	\begin{enumerate}
		\item If $\pi$ is in $\mathcal{G}_\psi^+$ then $\beta(\pi) = \beta(\sigma) = 1$;
		\item If $\pi$ is in $\mathcal{G}_\psi^-$ then $\beta(\pi) = -1$ whilst $\beta(\sigma) = 1$.
	\end{enumerate}
\end{remark}

When restricted to Bernstein blocks of a given rank, the Aubert involution can be lifted to the level of representations $\pi \mapsto \hat{\pi}$. A precise recipe can be found in \cite[(A.2.2)]{KMSW}. Suppose that all irreducible subquotients of $\pi$ has cuspidal support of the form $[L, \pi']$, for a fixed Levi subgroup $L$. The formula in \textit{loc.\ cit.} expresses $\hat{\pi}$ as a canonical quotient of
\[ \bigoplus_{\substack{P \supset B^{\leftarrow} \\ \dim A_P = \dim A_L}} i_{\tilde{P}} r_{\tilde{P}}(\pi), \]
involving images of $i_{\tilde{Q}} r_{\tilde{Q}}(\pi)$ with $\dim A_Q < \dim A_L$. When $\pi$ is irreducible, $[\hat{\pi}] = \beta(\pi) \mathbf{D}^{\tilde{G}}[\pi]$ in $\mathrm{Groth}(\tilde{G})$.

The same recipe also applies to $G^{\pm}$ as well. The endofunctor $\pi \mapsto \hat{\pi}$ (resp.\ $\sigma \mapsto \hat{\sigma}$) preserves $\mathcal{G}_\psi^{\pm}$ (resp.\ $\mathcal{G}^{\pm}$). The following property is immediate.

\begin{corollary}\label{prop:TW-vs-Aubert}
	The following diagram commutes up to a canonical isomorphism
	\[\begin{tikzcd}
		\mathcal{G}_\psi^{\pm} \arrow[d, "{(\cdot)^\wedge}"'] \arrow[r, "\sim"] & \mathcal{G}^{\pm} \arrow[d, "{(\cdot)^\wedge}"] \\
		\mathcal{G}_\psi^{\pm} \arrow[r, "\sim"'] & \mathcal{G}^{\pm}
	\end{tikzcd}\]
	where the horizontal equivalences are induced by $\mathrm{TW}: H^{\pm} \rightiso H_\psi^{\pm}$.
\end{corollary}
\begin{proof}
	Consider $\pi$ in $\mathcal{G}_\psi^{\pm}$ and $\sigma \in \mathcal{G}^{\pm}$. The maps involved in $\hat{\pi}$ are built from the functorial maps $r_{\tilde{P}} \to r_{\tilde{P}'}$ when $P' \subset P$, together with units and co-units of adjunctions. The same holds for $\hat{\sigma}$. Some signs are also involved in \cite[(A.2.2)]{KMSW}, but this does not affect the quotient $\hat{\pi}$.
	
	The case of $\mathcal{G}_\psi^+$ corresponds to taking $L = T$ in the construction above. Hence it reduces to the matching between $i_{\tilde{P}} r_{\tilde{P}}$ and $i_{P^\pm} r_{P^\pm}$ as before.
	
	The case of $\mathcal{G}_\psi^-$ corresponds to taking $L = M^1$. It follows that a parabolic subgroup $P \supset B^{\leftarrow}$ appears in the construction of $\hat{\pi}$ only when $P$ matches some $P^- \supset P_{\min}^-$. Therefore, the same arguments carry over verbatim.
\end{proof}

\subsection{Intertwining operators for general Levi subgroups}
Let $P = MU \supset B^{\leftarrow}$ be a parabolic subgroup of $G$. Let $P^{\pm} = M^{\pm} U^{\pm} \supset P_{\min}^\pm$ be the corresponding one for $G^{\pm}$. Specifically,
\begin{align*}
	M & = \Sp(W^\flat) \times \prod_{i=1}^r \GL(n_i), \\
	M^{\pm} & = \SO(V^{\pm, \flat}) \times \prod_{i=1}^r \GL(n_i), \quad \dim V^{\pm, \flat} = \dim W^\flat + 1.
\end{align*}
It is tacitly assumed that $n^\flat := \frac{1}{2}\dim W^{\flat} \geq 1$ in the $-$ case.

Let $\Omega^M_0$ (resp.\ $(\Omega^M_0)'$) be the avatar of $\Omega_0$ (resp.\ $\Omega'_0$) for $M$.

Consider the $+$ case first. Let $w \in \Omega_0$ be such that ${}^w M = M$, and $w$ has minimal length in its $\Omega^M_0$-coset. This implies that $w$ acts on the factor $\Sp(W^\flat)$ by conjugation of some $t \in T(F)$. Take a reduced expression
\[ w = t_{i_1} \cdots t_{i_\ell}, \]
and define $\tilde{w} = \tilde{t}_{i_1} \cdots \tilde{t}_{i_\ell} \in \widetilde{K_0}$ as in \S\ref{sec:statement-int-op-even}. The conjugation action of $\tilde{w}$ on $\tilde{M}$ depends only on its image $\dot{w} \in K_0$.

\begin{lemma}\label{prop:hatw-flat}
	Let $T^\flat := T \cap \Sp(W^\flat)$ be the standard maximal torus of $\Sp(W^\flat)$.	The element $\dot{w}$ acts on the factor $\Sp(W^\flat)$ of $M$ by conjugation by some element of $T^\flat(F) \simeq (F^\times)^{n^\flat}$, whose coordinates are all $\pm 1$.
\end{lemma}
\begin{proof}
	An inspection of the construction of $\tilde{w}$ shows that $\dot{w} = \pr(\tilde{w}) \in \Sp(W)$ is a product of elements of the following two types:
	\begin{itemize}
		\item permutations of the subscripts $1, \ldots, n$ of the symplectic basis,
		\item mapping $e_1$ to $f_1$ and $f_1$ to $-e_1$, leaving the other basis elements intact.
	\end{itemize}
	Moreover, the action must stabilize $W^\flat$. Since $w$ has minimal length in its $\Omega^M_0$-coset, one easily sees that $\dot{w}$ must send $e_i$ (resp.\ $f_i$) to $\pm e_i$ (resp.\ $\pm f_i$) for each $1 \leq i \leq n^\flat$.
\end{proof}

On the $G^+$ side, one has $\Omega_0^+ \simeq \Omega_0$, and ditto for Levi subgroups. For $w$ as above, one constructs a representative $\ddot{w} \in K^+$ of $w$ in the same (standard) way as \S\ref{sec:statement-int-op-even}, namely by using an $F$-pinning. It also acts on $\SO(V^{+, \flat})$ via conjugation by an $\mathfrak{o}$-point of the maximal torus.

For every automorphism $\eta$ of an algebra $H$, define the transport of structures for $H$-modules $N \mapsto {}^\eta N := (\eta^{-1})^* N$ so that ${}^{\eta\eta'} N = {}^\eta ({}^{\eta'} N)$.

\begin{lemma}\label{prop:Hecke-w}
	Denote by $H_i$ the Iwahori--Hecke algebra of $\GL(n_i)$ for each $1 \leq i \leq r$. Let $w$ be as above. There exists a canonical automorphism $\xi$ of $\bigotimes_{i=1}^r H_i$ such that for each $\pi$ in $\tilde{M}\dcate{Mod}$ (resp.\ $\sigma$ in $M^+\dcate{Mod}$), we have the equality
	\[ \mathbf{M}_{\tau_{0, \tilde{M}}}\left( {}^{\tilde{w}} \pi \right) = {}^{\identity \otimes \xi} \mathbf{M}_{\tau_{0, \tilde{M}}}(\pi) , \quad \text{resp.} \quad
	\left( {}^{\ddot{w}}\sigma \right)^{I^{M^+}} = {}^{\identity \otimes \xi} \left( \sigma^{I^{M^+}}\right) \]
	of $H_\psi^{\tilde{M}, +}$-modules (resp.\ $H^{M^+}$-modules).
\end{lemma}
\begin{proof}
	Begin with the case of $\pi$. The idea is to apply Proposition \ref{prop:Hecke-transport} to $\tilde{M}$ and the automorphism $\theta(x) = \dot{w}x\dot{w}^{-1}$. We claim that ${}^{\dot{w}} I = I$ and ${}^{\dot{w}} \tau_{0, \tilde{M}} = \tau_{0, \tilde{M}}$.
	
	By decomposing $I \cap M(F)$ using Lemma \ref{prop:cpt-intersection-Levi}, it suffices to consider the factors $\Mp(W^\flat)$ and $\prod_i \GL(n_i, F)$ separately. The action of $\dot{w}$ on $\Mp(W^\flat)$ is conjugation by some $a \in T^\flat(F)$ with coordinates $\pm 1$. In particular, $a \in I \cap \Sp(W^\flat)$. Since $\tau^\flat_0$ is $\omega_\psi$ restricted to $\Schw\left( \bigoplus_{i=1}^{n^\flat} \frac{\mathfrak{o} f_i}{(2) f_i} \right)$, the explicit formula \eqref{eqn:Siegel-action} for $\omega_\psi$ implies $\tau^\flat_0(\tilde{a}) = \identity$ where $\tilde{a} \in \widetilde{T^\flat}$ is the image of $a$ under \eqref{eqn:Levi-splitting} (for $M=T$). The claim about ${}^{\dot{w}} \tau_{0, \tilde{M}}$ follows for the $\Mp(W^\flat)$ factor.
	
	As for $\prod_i \GL(n_i, F)$, note that $\dot{w}$-conjugation preserves $T \cap \prod_i \GL(n_i)$; it also preserves the opposite of the standard Borel subgroup of $\prod_i \GL(n_i)$ by the assumption on length, hence the standard Borel subgroup is preserved as well. Everything in the construction of $\dot{w}$ is defined over $\mathfrak{o}$, hence it follows that $I \cap \prod_i \GL(n_i, F)$ is preserved.
	
	The $\GL$-components of $\tau_{0, \tilde{M}}$ are trivial, hence ${}^{\tilde{w}} \tau_{0, \tilde{M}} = \tau_{0, \tilde{M}}$ holds trivially on $\prod_i \GL(n_i, F)$. The claim is proved.
	
	Applying Proposition \ref{prop:Hecke-transport} and the claim, we see $\mathbf{M}_{\tau_{0, \tilde{M}}}\left( {}^{\dot{w}} \pi \right) \xrightarrow{=} \mathbf{M}_{\tau_{0, \tilde{M}}}\left( {}^{\dot{w}} \pi \right)$ is equivariant with respect to the isomorphism
	\begin{align*}
		\Xi: \mathcal{H}(\tilde{G} \sslash \tilde{I}, \tau_0) & \rightiso \mathcal{H}(\tilde{G} \sslash \tilde{I}, \tau_0) \\
		f & \mapsto [x \mapsto f(\dot{w}^{-1} x \dot{w})],
	\end{align*}
	cf.\ \eqref{eqn:Hecke-transport}. Again, we inspect $\Xi$ on the factors $\Mp(W^\flat)$ and $\prod_i \GL(n_i, F)$ separately.
	
	For $\Mp(W^\flat)$, it is given by conjugation by $a \in T^\flat(F) \cap I$. We have seen $\tau^\flat_0(\tilde{a}) = \identity$, hence $\Xi$ is trivial there. For $\prod_i \GL(n_i, F)$, we obtain the automorphism $\xi$ of $\bigotimes_i H_i$.
	
	Now consider the case of $\sigma$. The same arguments carry over; in fact it is simpler as one can put $\tau = \mathbf{1}$ in Proposition \ref{prop:Hecke-transport}. However, since the actions of $\dot{w}$ and $\ddot{w}$ on $\prod_i \GL(n_i)$ can differ, the resulting automorphisms of $\bigotimes_i H_i$ could be different, a priori.
	
	To show that the same automorphism $\xi$ works, one analyzes the action of $\dot{w}$ (resp.\ $\ddot{w}$) on the symplectic basis (resp.\ hyperbolic basis) to determine its effect on $\bigotimes_i H_i$. The upshot is that two actions differ only up to $\mathfrak{o}$-points in the standard maximal torus of $\prod_i \GL(n_i)$. Such elements are harmless since they belong to the standard Iwahori subgroup of $\prod_i \GL(n_i, F)$.
\end{proof}

\begin{lemma}\label{prop:Hecke-Weyl}
	Let $w \in \Omega_0$ be such that ${}^w M = M$, and $w$ has minimal length in its $\Omega^M_0$-coset. For all $\pi$ in $\tilde{M}\dcate{Mod}$ and $\sigma$ in $M^\pm\dcate{Mod}$, we have
	\[ \left\{ \text{Hecke-equivariant}\; \mathbf{M}_{\tau_{0, \tilde{M}}}(\pi) \rightiso \sigma^{I^{M^+}} \right\} = \left\{ \text{Hecke-equivariant}\; \mathbf{M}_{\tau_{0, \tilde{M}}}({}^{\tilde{w}} \pi) \rightiso ({}^{\ddot{w}} \sigma)^{I^{M^+}} \right\} \]
	where the equivariance is understood via $\mathrm{TW}: H^{M^+} \rightiso H_\psi^{\tilde{M}, +}$.
\end{lemma}
\begin{proof}
	Immediate from the Hecke-equivariant equalities in Lemma \ref{prop:Hecke-w}, by noting that $\mathrm{TW}: H^{M^+} \rightiso H_\psi^{\tilde{M}, +}$ is $\identity$ on $\bigotimes_i H_i$, hence it commutes with $\identity \otimes \xi$.
\end{proof}

Also observe that if $\Phi: \mathbf{M}_{\tau_{0, \tilde{M}}}(\pi) \rightiso \sigma^{I^{M^+}}$ is Hecke-equivariant, $\Phi$ also defines $\mathbf{M}_{\tau_{0, \tilde{M}}}(\pi \otimes \chi) \rightiso (\sigma \otimes \chi)^{I^{M^+}}$ for every unramified character $\chi: \prod_i \GL(n_i, F) \to \CC^{\times}$. Moreover, all these results become trivial when $M=T$.

\begin{corollary}\label{prop:comparison-J-general}
	Let $\pi$ (resp.\ $\sigma$) be of finite length in $\mathcal{G}_\psi^{\tilde{M}, +}$ (resp.\ $\mathcal{G}^{M^+}$), and assume that there is a chosen Hecke-equivariant isomorphism
	\[ \Phi: \mathbf{M}_{\tau_{0, \tilde{M}}}(\pi) \rightiso \sigma^{I^{M^+}}. \]
	
	Let $w$ be as in Lemma \ref{prop:Hecke-Weyl}. Following Definition \ref{def:int-op-even}, we have the rational family
	\[ J_w(\pi \otimes \chi) \in \Hom_{\tilde{G}}\left(i_{\tilde{P}}(\pi \otimes \chi), i_{\tilde{P}}({}^{\tilde{w}}(\pi \otimes \chi))\right) \]
	in unramified characters $\chi: \prod_{i=1}^r \GL(n_i, F) \to \CC^{\times}$. Denote by $\mathcal{J}_w(\pi \otimes \chi)$ the corresponding rational family of homomorphisms of $H_\psi^+$-modules and interpret $i_{\tilde{P}}$ as $\Hom_{H_\psi^{\tilde{M}, +}}(H_\psi^+, \cdot)$.
	
	As in \S\ref{sec:statement-int-op-even}, there are also counterparts $J^+_w(\sigma \otimes \chi)$ and $\mathcal{J}^+_w(\sigma \otimes \chi)$. By using $\Phi$ and the Hecke-equivariant $\mathbf{M}_{\tau_0}({}^{\dot{w}} (\pi \otimes \chi)) \to {}^{\ddot{w}} (\sigma \otimes \chi)^{I^{M^+}}$ given by $\Phi$ (combine Lemma \ref{prop:Hecke-Weyl} and the earlier observation), identifying Hecke algebras via $\mathrm{TW}$, the families $\mathcal{J}_w(\sigma \otimes \chi)$ and $\mathcal{J}^+_w(\sigma \otimes \chi)$ can be compared: we have
	\[ \mathcal{J}_w(\pi \otimes \chi) = |2|_F^{t(w)/2} \mathcal{J}^+_w(\sigma \otimes \chi). \]
\end{corollary}
\begin{proof}
	The following reduction steps are standard. Cf.\ \cite[pp.29--30]{Ar89-IOR1}.
	
	First, when $P = B^{\leftarrow}$, $P^+ = P_{\min}^+$ and $\pi$ arises from an unramified character of $T(F)$, this is simply Theorem \ref{prop:comparison-J}.
	
	Next, assume that $\pi$ is of the form $i^{\tilde{M}}_{\tilde{B}^{\leftarrow} \cap \tilde{M}}(\eta)$ for some unramified character $\eta$ of $T(F)$ (the exponent means induction within $\tilde{M}$). Then $\sigma = i^{M^+}_{P_{\min}^+ \cap M^+}(\eta)$. By parabolic induction in stages and its compatibility with standard intertwining operators, the assertion reduces to the previous case.
	
	In general, we first perform a routine reduction to the case of irreducible $\pi$. There exists $\eta$ as above such that $\pi$ embeds into $i^{\tilde{M}}_{\tilde{B}^{\leftarrow} \cap \tilde{M}}(\eta)$. Accordingly, $\sigma$ embeds into $i^{M^+}_{P_{\min}^+ \cap M^+}(\eta)$. By the functoriality of standard intertwining operators with respect to these embeddings, the equality reduces to the previous case.
\end{proof}

Now consider the $-$ case. Again, suppose that ${}^w M = M$ and $w$ has minimal length in its $\Omega^M_0$-coset. A basic observation is that $w \in \Omega'_0$ automatically. Indeed, $w$ cannot alter $\epsilon_1, \ldots, \epsilon_{n^\flat}$. Hence $w$ is also of minimal length in its $(\Omega^M_0)'$-coset. Therefore we have
\begin{align*}
	J'_w(\pi \otimes \chi) & \in \Hom_{\tilde{G}}\left( i_{\tilde{P}}(\pi \otimes \chi), i_{\tilde{P}}({}^{\mathring{w}} (\pi \otimes \chi) \right), \\
	J^-_w(\sigma \otimes \chi) & \in \Hom_{G^-}\left( i_{P^-}(\sigma \otimes \chi), i_{P^-}({}^{\ddot{w}} (\sigma \otimes \chi)) \right),
\end{align*}
as rational families in unramified characters $\chi: \prod_{i=1}^r \GL(n_i, F) \to \CC^{\times}$. The operators $J'_w$ are as in Definition \ref{def:int-op-odd}, whilst $J^-_w$ is the standard intertwining operator for $P^- \subset G^-$, using representatives $\ddot{w}$ and Haar measures normalized by $K^- \subset G^-(F)$.

\begin{lemma}
	The counterparts of Lemmas \ref{prop:Hecke-w}, \ref{prop:hatw-flat} and \ref{prop:Hecke-Weyl} hold in the $-$ case. It suffices to use the Hecke algebras $H_\psi^-$, $H^-$, the corresponding Hecke modules and the representatives $\mathring{w}$, $\ddot{w}$ above.
\end{lemma}
\begin{proof}
	The arguments are identical to the $+$ case. In more details, one should first break $J \cap M(F)$ into $\Sp(W^\flat)$ and $\GL$ factors using Lemma \ref{prop:cpt-intersection-Levi}. Next, the representative $\mathring{w}$ lies in the subgroup $\Mp(W')$ where $W' = \bigoplus_{i=2}^n Fe_i \oplus Ff_i$. Hence the effect of $\mathring{w}$ on $\tau_{1, \tilde{M}}^-$, Hecke algebras and Hecke modules can be analyzed in the same way as in the $+$ case. More precisely, $\mathring{w}$ does not affect the first $\otimes$-slot $\Schw(\mathfrak{o}/(2\varpi))$ of $\tau_{1, \tilde{M}}^-$.
\end{proof}

The following result can thus be proved in the same way as Corollary \ref{prop:comparison-J-general}.

\begin{corollary}\label{prop:comparison-J-general-odd}
	Let $\pi$ (resp.\ $\sigma$) be of finite length in $\mathcal{G}_\psi^{\tilde{M}, -}$ (resp.\ $\mathcal{G}^{M^+}$), and assume that there is a chosen Hecke-equivariant isomorphism
	\[ \Phi: \mathbf{M}_{\tau_{1, \tilde{M}}^-}(\pi) \rightiso \sigma^{I^{M^-}}, \]
	where we identify $H_\psi^{\tilde{M}, -}$ and $H^{M^-}$ via $\mathrm{TW}$. Let $w \in \Omega'_0$ be as above. Denote by $\mathcal{J}'_w(\pi \otimes \chi)$ (resp.\ $\mathcal{J}^-_w(\sigma \otimes \chi)$) the corresponding rational family of homomorphisms of $H_\psi^-$-modules (resp.\ $H^-$-modules).

	By using $\Phi$ and the Hecke-equivariant $\mathbf{M}_{\tau_{1, \tilde{M}}^-}({}^{\mathring{w}} (\pi \otimes \chi)) \to {}^{\ddot{w}} (\sigma \otimes \chi)^{I^{M^-}}$ given by $\Phi$, identifying Hecke algebras via $\mathrm{TW}$, the families $\mathcal{J}'_w(\sigma \otimes \chi)$ and $\mathcal{J}^-_w(\sigma \otimes \chi)$ can be compared: we have
	\[ \mathcal{J}'_w(\pi \otimes \chi) = |2|_F^{t(w)/2} \mathcal{J}^-_w(\sigma \otimes \chi). \]
\end{corollary}

\begin{remark}\label{rem:general-Levi-change}
	If one replaces $\mathcal{J}'_w$ by $\mathcal{J}_w$, i.e.\ by taking the representatives and Haar measures as in the $+$ case, the comparison reads
	\[ \mathcal{J}_w(\pi \otimes \chi) = (-q^{-1})^{t(w)} |2|_F^{t(w)/2} \mathcal{J}^-_w(\sigma \otimes \chi). \]
	
	Indeed, by unraveling the proof of Corollary \ref{prop:comparison-J-general}, it is enough to deal with the case $P = P^1$ and $\pi = \omega_\psi^{\flat, -} \boxtimes \eta$ where $\eta$ is an unramified character of $T^1(F)$. Theorem \ref{prop:comparison-J-odd-2} can now be applied.
\end{remark}

\subsection{Normalization}\label{sec:normalization}
By \cite[Corollary C]{BHK11}, the equivalence $\mathcal{G}_\psi^{\pm} \rightiso \mathcal{G}^{\pm}$ induced by $\mathrm{TW}: H^{\pm} \rightiso H_\psi^{\pm}$ preserves Plancherel measures up to an explicit positive constant.

The same result holds on Levi subgroups. Therefore, the formal degrees for discrete series of Levi subgroups are preserved. In turn, this implies that Harish-Chandra $\mu$-functions for parabolic inductions of discrete series are preserved. Again, all these are up to a positive constant depending on the dimension of various types (for $\tilde{G}$ or its Levi subgroups), and on the convention for Haar measures.

Recall the natural bijection
\[ \left\{\begin{array}{r|l}
	M & \text{Levi subgroup of}\; G \\
	& \text{standard relative to}\; B^{\leftarrow} \\
	& M \supset M^1 \;\text{in the $-$ case}
\end{array}\right\} \xleftrightarrow{1:1}
\left\{\begin{array}{r|l}
	M^{\pm} & \text{Levi subgroup of}\; G^{\pm} \\
	& \text{standard relative to}\; P_{\min}^{\pm}
\end{array}\right\}, \]
and if $M \leftrightarrow M^{\pm}$ in this manner, then there is a natural bijection $\mathcal{P}(M) \xleftrightarrow{1:1} \mathcal{P}(M^{\pm})$ where $\mathcal{P}(M)$ is the set of parabolic subgroups of $G$ with Levi $M$, and so on.

For the next result, we cite \cite[\S 3.1]{Li12b} for the notion of normalizing factors $r_{\tilde{Q}|\tilde{P}}(\pi)$ for $\tilde{G}$ or its Levi subgroups, where $P, Q \in \mathcal{P}(M)$ and $\pi$ in $\tilde{M}\dcate{Mod}$ is genuine and irreducible. It is a simple generalization of \cite[p.28]{Ar89-IOR1}, and can be compared with normalizing factors for $G^{\pm}$. We only consider normalizing factors whose inducing datum $\pi$ is taken from a given block.

\begin{proposition}\label{prop:matching-r}
	Let $M$ and $M^{\pm}$ be standard Levi subgroups relative to $B^{\leftarrow}$ and $P_{\min}^{\pm}$ respectively, $M \supset M^1$ in the $+$ case, such that $M \leftrightarrow M^{\pm}$. Suppose that a family of normalizing factors $r_{Q^{\pm}|P^{\pm}}(\sigma)$ for $G^{\pm}$ is given, where $P^{\pm}, Q^{\pm} \in \mathcal{P}(M^\pm)$ and $\sigma$ in $\mathcal{G}^{M^{\pm}}$ is irreducible.
	
	There are positive constants $c_{Q|P, M}$ such that
	\[ r_{\tilde{Q}|\tilde{P}}(\pi) := c_{Q|P, M} r_{Q^{\pm}|P^{\pm}}(\sigma), \quad P, Q \in \mathcal{P}(M) \]
	provide a family of normalizing factors for $\tilde{G}$, where $(M, P, Q) \leftrightarrow (M^{\pm}, P^{\pm}, Q^{\pm})$ and $\pi \in \mathcal{G}_\psi^{\tilde{M}, \pm}$ corresponds to $\sigma$. Ditto if $\tilde{G}$ is replaced by some Levi subgroup containing $\tilde{M}$.
\end{proposition}
\begin{proof}
	According to \cite{Ar89-IOR1} or \cite[\S 3]{Li12b}, the construction of normalizing factor reduces to the problem of ``taking square roots'' of Harish-Chandra $\mu$-function of a genuine discrete series $\pi$, when $\dim A_M/A_G = 1$. One can now conclude by applying the earlier discussion about matching $\mu$-functions up to positive constant.
\end{proof}

Once the normalizing factors are chosen as above, the normalized intertwining operators are defined as
\begin{align*}
	R_{\tilde{Q}|\tilde{P}}(\pi) & := r_{\tilde{Q}|\tilde{P}}(\pi)^{-1} J_{\tilde{Q}|\tilde{P}}(\pi), \\
	R_{Q^{\pm}|P^{\pm}}(\sigma) & := r_{Q^{\pm} | P^{\pm}}(\sigma)^{-1} J_{Q^{\pm}|P^{\pm}}(\sigma).
\end{align*}
They should actually be viewed as rational families of operators, by incorporating twists of $\pi$ (resp.\ $\sigma$) by unramified characters of $\prod_i \GL(n_i, F)$. Normalized intertwining operators are unitary when the inducing datum is unitarizable.

\subsection{An intertwining relation}\label{sec:int-relation}
Retain the conventions in \S\ref{sec:normalization}, and suppose that the normalizing factors for the blocks $\mathcal{G}^{\pm}$ and $\mathcal{G}_\psi^{\pm}$ are chosen compatibly as in Proposition \ref{prop:matching-r}.

\begin{definition}
	Let $\pi$ be irreducible in $\mathcal{G}_\psi^{\tilde{M}, +}$ (resp.\ $\mathcal{G}_\psi^{\tilde{M}, -}$), with $P \in \mathcal{P}(M)$ and $P \supset B^{\leftarrow}$ (resp.\ $P \supset P^1$). For $w \in \Omega_0$ (resp.\ $\Omega'_0$) such that
	\begin{itemize}
		\item ${}^w M = M$,
		\item $w$ has minimal length in its $\Omega^M_0$-coset,
	\end{itemize}
	we define
	\begin{equation}\label{eqn:R-operator}
		R_w(\pi) := r_{\tilde{P}^w|\tilde{P}}(\pi)^{-1} J_w(\pi), \quad \text{resp.}\; R'_w(\pi) := r_{\tilde{P}^w|\tilde{P}}(\pi)^{-1} J'_w(\pi),
	\end{equation}
	where $J_w(\pi)$ (resp.\ $J'_w(\pi)$) is as in Definition \ref{def:int-op-even} (resp.\ Definition \ref{def:int-op-odd}).
	
	On the side of $G^{\pm}$ we define the normalized avatar $R_w^\pm(\sigma)$ of $J_w^\pm(\sigma)$, for $\sigma \in \mathcal{G}^{M^{\pm}}$ that corresponds to $\pi$. These normalized operators are all viewed as rational families, and they are unitary when $\pi$ (equivalently, $\sigma$) is unitarizable.
\end{definition}

Let $\pi$, $\sigma$ and $w$ be as above. Consider the $+$ case, with the assumption that ${}^{\tilde{w}} \pi \simeq \pi$. By functoriality of $(\cdot)^{\wedge}$ (see \S\ref{sec:Aubert}) and the fact \cite[\S A.2]{KMSW} that $\tilde{w}$-conjugation commutes with $(\cdot)^\wedge$, every isomorphism $A(\tilde{w}): {}^{\tilde{w}} \pi \rightiso \pi$ induces
\[ {}^{\tilde{w}} \hat{\pi} \simeq \left( {}^{\tilde{w}} \pi\right)^{\wedge} \rightiso \hat{\pi}. \]
We denote the composition above by $\hat{A}(\tilde{w})$. The same applies to $\sigma$ and $\ddot{w}$-conjugation, and the resulting isomorphisms are denoted by $A(\ddot{w})$, $\hat{A}(\ddot{w})$.

The $-$ case is the similar, with different choices of representatives. We obtain the isomorphisms $A(\mathring{w})$ and $A(\ddot{w})$ in this way.

\begin{theorem}\label{prop:intertwining-relation}
	Let $\pi$, $\sigma$ and $w$ be as above. Assume that $\pi$ and $\hat{\pi}$ are both unitarizable (equivalently, $\sigma$ and $\hat{\sigma}$ are both unitarizable).
	
	In the $+$ case, assume furthermore that ${}^{\tilde{w}} \pi \simeq \pi$ (equivalently, ${}^{\ddot{w}} \sigma \simeq \sigma$). Choose isomorphisms $A(\tilde{w})$ and $A(\ddot{w})$ as before. There exists $c^+(w) \in \CC^{\times}$ such that
	\begin{align*}
		\Tr\left( \hat{A}(\tilde{w}) R_w(\hat{\pi}) i_{\tilde{P}}(\hat{\pi}, f) \right) & = c^+(w) \Tr\left( \left(A(\tilde{w}) R_w(\pi) \right)^{\wedge} \; i_{\tilde{P}}(\pi)^\wedge(f) \right), \\
		\Tr\left( \hat{A}(\ddot{w}) R^+_w(\hat{\sigma}) i_{P^+}(\hat{\sigma}, f^+) \right) & = c^+(w) \Tr\left( \left(A(\ddot{w}) R^+_w(\sigma)\right)^\wedge \; i_{P^+}(\sigma)^{\wedge} (f^+) \right)
	\end{align*}
	for all anti-genuine $f \in \mathcal{H}(\tilde{G})$ (resp.\ $f^+ \in \mathcal{H}(G^+)$); see \S\ref{sec:operations-Hecke}. Here $P \in \mathcal{P}(M)$ satisfies $P \supset B^{\leftarrow}$ and $P \leftrightarrow P^+ \in \mathcal{P}(M^+)$.
	
	Similarly, in the $-$ case, assume furthermore that ${}^{\mathring{w}} \pi \simeq \pi$ (equivalently, ${}^{\ddot{w}} \sigma \simeq \sigma$).Choose isomorphisms $A(\mathring{w})$ and $A(\ddot{w})$ as before. There exists $c^-(w) \in \CC^{\times}$ such that
	\begin{align*}
		\Tr\left( \hat{A}(\mathring{w}) R'_w(\hat{\pi}) i_{\tilde{P}}(\hat{\pi}, f) \right) & = c^-(w) \Tr\left( \left(A(\mathring{w}) R'_w(\pi)\right)^{\wedge} \; i_{\tilde{P}}(\pi)^\wedge(f) \right), \\
		\Tr\left( \hat{A}(\ddot{w}) R^-_w(\hat{\sigma}) i_{P^-}(\hat{\sigma}, f^-) \right) & = c^-(w) \Tr\left( \left(A(\ddot{w}) R^-_w(\sigma)\right)^{\wedge} \; i_{P^-}(\sigma)^{\wedge} (f^-) \right)
	\end{align*}
	with the same conventions.
	
	The constants $c^{\pm}(w)$ depend on $\pi$, but are independent of the choices of $A(\tilde{w})$, $A(\ddot{w})$ and $A(\mathring{w})$.
\end{theorem}
\begin{proof}
	Since $\mathrm{TW}$ is an isomorphism of Hilbert algebras and commutes with $(\cdot)^\wedge$ by Corollary \ref{prop:TW-vs-Aubert}, $\pi$ (resp.\ $\hat{\pi}$) is unitarizable if and only if $\sigma$ (resp.\ $\hat{\sigma}$) is; see \cite[Corollary C]{BHK11}. In particular, all the normalized intertwining operators in view are unitary.
	
	The equivalences concerning $w$-twists are immediate in view of Lemma \ref{prop:Hecke-w}.
	
	Also note that the isomorphism denoted by $A(\cdots)$ are all unique up to $\CC^{\times}$; rescaling them does not affect the trace identities.
	
	Let us show that there exists $c^{\pm}(w) \in \CC^{\times}$ satisfying the equalities involving $G^{\pm}$. We argue as in \cite[\S A.3.3]{KMSW} below. For the ease of notation, only the $+$ case is discussed.
	
	By \cite[(A.2.3)]{KMSW}, there is a natural isomorphism $i_{P^+}(\sigma)^\wedge \simeq i_{\overline{P^+}}(\hat{\sigma})$. We claim that the diagram
	\[\begin{tikzcd}[column sep=large]
		i_{P^+}(\sigma)^\wedge \arrow[r, "{( A(\ddot{w}) R^+_w(\sigma) )^\wedge }" inner sep=0.6em] \arrow[d, "\sim"' sloped] & i_{P^+}(\sigma)^\wedge \arrow[d, "\sim" sloped] \\
		i_{\overline{P^+}}(\hat{\sigma}) \arrow[r, "{ [\hat{A}(\ddot{w}) R^+_w(\hat{\sigma})]^-}"' inner sep=0.6em] & i_{\overline{P^+}}(\hat{\sigma})
	\end{tikzcd}\]
	commutes in $G^+\dcate{Mod}$ up to some $c^+(w) \in \CC^{\times}$. The notation $[\cdots]^-$ means that the operator acts on parabolic induction from $\overline{P^+}$, not $P^+$.
	
	Indeed, by \cite[Proposition A.2.1]{KMSW} the composed operators in both ways give rational families if unramified twists on $\sigma$ are incorporated. Hence they differ by a rational function by generic irreducibility. There is neither pole nor zero at the given $\sigma$, and this yields $c^+(w) \in \CC^{\times}$.
	
	The claim implies for all $f^+$, the operators
	\[ \left( A(\ddot{w}) R^+_w(\sigma) \right)^\wedge \; i_{P^+}(\sigma)^{\wedge}(f^+), \quad \left[ \hat{A}(\ddot{w}) R^+_w(\hat{\sigma}) \right]^- i_{\overline{P^+}}(\hat{\sigma}, f^+) \]
	have the same trace up to $c^+(w)$. However, the second one equals
	\[ \underbracket{R_{\overline{P^+}|P^+}(\hat{\sigma})}_{= R_{P^+|\overline{P^+}}(\hat{\sigma})^{-1}} \hat{A}(\ddot{w}) R^+_w(\hat{\sigma}) i_{P^+}(\hat{\sigma}, f^+) R_{P^+|\overline{P^+}}(\hat{\sigma}) \]
	by the properties $\mathrm{(R)}_1$ and $\mathrm{(R)}_2$ in \cite[p.28]{Ar89-IOR1} for normalized intertwining operators. The required trace identity follows at once.
	
	The $-$ case is the same and gives $c^-(w) \in \CC^{\times}$. The method actually applies to the equalities involving $\tilde{G}$; the point is to show that the same $c^{\pm}(w)$ works for both $\tilde{G}$ and $G^{\pm}$.
	
	Again, we only discuss the $+$ case. Fix a Hecke-equivariant isomorphism $\Phi: \mathbf{M}_{\tau_{0, \tilde{M}}}(\pi) \rightiso \sigma^{I^{M^+}}$ relative to $\mathrm{TW}$. Interpreting $(\cdot)^\wedge$ in terms of Hecke modules, $\Phi$ induces $\mathbf{M}_{\tau_{0, \tilde{M}}}(\hat{\pi}) \rightiso \hat{\sigma}^{I^{M^+}}$ by Corollary \ref{prop:TW-vs-Aubert}.
	
	In view of Lemma \ref{prop:Hecke-Weyl}, given $A(\ddot{w})$, one may choose $A(\tilde{w})$ so that the following diagram commutes:
	\[\begin{tikzcd}
		\mathbf{M}_{\tau_{0, \tilde{M}}}({}^{\tilde{w}} \pi) \arrow[d, "\Phi"'] \arrow[r, "{\mathbf{M}_{\tau_{0, \tilde{M}}}(A(\tilde{w})) }" inner sep=0.8em] & \mathbf{M}_{\tau_{0, \tilde{M}}}(\pi) \arrow[d, "\Phi"] \\
		\left( {}^{\ddot{w}} \sigma\right)^{I^{M^+}} \arrow[r, "{A(\ddot{w})}"'] & \sigma^{I^{M^+}}
	\end{tikzcd}\]
	
	Now, by combining Corollary \ref{prop:TW-vs-Aubert}, \ref{prop:comparison-J-general} and Proposition \ref{prop:matching-r}, the operators below are matched via $\mathcal{G}_\psi^+ \simeq \mathcal{G}^+$:
	\begin{align*}
		c_{P^w | P, M} A(\tilde{w}) R_w(\pi) & \quad \text{and} \quad |2|_F^{t(w)/2} A(\ddot{w}) R^+_w(\sigma), \\
		c_{P^w | P, M} \hat{A}(\tilde{w}) R_w(\hat{\pi}) & \quad \text{and} \quad |2|_F^{t(w)/2} \hat{A}(\ddot{w}) R^+_w(\hat{\sigma}).
	\end{align*}
	
	Applying Proposition \ref{prop:trace-type}, we infer from the known trace identity for $G^+$ that
	\begin{equation}\label{eqn:intertwining-relation}
		\Tr\left( \hat{A}(\tilde{w}) R_w(\hat{\pi}) i_{\tilde{P}}(\hat{\pi}, f) \right) = c^+(w) \Tr\left( \left(A(\tilde{w}) R_w(\pi) \right)^{\wedge} \; i_{\tilde{P}}(\pi)^\wedge(f) \right)
	\end{equation}
	holds for all
	\[ f = \Upsilon_{\tau_0}(h \otimes 1) \in e_{\tau_0} \star \mathcal{H}(\tilde{G}) \star e_{\tau_0}, \quad h \in H_\psi^+. \]
	
	To proceed, observe that $i_{\tilde{P}}(\pi)$ and $i_{\tilde{P}}(\hat{\pi})$ are semi-simple since $\pi$ and $\hat{\pi}$ are unitarizable, so \eqref{eqn:intertwining-relation} can be rewritten as
	\[ \sum_{i=1}^n c_i \Tr \pi_i(f) = 0. \]
	where $\sum_i c_i [\pi_i] \in \mathrm{Groth}(\tilde{G}) \otimes \CC$ and each $\pi_i$ is irreducible in $\mathcal{G}_\psi^+$.
	
	We contend that \eqref{eqn:intertwining-relation} extends to all $f \in e_{\tau_0} \star \mathcal{H}(\tilde{G}) \star e_{\tau_0}$. To see this, we use $\Upsilon_{\tau_0}$ and Proposition \ref{prop:Hecke-Morita} to transport the $\pi_i$ above to irreducible $H_\psi^+ \otimes \End_{\CC}(V_{\tau_0})$-modules $\Pi_i$; we know that $\sum_i c_i \Tr \Pi_i(h \otimes 1) = 0$ for all $h \in H_\psi^+$. By Morita equivalence, we may write $\Pi_i = \Pi'_i \boxtimes V_{\tau_0}$, and it follows that for all $h \otimes a$,
	\[ \sum_i c_i \Tr \Pi_i(h \otimes a) = \Tr(a) \sum_i c_i \Tr \Pi'_i(h) = 0. \]
	
	Next, we contend that  \eqref{eqn:intertwining-relation} extends from $e_{\tau_0} \star \mathcal{H}(\tilde{G}) \star e_{\tau_0}$ to $\mathcal{H}(\tilde{G})$. In fact, we shall prove
	\begin{equation}\label{eqn:intertwining-relation-aux}
		\sum_i c_i [\pi_i] = 0.
	\end{equation}
	
	Set $\pi_i^{\tau_0} := \pi_i(e_{\tau_0}) V_{\pi_i}$. We have seen that $\sum_i c_i \Tr \pi_i^{\tau_0}(f) = 0$ for all $f \in e_{\tau_0} \star \mathcal{H}(\tilde{G}) \star e_{\tau_0}$. In view of \cite[VIII, p.375, Proposition 6]{BouA8}, this implies $\sum_i c_i [\pi_i^{\tau_0}] = 0$ in the Grothendieck group of $e_{\tau_0} \star \mathcal{H}(\tilde{G}) \star e_{\tau_0}$-modules. On the other hand, the results in \cite[p.1119]{TW18} imply that $e_{\tau_0}$ is a special idempotent in the sense of \cite[Definition 3.11]{BK98}: the subcategory of $\tilde{G}\dcate{Mod}$ cut out by $e_{\tau_0}$ is exactly $\mathcal{G}_\psi^+$. The identity \eqref{eqn:intertwining-relation-aux} follows.

	The $-$ case is completely analogous, the special idempotent in question being $e_{\tau_1^-}$; see \cite[p.1120]{TW18}.
\end{proof}

When $\sigma$ is tempered, $\hat{\sigma}$ is known to be unitarizable and the factor $c^{\pm}(w)$ in Theorem \ref{prop:intertwining-relation} can be uniquely and precisely determined via endoscopy. This is part of Arthur's local intertwining relation: see \cite[(7.1.14)]{Ar13} for the case of $G^+$; as for $G^-$, see \cite[\S 4]{Is23} for partial results.

\begin{remark}\label{rem:intertwining-relation-odd}
	The statements hold true if we use $R_w(\cdots)$ instead of $R'_w(\cdots)$ in the $-$ case. Indeed, by Remark \ref{rem:general-Levi-change}, this just causes both sides of the equality
	\[ \Tr\left( \hat{A}(\mathring{w}) R'_w(\hat{\pi}) i_{\tilde{P}}(\hat{\pi}, f) \right) = c^-(w) \Tr\left( \left(A(\mathring{w}) R'_w(\pi)\right)^{\wedge} \; i_{\tilde{P}}(\pi)^\wedge(f) \right) \]
	to be multiplied by $(-q^{-1})^{t(w)}$.
\end{remark}


\printbibliography[heading=bibintoc]

\vspace{1em}
\begin{flushleft} \small
	F. Chen: School of Mathematical Sciences, Peking University. No.\ 5 Yiheyuan Road, Beijing 100871, People's Republic of China. \\
	E-mail address: \href{mailto:cf_97@pku.edu.cn}{\texttt{cf\_97@pku.edu.cn}}
\end{flushleft}

\begin{flushleft} \small
	W.-W. Li: Beijing International Center for Mathematical Research / School of Mathematical Sciences, Peking University. No.\ 5 Yiheyuan Road, Beijing 100871, People's Republic of China. \\
	E-mail address: \href{mailto:wwli@bicmr.pku.edu.cn}{\texttt{wwli@bicmr.pku.edu.cn}}
\end{flushleft}

\end{document}